%% file: mise.tex
\documentclass[12pt,titlepage]{article}

\usepackage[default]{jasa_harvard}    
\usepackage{JASA_manu}

\RequirePackage[OT1]{fontenc}
\RequirePackage{amsthm,amsmath}
\RequirePackage[colorlinks,citecolor=blue,urlcolor=blue]{hyperref}
\RequirePackage{graphicx,subfigure,latexsym,amssymb}
\RequirePackage{float,epsfig,multirow,rotating,times}
\RequirePackage{upgreek,wrapfig}
\RequirePackage{comment}

\newcommand {\ctn}{\citeasnoun} 
\usepackage{graphicx,subfigure,amsmath,latexsym,amssymb}
\usepackage{float,epsfig,multirow,rotating,times}
\usepackage{upgreek,wrapfig}

\newtheorem{theorem}{Theorem}[section]
\newtheorem{lemma}[theorem]{Lemma}

\newcommand{\abs}[1]{|#1|}

\newcommand{\bTheta}{\boldsymbol{\Theta}}

\newcommand{\bY}{\boldsymbol{Y}}

\begin{document}

\normalsize

\title{\vspace{-0.8in}
{\bf Bayesian MISE convergence rates of Polya urn based density estimators: 
asymptotic comparisons and choice of prior parameters}}
\author{Sabyasachi Mukhopadhyay and Sourabh Bhattacharya\thanks{
Sabyasachi Mukhopadhyay is a postdoctoral research associate at Biostatistics unit, University of Hohenheim, 
and Sourabh Bhattacharya 
is an Associate Professor in
Interdisciplinary Statistical Research Unit, Indian Statistical
Institute, 203, B. T. Road, Kolkata 700108.
Corresponding e-mail: sourabh@isical.ac.in.}}
\date{\vspace{-0.5in}}
\maketitle

\begin{abstract}

Mixture models are well-known for their versatility, and the Bayesian paradigm is a suitable 
platform for mixture analysis, particularly when the number of components is unknown. 
\ctn{Bhattacharya08} introduced a mixture model based on the Dirichlet process, where 
an upper bound on the unknown number of components is to be specified. Here we consider a Bayesian asymptotic
framework for objectively specifying the upper bound, which we assume to depend on the sample size. 
In particular, we define a Bayesian analogue of the mean integrated squared error (Bayesian $MISE$), and
select that form of the upper bound, and also that form of the precision parameter of the 
underlying Dirichlet process, 
for which Bayesian $MISE$ 
of a specific density estimator, which is a 
suitable modification of the Polya-urn based prior predictive model, converges at sufficiently fast rate. 
As a byproduct of our approach, we investigate asymptotic choice of the precision parameter of the 
traditional Dirichlet process mixture model; the density estimator we consider here is a modification of the 
prior predictive distribution of \ctn{Escobar95} associated with the Polya urn model.
Various asymptotic issues related to the two aforementioned mixtures, including comparative
performances, are also investigated.
We also perform simulation experiments for comparing the performances of the approaches associated with \ctn{Bhattacharya08} 
and \ctn{Escobar95} in terms of Bayesian $MISE$ for various choices of the true, data-generating distribution,
and demonstrate that the approaches related to \ctn{Bhattacharya08} generally outperform those associated with \ctn{Escobar95}.
%
\\[2mm]

\end{abstract}

\tableofcontents

\newpage

\section{{\bf Introduction}}
\label{sec:intro}

In recent years, the use of nonparametric prior in the context of Bayesian density estimation arising out of 
mixtures has received wide attention thanks to their flexibility and advances in computational methods. 
The study of nonparametric priors in the context of Bayesian density estimators has been initiated by 
\ctn{Ferguson83} and \ctn{Lo84} who derived the associated posterior 
and predictive distributions. 

The set-up of for nonparametric Bayesian density estimation with mixture priors can be represented in the following hierarchical form:
for $i=1,\ldots,n$, $Y_i\sim K(\cdot\mid\theta_i)$ independently;
$\theta_1, \ldots, \theta_n\stackrel{iid}{\sim} F$ and $F\sim \Upsilon$, where $F$ is a random probability measure
and $\Upsilon$ is some appropriate nonparametric prior distribution on the set of probability measures.
An important choice of $\Upsilon$ is of course the Dirichlet process prior, which we denote by $DP(\alpha G_0)$,
$G_0$ being the expected probability measure and $\alpha$ being the precision parameter. 

\subsection{{\bf Two competing models based on Dirichlet process}}
\label{subsec:ew_and_sb}

\subsubsection{{\bf The EW model}}
\label{subsubsec:ew}

With the Dirichlet process prior
the set-up of 
\ctn{Lo84} boils down to the \ctn{Escobar95} (henceforth EW) model. For our purpose in this paper, 
our interest as a density estimator 
is the following modification of the prior predictive associated with EW:
\begin{equation}
\hat f_{EW}(y\mid \theta_1,\ldots,\theta_n) = \frac{\alpha}{\alpha+n}A(y)+\frac{1}{\alpha+n}\sum_{i=1}^{n}K^\star(y\vert\theta_i),
\label{eq:ew_prior_pred_0} 
\end{equation}
where $A(y)=\int K^\star(y\vert\theta)dG_0(\theta)$; $K^\star(y\vert\theta)$ being a modification
of the original $K(y\vert\theta)$. 
That is, we are interested in the posterior distribution of the statistic $\hat f_{EW}(y\mid \theta_1,\ldots,\theta_n)$
given data modeled by the original EW model having kernel $K(\cdot\vert\theta)$. We shall consider priors for $\sigma$ that are dependent upon the
sample size $n$, such that $\sigma\rightarrow 0$ in probability with respect to the prior. Thus, $\sigma+\hat k_n\rightarrow k~(>0)$ in probability.

Specifically, we shall consider the situation where the model  
$K(y\vert\theta)$ is the density of $N(y:\theta,\sigma^2)$, the normal density with mean $\theta$, variance
$\sigma^2$, and evaluated at $y$. The kernel associated with the density estimator, 
$K^\star(y\vert\theta)$, 
is the density of the truncated normal density $N(y:\theta,(\sigma+\hat k_n)^2)\mathbb I_{\{|y|\leq a\}}$,
where $a>0$, and for any set $S$, $\mathbb I_S$ is the indicator function of the set $S$. In the above, $\hat k_n$ is a strongly
consistent estimator, based on $n$ data points, of the scale $k~(>0)$ associated with the true data-generating density
(see Section \ref{subsec:true_distribution}).
We assume that there exists $\eta>0$ such that $\hat k_n>\eta$ for all $n\geq 1$, for almost all sequences $\hat k_n$.
We further assume that $(\hat k_n-k)^2$ is uniformly integrable with respect to the true, data-genrating distribution.
The last assumption ensures that $(\hat k_n-k)^2$ converges to zero even in expectation with respect to the true
distribution. An example of such a consistent estimator is provided in Section \ref{subsubsec:hat_k_n}.


In (\ref{eq:ew_prior_pred_0}), the random measure $F$ has been integrated out
to arrive at the following Polya urn distribution of $\theta_1,\ldots,\theta_n$:
\begin{eqnarray}
\theta_1 &\sim & G_0\nonumber\\
\theta_i\vert\theta_1,\ldots,\theta_{i-1} &\sim & \frac{\alpha}{\alpha+i-1}G_0+\frac{1}{\alpha+i-1}\sum_{j=1}^{i-1}\delta_{\theta_j},
\ \ \mbox{for} \ \ i=2,\ldots,n,
\nonumber
\end{eqnarray}
where $\delta_{\theta_j}$ denotes point mass at $\theta_j$.

In our case, we shall assume compact support of the base measure $G_0$. It follows that $K^*$ is a compactly supported Gaussian kernel. Note that
in the frequentist literature compactly supported kernels are often used for density estimation, particularly for deriving theoretical results.
See, for example, \ctn{Moreira12}, \ctn{Liang13} (see also \ctn{Huntsman17} and the references therein), 
for some relatively recent works in this regard. Since in this paper we deal with density estimation, 
considering compact support of $K^*$ is not that retrogressive.

\subsubsection{{\bf The SB model}}
\label{subsubsec:sb}

Though very well known, the EW model has several draw backs in terms of computational efficiency 
which manifest themselves particularly when applied to massive data. 
\ctn{Bhattacharya08} (henceforth SB) proposed a new model which is shown to bypass the problems 
of the EW model (see \ctn{Sabya10a}, \ctn{Sabya10}, \ctn{Sabya13} for the details). 
The essence of the SB model lies in the assumption that data points are independently and identically distributed
as an $M$-component mixture model, where the parameters of the mixture components, which we
denote by $\theta_1,\ldots,\theta_M$, are samples from a Dirichlet process. 
In other words, the model of SB is given by the following hierarchical structure:
\begin{eqnarray}
y_1,\ldots,y_n &\stackrel{iid}{\sim}&\frac{1}{M}\sum_{i=1}^MK(\cdot\vert\theta_i)\label{eq:hier_sb1}\\
\theta_1,\ldots,\theta_M &\stackrel{iid}{\sim}& F\label{eq:hier_sb2}\\
F&\sim & DP(\alpha G_0).\label{eq:hier_sb3}\nonumber
\end{eqnarray}
The density estimator that we are interested in is of similar form as (\ref{eq:hier_sb1}) but
$K(\cdot\vert\theta_i)$ is modified to $K^\star(\cdot\vert\theta_i)$. 
In other words, the
density estimator corresponding to the SB model which we shall work on is the following:
\begin{equation}
\hat f_{SB}(y\vert\theta_1,\ldots,\theta_M)=\frac{1}{M}\sum_{i=1}^MK^\star(y\vert\theta_i).
\label{eq:sb1}
\end{equation}
In other words, we are interested in the posterior distribution of $\hat f_{SB}(y\vert\theta_1,\ldots,\theta_M)$
given data modeled by the original SB model having kernel $K(\cdot\vert\theta)$.
As in the case of EW, for the SB model also we set 
$K(y\vert\theta)\equiv N(y:\theta,\sigma^2)$ and
$K^\star(y\vert\theta)\equiv N(y:\theta,(\sigma+\hat k_n)^2)\mathbb I_{\{|y|\leq a\}}$.

Marginalizing out $F$ results in the Polya urn distribution of $\theta_1,\ldots,\theta_M$.
Thus, the total number of
distinct components of the SB mixture, although random, is bounded above by $M$, while
in the EW mixture (\ref{eq:ew_prior_pred_0}) the corresponding upper bound is $n$. 
If $M$ is chosen to be much 
less than $n$,
then this idea entails great computational efficiency compared to the EW model,
particularly in the case of massive data.
Moreover, if $M=n$, and $Y_i$ is associated with $\theta_i$ for every $i$, then the SB model reduces to
the EW model, showing that the EW model is a special case of the SB model (see \ctn{Sabya10a}, for example).

\subsubsection{{\bf Discussion of the density estimators $\hat f_{EW}(y|\theta_1,\ldots,\theta_n)$ and $\hat f_{SB}(y|\theta_1,\ldots,\theta_M)$}}
\label{subsubsec:density_discussion}
The issue of modifying the original kernel $K(y|\theta)$ to $K^*(y|\theta)$ for both EW and SB asymptotics needs some discussion.
First note that in the literature asymptotics of traditional DP mixtures (also, the EW model) concerns convergence of the posterior distribution of the random probability measure
as the data size increases; convergence rates of the corresponding posterior predictive density are a byproduct of the
posterior contraction rate; see Corollary 5.1 of \ctn{Ghosal01}. 
It is crucial to note that here the data are modeled as: $y_1,\ldots,y_n\stackrel{iid}{\sim}\int K(\cdot|\theta)dF(\theta)$, where $F$ follows the Dirichlet process. 
This $iid$ set-up is very convenient for asymptotic calculations associated with the posterior of $F$.

Now note under the SB model, given $F$, 
for any value $M$,
\begin{eqnarray}
f_{SB,F}(y)&=& \frac{1}{M}\sum_{i=1}^M \int K(y\vert\theta_i)\prod_{j=1}^MdF(\theta_j)\nonumber\\
&=&\frac{1}{M}\sum_{i=1}^M\int K(y\vert\theta_i)dF(\theta_i)\nonumber\\
&=&\int K(y\vert\theta)dF(\theta),\notag
\end{eqnarray}
so that the marginal distribution of any data point given $F$ is the same as that of EW. However, given $F$, $y_1,\ldots,y_n$ are
not independent. Indeed, their joint distribution conditional on $F$ is 
\begin{equation*}
[y_1,\ldots,y_n|F]=\frac{1}{M^n}\int \left\{\prod_{i=1}^n\left[\sum_{j=1}^MK(y_i|\theta_j)\right]\right\}\prod_{j=1}^MdF(\theta_j).
\end{equation*}
This dependent joint distribution is not as convenient for asymptotic calculations as in the $iid$ EW case.
Thus the traditional approach to DP mixture asymptotics and then derivation of the corresponding posterior predictive convergence rate as by-product, seems to be
unwieldy in the SB model scenario. 
The alternative approach described in Section \ref{subsubsec:sb} 
facilitates asymptotic posterior calculation such that the density estimator $\hat f_{SB}(y|\theta_1,\ldots,\theta_M)$ has fast convergence rate with respect to
Bayesian $MISE$ (introduced in Section \ref{sec:Bayesian_MISE}), to the true, 
data-generating distribution whenever the assumptions 
of the true density detailed in Section \ref{subsec:true_distribution} hold, that is, essentially when the true distribution is a normal mixture with respect to the mean. 
The simulation studies in Section \ref{sec:simulation_study} are not only in accordance with our theoretical results,
but they also demonstrate that when the true density is essentially a normal mixture of the mean, 
the SB-based density estimator $\hat f_{SB}(y|\theta_1,\ldots,\theta_M)$ 
significantly outperforms the
original and unrestricted density estimators proposed in EW and SB. 
Interestingly, these latter density estimators do not need any restrictive assumptions on bandwidth; in fact, in Section \ref{sec:simulation_study} we assume that
the kernel variances are all different and that $(\theta_i,\sigma_i)$ are jointly samples from the underlying Dirichlet process, and hence there is no need to introduce
the consistent bandwidth estimator $\hat k_n$. On the other hand, for the SB-based
density estimator $\hat f_{SB}(y|\theta_1,\ldots,\theta_M)$, we assume a single $\sigma$ with a prior depending on the sample size such that $\sigma$ tends to zero in probability
as the sample size goes to infinity. That in spite of such restriction associated with $\hat f_{SB}(y|\theta_1,\ldots,\theta_M)$ as compared to the original density estimators proposed in EW and SB,
the performance of the former is still much superior, shows the worth of introducing $\hat f_{SB}(y|\theta_1,\ldots,\theta_M)$ when the true density has the form described
above. 

We introduce the EW-based modified density estimator $\hat f_{EW}(y|\theta_1,\ldots,\theta_n)$ and derive its asymptotic theory mainly for comparability of our approach to SB-based asymptotics of 
$\hat f_{SB}(y|\theta_1,\ldots,\theta_M)$. Indeed, we use the same methods of asymptotics calculations for both the density estimators. Eventually we see that 
$\hat f_{SB}(y|\theta_1,\ldots,\theta_M)$ significantly outperforms the EW-based density estimator
$\hat f_{EW}(y|\theta_1,\ldots,\theta_n)$, both theoretically as well as in simulation studies, when the assumptions of the true density detailed in Section \ref{subsec:true_distribution} hold.
However, $\hat f_{EW}(y|\theta_1,\ldots,\theta_n)$ does outperform both the original density estimators of EW and SB, demonstrating its utility when the true model
has the above form. 

From the above arguments it is evident that at least when the true data generating distribution is of the form detailed in Section \ref{subsec:true_distribution}, the density estimator
$\hat f_{SB}(y|\theta_1,\ldots,\theta_M)$ is to be preferred over the other Bayesian density estimators from both theoretical and practical perspectives. 
In our future efforts, we shall generalize the class of true distributions and derive more general asymptotic results, even for the more complex and dependent SB set-up.

\subsection{{\bf Alternative truncated density estimators}}
\label{subsec:alternative_truncation}


In (\ref{eq:sb1}) and (\ref{eq:ew_prior_pred_0}), we assumed that each kernel of the mixture density
is a truncated normal. Alternatively, one may consider the following density estimators: 
\begin{equation}
\tilde f_{EW}(y\mid \theta_1,\ldots,\theta_n) =\varphi_1(\Theta_n) \left[\frac{\alpha}{\alpha+n}A(y)
+\frac{1}{\alpha+n}\sum_{i=1}^{n}K^\star(y\vert\theta_i)\right]\mathbb I_{\left\{|y|\leq a\right\}},
\label{eq:ew_prior_pred_1} 
\end{equation}
where  
$K(y\vert\theta)\equiv N(y:\theta,\sigma^2)$ and
$K^\star(y\vert\theta)\equiv N(y:\theta,(\sigma+\hat k_n)^2)$, and
\begin{align*}
\varphi_1(\Theta_n)&=\left\{\frac{\alpha}{\alpha+n}\int\left[\Phi\left(\frac{a-\theta}{\sigma+\hat k_n}\right)
-\Phi\left(\frac{-a-\theta}{\sigma+\hat k_n}\right)\right]dG_0(\theta)\right.\notag\\
&\qquad\qquad\left.+\frac{1}{\alpha+n}\sum_{i=1}^{n}\left[\Phi\left(\frac{a-\theta_i}{\sigma+\hat k_n}\right)
 -\Phi\left(\frac{-a-\theta_i}{\sigma+\hat k_n}\right)\right]\right\}^{-1}, 
\end{align*}
so that truncation of each mixture kernel is not required. Similarly, the alternative SB density estimator 
will have the following form with $K(y\vert\theta)\equiv N(y:\theta,\sigma^2)$ and 
$K^\star(y\vert\theta)\equiv N(y:\theta,(\sigma+\hat k_n)^2)$:
\begin{equation}
\tilde f_{SB}(y\vert\theta_1,\ldots,\theta_M)=\varphi_2(\Theta_{M})\left[\frac{1}{M}\sum_{i=1}^MK^\star(y\vert\theta_i)\right]
\mathbb I_{\{|y|\leq a\}},
\label{eq:sb2}
\end{equation}
where, 
\begin{equation*}
\varphi_2(\Theta_{M})=
\left\{\frac{1}{M}\sum_{i=1}^M
\left[\Phi\left(\frac{a-\theta_i}{\sigma+\hat k_n}\right)
-\Phi\left(\frac{-a-\theta_i}{\sigma+\hat k_n}\right)\right]\right\}^{-1}.
\end{equation*}
The true distribution, detailed in Section \ref{subsec:true_distribution}, can be modified analogously.

The density estimators (\ref{eq:ew_prior_pred_1}) and (\ref{eq:sb2}) are ratio estimators,
and the delta-method may be invoked for handling the asymptotic theory of such estimators. 
Indeed, we have verified that all the asymptotic results with these ratio estimators remain the same as
those associated with (\ref{eq:ew_prior_pred_0}) and (\ref{eq:sb1}) and require exactly the same set of assumptions, 
only except the result that both (\ref{eq:ew_prior_pred_1}) and (\ref{eq:sb2})
converge to the same true distribution. Although we expect the result to hold, the proof
that (\ref{eq:ew_prior_pred_0}) and (\ref{eq:sb1}) converge to the same true distribution, presented in this paper, 
can not be extended in the case of (\ref{eq:ew_prior_pred_1}) and (\ref{eq:sb2}).
In any case, we do not pursue (\ref{eq:ew_prior_pred_1}) and (\ref{eq:sb2}) any further and henceforth,
concentrate only on (\ref{eq:ew_prior_pred_0}) and (\ref{eq:sb1}).

\subsection{{\bf Importance of Bayesian version of mean integrated squared error}}
\label{subsec:importance_Bayesian_MISE}
Mean integrated squared error ($MISE$) is classically a very well-established measure for evaluating 
classical density estimators; see, for example,
\ctn{Silverman86}. Attractively, it is additive in integrated squared bias and integrated variance, so that
the desired density estimator can be adjusted to account for this trade-off. Measures based on other distances,
such as the Hellinger distance, does not enjoy such property. In the context of Bayesian density estimation
with respect to (\ref{eq:ew_prior_pred_0}) and (\ref{eq:sb1}), note that too many mixture components are expected
to reduce the bias, but can inflate the variance significantly. 
Since $\alpha$ controls the number of mixture components of EW and both $\alpha$ and $M$ control the number of mixture
components of SB, it is clear that they must be chosen by appropriately accounting for the Bayesian bias-variance
trade-off. In other words, the bias-variance trade-off is
very important for Bayesian density estimation, and some appropriate Bayesian version of classical $MISE$
is necessary to quantify such trade-off.
In this regard, we introduce our Bayesian $MISE$ measure in Section \ref{sec:Bayesian_MISE}. 
It is important to note that for our Bayesian density estimators, $\hat k_n$ essentially plays the role
of the bandwidth in classical kernel density estimators, and since it is a strongly consistent estimator of
the scale associated with the true distribution, there is essentially no bandwidth selection problem associated with our
Bayesian $MISE$.

\section{{\bf Overview of our contributions}}
\label{sec:overview_contributions}

Assuming the Polya-urn based mixture set-up 
in this paper we investigate choices of $M$ and $\alpha$ by obtaining the Bayesian $MISE$ convergence rate of 
our SB-based density estimator given by (\ref{eq:sb1}). 
We will assume $M$ to be increasing with $n$; in fact, our subsequent asymptotic calculations
show that $M$ increasing at a rate slower that $\sqrt{n}$, is adequate. 
Since the interplay between $M$ and $\alpha$ is important, we also assume $\alpha$ to be increasing with $n$.
But if $\alpha$ increases too fast
then convergence to the true distribution need not attain; we will investigate choices of $\alpha$
that lead to convergence and non-convergence to the correct model.
To reflect the dependence of $M$ and $\alpha$ on $n$ henceforth we shall write $M_{n}$ and $\alpha_n$. 
%
We show that the prior parameters driving the model can be selected in a way that
the Bayesian $MISE$ of the respective model convergences to zero at a desirable rate.
Thus, we obtain objective, asymptotic choices of the prior parameters. 
This is important since in applications the prior 
parameters are almost always chosen by {\it ad hoc} means. 

In parallel with the development related to the SB-based density estimator, we develop the 
corresponding Bayesian $MISE$-based asymptotic theory for the EW-based density
estimator (\ref{eq:ew_prior_pred_0}), where we discuss choices of the prior parameters associated with the
EW model. In fact, while we proceed, we shall always state the results related to the EW
model first, and then the corresponding result on the SB model, since the former is a simpler model compared to SB, 
and so, the results/calculations are simpler and make the
SB-based calculation steps easier to follow.

We show that both our density estimators corresponding to EW and SB converge to the same true distribution, 
and that for the same choices of the prior parameters common to both the EW and the SB models,
the SB model converges much faster to the true distribution with respect to Bayesian $MISE$.

We back up our theoretical results with simulation experiments where we also include, in addition to the
density estimators (\ref{eq:sb1}) and (\ref{eq:ew_prior_pred_0}), the original density estimators
proposed in \ctn{Bhattacharya08} and \ctn{Escobar95}, which allow the scales of the mixture components
to be different and random, and consider them along with the mean parameters as samples from a bivariate
Dirichlet process. We demonstrate that the methods based on SB generally outperform those associated with EW
in terms of Bayesian $MISE$.


There is also an important question regarding the conditions leading to convergence of the mixtures
to the wrong models (that is, models that did not generate the data). In other words, this is a question
of model mis-specification. We show that the model of EW
can converge to a wrong model under relatively weak conditions, whereas much stronger conditions must
be enforced to get the SB model to converge to the wrong model. 

Furthermore, we consider a modified
version of SB's model that accommodates continuous mixing probabilities; 
however, as we demonstrate, all the results remain intact under this modified version.

Proofs of all the results 
are provided in
the supplement, whose sections have the prefix
``S-'' when referred to in this paper.
Additionally, in Section S-6 of the supplement, we 
we investigate the ``large $p$, small $n$" problem of both the EW and the SB set-up.
%

For all our $MISE$-based comparisons we assume that the kernel-based parameters (usually, location
and scale parameters) and the random measure $F$ have the same 
prior distributions under both EW and SB. 

The rest of the work is organized as follows. 
We introduce our notion of Bayesian $MISE$ in Section \ref{sec:Bayesian_MISE}.
In Section \ref{sec:assumptions} 
we provide details of the explicit forms of the 
EW-based and the SB-based
models and provide discussions on the assumptions used in our subsequent asymptotic calculations.
The assumptions regarding the true, data-generating distribution are also provided in the same section.
Section \ref{sec:posterior_mean_convergence}
provides results
showing convergence of the posterior expectations of the EW-based and the SB-based models, 
respectively, to the same true distribution, also providing the rates of convergence.
In Section \ref{sec:mise_bounds} we compute Bayesian $MISE$-based 
rates of convergence of the EW and the SB models. In Section \ref{sec:comparison_sb_ew}  
the $MISE$ rates of the two models are compared with each other 
while 
also demonstrating how asymptotic choices of the prior parameters can be made.
Using simulation experiments we compare the Bayesian $MISE$ based performances of the SB and EW based density estimators
in Section \ref{sec:simulation_study}, for various choices of the true distribution, demonstrating that
the SB based density estimators outperform those based on EW in most of the cases considered.
In Section \ref{sec:param_model}, the conditions, under which the models may converge to wrong distributions, 
are investigated.
Asymptotics of a modified version of the SB model are discussed in Section \ref{sec:modified_sb_model}.

\section{{\bf Bayesian MISE}}
\label{sec:Bayesian_MISE}

Assuming that $\hat f_n$ is an estimate of the true density $f_0$ based on the observed data
$\bY_n=(Y_1,\ldots,Y_n)'$,
the MISE of $\hat f_n$ is given by
\begin{equation}
MISE=\int E\{\hat f_n(y)-f_0(y)\}^2dy,
\label{eq:mise_classical}
\end{equation}
where the expectation is with respect to the data $\bY_n$. In our Bayesian context, 
we consider the following analogue of the classical definition:
\begin{equation} 
MISE^*_1=\int\{\hat f(y\vert\bY_n)-f_0(y)\}^2dy,
\label{eq:mise_star_1}
\end{equation}
where $\hat f(y\vert\bY)=\int \hat f(y\vert\Theta)\pi(\Theta\vert\bY_n)d\Theta$ denotes any choice of the
posterior predictive density estimator. 
Note that the choice of the posterior predictive density estimator is determined by the choice of $\hat f(y\vert\Theta)$.

We further modify the above definition by considering a weighted version,
given by
\begin{equation} 
MISE^*_2=\int\{\hat f(y\vert\bY_n)-f_0(y)\}^2f_0(y)dy.
\label{eq:mise_star_2}
\end{equation}
Thus, in (\ref{eq:mise_star_2}) $f_0(y)$ downweights those squared error terms $\{\hat f(y\vert\Theta)-f_0(y)\}^2$
which correspond to extreme values of $y$. 
Such weighting strategies that use the true distribution as weight, are not uncommon in the
statistical literature. The well-known Cram\'{e}r-von Mises test statistic (see, for example,
\ctn{Serfling80}) is a case in point.

It is easy to see that
\[MISE^*_1\leq MISE_1\] and
\[MISE^*_2\leq MISE_2,\]
where $MISE_1$ and $MISE_2$ are given by
\begin{eqnarray}
MISE_1&=&\int E\{\hat f(y\vert\Theta)-f_0(y)\}^2dy\nonumber\\
&=&\int \int\{\hat f(y\vert\Theta)-f_0(y)\}^2\pi(\Theta\mid\bY_n)d\Theta dy,
\label{eq:mise_bayesian_1}
\end{eqnarray}
and
\begin{eqnarray}
MISE_2&=&\int E\{\hat f(y\vert\Theta)-f_0(y)\}^2f_0(y)dy\nonumber\\
&=&\int\int\{\hat f(y\vert\Theta)-f_0(y)\}^2\pi(\Theta\mid\bY_n)f_0(y)d\Theta dy,
\label{eq:mise_bayesian_2}
\end{eqnarray}
Because of the inherent advantages of the weighted version in the case of extreme values, in this paper
we focus on $MISE^*_2$, which is dominated by $MISE_2$. More specifically, for both EW and SB models, 
we shall obtain rates of convergence of the respective $MISE_2$ to zero.
Note that even though $MISE^*_2$ is a measure regarding how close the posterior predictive density
$\hat f(\cdot|\bY_n)$ is close to the true density $f_0(\cdot)$, $MISE_2$ no longer considers the
distance between the point estimate $\hat f(\cdot|\bY_n)$ and $f_0(\cdot)$ directly; instead, it considers
the distance between the random density estimator $\hat f(\cdot|\Theta)$ and $f_0(\cdot)$, suitably weighted
by the posterior and the true density. Hence, $MISE_2$ seems to be a more ``Bayesian measure" compared to $MISE^*_2$. 
Further justification of dealing with $MISE_2$ is provided by the following argument. Note that by Markov's inequality, for any $\epsilon>0$,
\begin{align}
&P\left(\int\left\{\hat f(y|\Theta)-f_0(y)\right\}^2f_0(y)dy>\epsilon\Bigg\vert\bY_n\right) \notag\\
&\qquad\qquad<\epsilon^{-1}E\left[\int\left\{\hat f(y|\Theta)-f_0(y)\right\}^2f_0(y)dy\right]\notag\\
&\qquad\qquad =\epsilon^{-1}\int\int\{\hat f(y\vert\Theta)-f_0(y)\}^2\pi(\Theta\mid\bY_n)f_0(y)d\Theta dy\notag\\
&\qquad\qquad =\epsilon^{-1}MISE_2.\notag
\end{align}
In words, $MISE_2$ also bounds the posterior probability of the weighted, $\Theta$-specific random $MISE$ given by $\int\left\{\hat f(y|\Theta)-f_0(y)\right\}^2f_0(y)dy$, to exceed $\epsilon$.
Hence, it is important to have $MISE_2$ to converge to zero at a fast enough rate.

Since the bounds that we provide for $MISE_2$ automatically bound $MISE^*_2$, it is interesting to observe that the bounds 
for the $MISE$ associated with random density estimators are also bounds for the $MISE$ associated with 
the posterior predictive density estimators of the form $\hat f(\cdot|\bY_n)$. Thus, although we are interested
in the random density estimators and the corresponding Bayesian measure $MISE_2$, our techniques automatically provide
inference regarding the point density estimator $\hat f(\cdot|\bY_n)$.
Henceforth, for notational simplicity we refer to $MISE_2$ simply as $MISE$.

$MISE$ of the form (\ref{eq:mise_bayesian_2}) can be expressed conveniently as
\begin{eqnarray}
MISE &=& \int Var\left(\hat f\left(y\mid \Theta\right)\vert\bY_n\right)f_0(y)dy\nonumber\\ 
&&\quad\quad+ \int \left\{Bias(\hat f(y\mid \Theta)\vert\bY_n)\right\}^2f_0(y)dy,
\label{eq:mise_bayesian_3}
\end{eqnarray}
where
$Var\left(\hat f\left(y\mid \Theta\right)\vert\bY_n\right)$ denotes the variance of 
$\hat f\left(y\mid \Theta\right)$ with respect to the posterior $[\Theta\mid\bY_n]$ and 
\begin{eqnarray}
Bias(\hat f(y\mid \Theta)\vert\bY_n)&=&\left |E\left(\hat f(y\mid \Theta)\bigg\vert \bY_n\right)-f_0(y)\right |,
\label{eq:bias}
\end{eqnarray}
$E\left(\hat f(y\mid \Theta)\bigg\vert \bY_n\right)$ denoting the expectation of 
$\hat f(y\mid \Theta)$ with respect to $[\Theta\mid\bY_n]$.

\section{{\bf Assumptions for the competing models and the true data generating distribution}}
\label{sec:assumptions}

\subsection{{\bf The EW model and the associated assumptions}}
\label{subsec:EW_post}

We assume the following version of the EW model: for every $i=1,\ldots,n$; $n=1,2,\ldots$,
$\left[Y_{ni}\mid \theta_i,\sigma\right] \sim N(\theta_i, \sigma^2)$, 
the normal distribution with mean $\theta_i$
and variance $\sigma^2$.
In the above, $\theta_i \stackrel{iid}{\sim} F$, $F \sim D(\alpha_n G_0)$, 
where $G_0$ is a completely specified, compactly supported probability measure. 
We assume in particular that $G_0$ is supported on some compacts set $\mathbb S$ such that
$\mathbb S\supset [-a-c,a+c]\cup [-2a,2a]$, where $[-a-c,a+c]$ is the support of the mixing distribution of the
true distribution; see Section \ref{subsec:true_distribution}. An alternative to the assumption of compact
support of $G_0$ is to assume that the expectation of $\exp\left(\frac{4a|\theta|}{\eta^2}\right)$ exists
with respect to $G_0$ and is finite, which would yield the same results as reported in this paper. 
However, for large enough $a$ and/or sufficiently small $\eta$, this would imply that $G_0$ is extremely thin-tailed, 
which would severely (and unrealistically) restrict the class of possible base measures. Hence, a sufficiently large 
compact support of $G_0$ that contains $[-a-c,a+c]$ seems to be a much more realistic assumption, 
which we adopt for our purpose. 
Choices of the parameter $\alpha_n$ will be discussed subsequently. 

Further we assume a sequence of priors on $\sigma$ as 
$\sigma/ \sigma_n\sim G$, where $\sigma_n~(>0)$ is a sequence of constants such that 
$\sigma_n\rightarrow 0$,
and $G$ is fixed. 
Denoting $G_n(s)=G(s/\sigma_n)$, it follows that $\sigma\sim G_n$. This assumption regarding the prior of 
$\sigma$ is very similar to that of \ctn{Ghosal07}. 
Following \ctn{Ghosal07} we also assume that $P(\sigma>\sigma_n)=O(\epsilon_n)$, where $\epsilon_n\rightarrow 0$. 
As we make precise later, we let the choice of $\epsilon_n$ depend upon the other prior parameters.
We also assume that there exists a positive sequence $\{b_n\}_{n=0}^{\infty}$ satisfying $0<b_n<\sigma_n$
and $P(\sigma>\sigma_n)/P(b_n<\sigma\leq\sigma_n)=O\left(\frac{\epsilon_n}{1-\epsilon_n}\right)$.
Additionally, we shall also require that $\sigma_n\sim b_n$, that is, $\sigma_n/b_n\rightarrow 1$, as $n\rightarrow\infty$.

That the above conditions on the prior of $\sigma$ are not self-contradicting can be easily seen from
the following example. Let $\pi(\sigma)=\theta^{-1}_n\exp\left(-\sigma/\theta_n\right)$ be the exponential prior distribution
of $\sigma$ with mean $\theta_n$. Let $\theta_n=1/(n^r+n^{s+\delta})$, $\sigma_n=1/(n^r+n^s)+1/n^r$
and $b_n=2/(n^r+n^s)$, where $r,s,\delta>0$, $r>s$ and $s+\frac{\delta}{2}>r$. Then $\sigma_n\sim b_n$,
$P(\sigma>\sigma_n)=\exp\left(-\frac{\sigma_n}{\theta_n}\right)\rightarrow 0$,
$P(b_n<\sigma\leq\sigma_n)=\exp\left(-\frac{b_n}{\theta_n}\right)-\exp\left(-\frac{\sigma_n}{\theta_n}\right)
\rightarrow 0$, 
$P(\sigma>\sigma_n)/P(b_n<\sigma\leq\sigma_n)\sim \exp\left(-n^{2(s+\frac{\delta}{2}-r)}\right)$.
With $\epsilon_n=\exp\left(-n^{2(s+\frac{\delta}{2}-r)}\right)$, it is easily seen that
$P(\sigma>\sigma_n)=O\left(\epsilon_n\right)$ and 
$P(\sigma>\sigma_n)/P(b_n<\sigma\leq\sigma_n)=O\left(\frac{\epsilon_n}{1-\epsilon_n}\right)$.

From the pure Bayesian perspective it may be preferable to choose the prior of $\sigma$ and $F$
to be independent of $n$, but our choices, which depend upon $n$, lead to fast convergence rates with respect to Bayesian $MISE$, and
hence can perhaps qualify as appropriate objective priors.
We let $\Theta_n = (\theta_1, \ldots, \theta_n)'$. 

Observe that for every value of the sample size $n=1,2,\ldots$, we have a data set $\{Y_{ni};i=1,\ldots,n\}$ of size $n$ with associated
parameters $\Theta_n$, $\alpha_n$, $\sigma_n$, $\epsilon_n$. The data points $Y_{ni}$ are assumed to be independent
for each $n$ and $i$. The array of random variables $\{Y_{ni};i=1,\ldots,n;n=1,2,\ldots\}$ is the well-known triangular array
of random variables; see, for example, \ctn{Serfling80}.
For notational simplicity we drop the suffix $``n"$ in 
$\{Y_{ni};i=1,\ldots,n\}$ and simply denote it by $\{Y_1,\ldots,Y_n\}$.




Assuming that the data are modeled by the original EW approach we study asymptotic properties 
of the posterior distribution of density estimators of 
the following specific form: 
\begin{equation}
\hat f_{EW}(y\mid \Theta_n,\sigma) = \frac{\alpha_n}{\alpha_n+n}A_n(y)
+\frac{1}{\alpha_n+n}\sum_{i=1}^{n}\frac{\varphi(\theta_i,\sigma+\hat k_n)}{(\sigma+\hat k_n)}\phi\left(
\frac{y-\theta_i}{\sigma+\hat k_n}\right)\mathbb I_{\{|y|\leq a\}},
\label{eq:ew_prior_pred} 
\end{equation}
where $\phi(\cdot)$ is the standard normal density, 
$\varphi(\theta,\sigma+\hat k_n)=\left[\Phi\left(\frac{a-\theta}{\sigma+\hat k_n}\right)
-\Phi\left(\frac{-a-\theta}{\sigma+\hat k_n}\right)\right]^{-1}$, 
where $\Phi(\cdot)$ is the distribution function of the standard normal distribution,
and 
$A_n(y) = \int_{\theta}\frac{\varphi(\theta,\sigma+\hat k_n)}{(\sigma+\hat k_n)}\phi\left(\frac{y-\theta}{\sigma+\hat k_n}
\right)\mathbb I_{\{|y|\leq a\}}dG_0(\theta)$. 
It is important to note the difference between the model assumption for the data and the random density
of our interest given by (\ref{eq:ew_prior_pred}); even though the latter adds $\hat k_n$ to $\sigma$, 
the former does not consider addition of any positive
constant to $\sigma$. 
In spite of slightly inflating the variance in (\ref{eq:ew_prior_pred}), the form of the true distribution
(\ref{eq:true_density_2}), to which (\ref{eq:ew_prior_pred}) converges {\it a posteriori}, is not severely restricted.

\subsection{{\bf SB model and the associated assumptions}}
\label{subsec:SB_post}
As in the case of the EW model, here we consider the following random density estimator:
\begin{equation}
\hat f_{SB}(y\mid \Theta_{M_{n}},\sigma) = \frac{1}{M_{n}}\sum_{i=1}^{M_{n}}
\frac{\varphi(\theta_i,\sigma+\hat k_n)}{(\sigma+\hat k_n)}\phi\left(\frac{y-\theta_i}{\sigma+\hat k_n}\right)
\mathbb I_{\{|y|\leq a\}},
\label{eq:sb_model}
\end{equation}
where $M_{n}$ is the maximum number of distinct components the mixture model can have and $\Theta_{M_{n}} = (\theta_1, \theta_2, \ldots, \theta_{M_{n}})'$.
As in the EW case, here also we assume the triangular array of random variables 
$\{Y_{ni};i=1,\ldots,n;n=1,2,\ldots\}$, and
we denote $\{Y_{ni};i=1,\ldots,n\}$ by $\{Y_1,\ldots,Y_n\}$ for notational simplicity.
Define $Z_i=j$ if $Y_{i}$ comes from the $j$-th component of the mixture model. Denote $z$ as the realized
vector of $Z$.
We make the same assumptions regarding $\alpha_n$ and the prior of $\sigma$ as in EW. 
Additionally, we let $M_{n}$ increase
with $n$. 

To perform our asymptotic calculations with respect to the SB model,
we need to shed light on an issue associated with the frequentist estimate of $\sigma^{2}_{T}$, 
the variance of the true density $f_0$ generating the data. The assumptions on the true distribution $f_0$
are provided in Section \ref{subsec:true_distribution}.

Let $n_j=\#\{t:z_t=j\}$, $n=\sum_{j=1}^{M_n}n_j$, $\bar Y_j=\frac{\sum_{t:z_t=j}Y_t}{n_j}$, and
$\hat\sigma^2_{T,n} = \frac{\sum_{j=1}^{M_n}\sum_{t:z_t=j}(Y_{t}-\bar Y_j)^2}{n} 
= \frac{\sum_{j=1}^{M_n}n_j
s_{j,n}^2}{n}$, where $s_{j,n}^2=\frac{\sum_{t:z_t=j}(Y_{t}-\bar Y_j)^2}{n_j}$. Now,
defining $\bar Y=\frac{\sum_{j=1}^{M_n}n_j\bar Y_j}{n}$ we note that $\frac{1}{n}
\sum_{i=1}^n\left(Y_i - \bar Y\right)^2$ 
can be expressed, 
for any allocation vector 
$z=(z_1,\ldots,z_n)'$, as
\begin{eqnarray}
\frac{1}{n}\sum_{i=1}^n\left(Y_i - \bar Y\right)^2 
&=&\frac{1}{n}\sum_{j=1}^{M_n}\sum_{t:z_t=j}(Y_t-\bar Y)^2\nonumber\\ 
&=&\frac{1}{n}\sum_{j=1}^{M_n}n_j\left(\bar Y_j - \bar Y\right)^2 + \hat\sigma^2_{T,n}. \nonumber
\end{eqnarray}
Since $\frac{1}{n}\sum_{i=1}^n\left(Y_i - \bar Y\right)^2\rightarrow\sigma^2_T$ a.s.,
it would follow from the above representation that $\hat\sigma^2_{T,n}\rightarrow\sigma^2_T$ a.s.
if it can be shown that $\frac{1}{n}\sum_{j=1}^{M_n}n_j\left(\bar Y_j - \bar Y\right)^2\rightarrow 0$ a.s.
Lemma S-1.1 of the supplement guarantees that it is indeed the case.




From Lemma S-1.1 
and the fact that 
$\frac{1}{n}\sum_{i=1}^{n}\left(Y_i - \bar Y\right)^2\rightarrow \sigma_{T}^{2}$ a.s. we can conclude
$\hat\sigma^2_{T,n}\rightarrow \sigma_{T}^{2}$, a.s.
%
%
So, as $n\rightarrow \infty$, $n\hat\sigma^2_{T,n}\sim n\sigma^2_T$ a.s., 
implying that as $n\rightarrow \infty$, 
$\sum_{j=1}^{M_n}\sum_{t:z_t=j}(Y_{t}
-\bar Y_j)^2$ becomes independent of $z$. 
We begin by writing $n{\hat \sigma_{T,n}}^2\sim C_n$, where $0<\frac{C_n}{n}<\bar C$ (for some sufficiently 
large constant $\bar C$) is a bounded sequence independent
of $z$ and has the same limiting behaviour as $\hat\sigma^2_{T,n}$.
Since we will perform our calculations when for each $n$, $|Y_i|<a;i=1,\ldots,n$, for some sufficiently large 
constant $a>0$, we have $0<\hat\sigma^2_{T,n}<4a^2$. Thus, we may choose $\bar C=4a^2$. 

To prove our results related to the SB model we will assume that for large $n$, $C_n/n$ is bounded below by an 
appropriate positive function of $\sigma_n$ and $\epsilon_n$ (to be made precise in the relevant lemmas and
theorems),
reasonably signifying that $\sigma^2_T$, and hence, $C_n/n$ should not be too small. 
In fact, we will compare the $MISE$ convergence rates of SB and EW assuming that $\sigma^2_T$
is large enough. In other words, we are interested in comparing the $MISE$ convergence rates
in challenging situations where it is quite difficult to learn about the true density.

\subsection{{\bf Assumptions regarding the true distribution}}
\label{subsec:true_distribution}

In this paper we assume that the true, data generating distribution is of the following form:
\begin{equation}
f_0(y)=\int_{-a-c}^{a+c}\frac{\varphi(\theta,k)}{k}\phi\left(\frac{y-\theta}{k}\right)
\mathbb I_{\left\{|y|\leq a\right\}}dF_0(\theta),
\label{eq:true_density_1}
\end{equation}
where $k$ is some unknown positive constant, and 
$F_0$ is an unknown distribution compactly supported on $[-a-c,a+c]$, for some constants $a>0$
and $c>0$. Thus, $f_0$ is compactly supported on $[-a,a]$.

Note that, using the mean value theorem for integrals, also known as the general mean value theorem (GMVT) 
we can re-write $f_0(y)$ as
\begin{equation}
f_0(y)=\frac{\varphi(\theta^*(y),k)}{k}\phi\left(\frac{y-\theta^*(y)}{k}\right)\mathbb I_{\left\{|y|\leq a\right\}},
\label{eq:true_density_2}
\end{equation}
where $\theta^*(y)\in (-a-c,a+c)$ may depend upon $y$.

For the EW and the SB models we will denote the respective $MISE$'s as $MISE(EW)$ and $MISE(SB)$, respectively.
Let $E^n_0$ denote the expectation of $\bY_n$ with respect to the true distribution $f_0$. 
We will compute and compare the rates of convergence to 0
of $E^n_0\left[MISE(EW)\right]$ and 
$E^n_0\left[MISE(SB)\right]$
when the true density $f_0$ is estimated using the EW model and the SB model, but with
the same set of data for any given sample size.

Before proceeding to the $MISE$ calculations, we first investigate the 
asymptotic forms of the posterior expectations of the EW-based
and the SB-based models given by (\ref{eq:ew_prior_pred})
and (\ref{eq:sb_model}), respectively. This we do in the next two sections.
These results, apart from being interesting in their own rights and showing explicitly the form
of the true distribution (the asymptotic form of posterior expected models), 
actually provide the orders of the bias terms of the corresponding 
$MISE$ calculations.


\section{{\bf Convergence of the posterior expectation of the competing models to the true distribution}}
\label{sec:posterior_mean_convergence}

\subsection{{\bf Convergence of the posterior mean of the EW model}}
\label{subsec:ew_posterior_convergence}



\begin{theorem}
\label{theorem:expec_ew}
Under the assumptions stated in Sections 
\ref{subsec:EW_post} and 
\ref{subsec:true_distribution},
\begin{eqnarray}
&&\sup_{|y|\leq a}\bigg\vert E\left(\hat f_{EW}(y\mid \Theta_n,\sigma)\bigg\vert \bY_n\right)
-f_0(y)\bigg\vert\nonumber\\
&&\quad\quad =O\left(\frac{\alpha_n}{\alpha_n+n}+\frac{n}{\alpha_n+n}(B_n+\epsilon^*_n+\sigma_n+|\hat k_n-k|)\right),
\label{eq:ew_expec}
\end{eqnarray}
where
\begin{equation}
B_n=\frac{\alpha_n+n}{\alpha_n}e^{-\frac{c^2}{4\sigma_n^2}},
\label{eq:B_n}
\end{equation}
and
\begin{equation}
\epsilon_n^*=\frac{\epsilon_n}{1-\epsilon_n}\exp{\left(\frac{n(a+c_1)^2}{2{b_n}^2}\right)}\frac{(\alpha_n+n)^n}{(\alpha_n)^n}
\frac{1}{H_0^n}.
\label{eq:epsilon_star}
\end{equation}
In the above, 
$c_1>0$, $\{b_n\}$ is a sequence of positive numbers such that $0<b_n<\sigma_n$ for all $n$ with
$P\left(\sigma>\sigma_n\right)=O\left(\epsilon_n\right)$ and
$P\left(\sigma>\sigma_n\right)/P\left(b_n<\sigma\leq\sigma_n\right)=O\left(\frac{\epsilon_n}{1-\epsilon_n}\right)$,
where $\epsilon_n=o(1)$, as $n\rightarrow\infty$.
Also, $\alpha_n=O(n^\omega)$, $0<\omega<1$, 
$H_0=\int_{-a-c}^{a+c} dG_0(x)$,
and 
$f_0(y)=\frac{\varphi(\mu^*(y),k)}{k}\phi\left(\frac{y-\mu^*(y)}{k}\right)\mathbb I_{\left\{|y|\leq a\right\}}$ 
is a well-defined density, where
$\mu^*(y)\in (-a-c, a+c)$ for each $y$.
The constant involved in the order (\ref{eq:ew_expec}) is independent of $\bY_n$.\\
\end{theorem}

\begin{proof}
See Section S-2.1.1 of the supplement. 
The proof depends upon several lemmas, the statements and proofs of which are provided in 
Section S-2.1 of the supplement. Below we provide a brief discussion of the lemmas.
\end{proof}

The terms $\epsilon_n^*$ and $B_n$ arise as the orders of the
posterior probabilities $P(\sigma> \sigma_n\vert \bY_n)$ and 
$P\left(\theta_i\in [-a-c, a+c]^{c}, \sigma\leq \sigma_n\vert \bY_n\right)$, respectively. 
The first term in the order 
(\ref{eq:ew_expec}) of Theorem \ref{theorem:expec_ew}) is
contributed by the order of the term $\frac{\alpha_n}{\alpha_n+n}A_n$, where $A_n$ is already defined in
connection with (\ref{eq:ew_prior_pred}). 
These results, used for proving Theorem \ref{theorem:expec_ew},
which also play important roles in proving our main Theorem \ref{theorem:mise_ew} on $MISE$
of the EW model, 
are made precise
in Lemmas S-2.1, S-2.2, and S-2.3 of the supplement, along with their proofs. 
We make several remarks below in connection with Theorem \ref{theorem:expec_ew}.
\\[2mm]
{\it Remark 1:}
To make the bias term implied by Theorem \ref{theorem:expec_ew} tend to zero as $n\rightarrow\infty$, 
we will choose $\epsilon_n$ such that $\epsilon_n^*\rightarrow 0$; in other words,
we choose $\epsilon_n$ such that
$\frac{\epsilon_n}{1-\epsilon_n}\prec \left[e^{\frac{n(a+c_1)^2}{2{b_n}^2}}\frac{(\alpha_n+n)^n}{(\alpha_n)^n H_0^n}
\right]^{-1}$ (for any two sequences $a^{(1)}_n$ and $a^{(2)}_n$ we say $a^{(1)}_n\prec a^{(2)}_n$ if
$\frac{a^{(1)}_n}{ a^{(2)}_n}\rightarrow 0$).
Furthermore, we will choose $\alpha_n$ such that $\frac{\alpha_n}{\alpha_n+n}\rightarrow 0$,
$\frac{n}{\alpha_n+n}\rightarrow 1$, and $B_n\rightarrow 0$. 
We will also discuss the consequences if these fail to hold.
\\[2mm] 
{\it Remark 2:}
An important point which we stated in Theorem \ref{theorem:expec_ew} is that the constant involved in the 
order (\ref{eq:ew_expec}) is independent of $\bY_n$.
Hence it follows that 
\begin{eqnarray}
&&E^n_0\left[\sup_{|y|\leq a}\bigg\vert E\left(\hat f_{EW}(y\mid \Theta_n,\sigma)\bigg\vert \bY_n\right)
-f_0(y)\bigg\vert\right]\nonumber\\
&&\quad\quad =O\left(\frac{\alpha_n}{\alpha_n+n}+\frac{n}{\alpha_n+n}(B_n+\epsilon^*_n+\sigma_n
+E^n_0|\hat k_n-k|)\right),\nonumber
\end{eqnarray}
where $E^n_0|\hat k_n-k|\rightarrow 0$ since $|\hat k_n-k|\stackrel{a.s.}{\longrightarrow}0$ and
$(\hat k_n-k)^2$ is uniformly integrable with respect to $f_0$ by assumption.
\\[2mm]
{\it Remark 3:}
The proof of Theorem \ref{theorem:expec_ew} shows that for each $y$, $\mu^*(y)$ corresponds to $\sigma=0$ (the limit of the
sequence $\sigma_n$), and so $\mu^*(y)$ is non-random, not depending upon the data.

\subsection{{\bf Convergence of the posterior mean of the SB model}}
\label{subsec:sb_posterior_convergence}

For proving results on the SB model it is necessary
to introduce some necessary concepts
and notation. These new ideas are needed for the SB model and not for 
the EW model since the latter is a much less complex model than the former. In particular, 
note that unlike the EW case where each $\theta_i$ is represented in $L(\Theta_{M_{n}}, \bY_n, z)$, 
$\theta_i$ in the SB model may or may 
not be allocated to $Y_{i}$ for some $i$, that is, there can exist $z$ such that 
$z_l \neq i$, $l=1, \ldots, n$. Suppose that $R_1^*=\{z: \mbox{no}\hspace{2mm}z_l=i\}$, 
$R_2^*=\left(R^*_1\right)^c=\{z: \mbox{at}\hspace{2mm}
\mbox{least}\hspace{2mm}\mbox{one}\hspace{2mm}z_l=i\}$. Note that 
\#$R_1^*=(M_{n}-1)^n$ and \#$R_2^*=M_{n}^n-(M_{n}-1)^n$. 

If $z\in R_1^*$, let $\Theta_{z}$ denote the set of $\theta_{l}$'s present in the
likelihood and let \#$\Theta_{z}$=$j$, where $j=1, \ldots, (M_{n}-1)$. By the definition of $R_1^*$, $\theta_i$ is not present in the likelihood. 
Without loss of generality let us assume that $\theta_1, \ldots, \theta_{j}$ are represented in the likelihood
$L(\Theta_{M_{n}}, z, \bY_n)$. For obtaining bounds of $L(\Theta_{M_{n}}, z, \bY_n)$ it is enough to consider only $\Theta_{z}$. 
For $z\in R_1^*$, we split
the range of integration in the numerator in the following way: 
\begin{eqnarray}
\int_{\Theta_{z}} L(\Theta_{M_{n}}, z, \bY_n) dG_n(\sigma)dH(\Theta_{M_{n}}) 
&=& \sum_{l=1}^{j}\int_{W_l}L(\Theta_{M_{n}}, z, \bY_n) dG_n(\sigma)dH(\Theta_{M_{n}}) \nonumber\\
&&\ \ +\int_{W_{j^c}}L(\Theta_{M_{n}}, z, \bY_n) dG_n(\sigma)dH(\Theta_{M_{n}}), \nonumber\\
\label{eq:sb_split}
\end{eqnarray}
where $W_1$=\{$\theta_1\in [-a-c, a+c]^c$\},
$W_l$=\{$\theta_1\in [-a-c, a+c],\ldots, \theta_{l-1}\in [-a-c, a+c], \theta_{l}\in [-a-c, a+c]^c$\} 
for $l=2,\ldots,j$,
$W_{j^c}$=\{$\theta_1\in [-a-c, a+c],\ldots, \theta_{j-1}\in [-a-c, a+c], \theta_{j}\in [-a-c, a+c]$\}. 

Also define $V_j$ and $E$ as the following:\\
$V_j=\{z\in R_1^*:\hspace{2mm}\mbox{exactly}\hspace{2mm}j\ \ \mbox{many}\hspace{2mm}\theta_{l}\hspace{0.5mm}\mbox{'s}
\hspace{2mm}\mbox{are}\hspace{2mm}\mbox{in}\hspace{2mm}
L(\Theta_{M_{n}},z,y)\}$,
and\\
%
$E$ = \{all $\theta_l$'s present in the likelihood are in $[-a-c, a+c]$\}.

\begin{theorem}
\label{theorem:sb_expectation}
Under the assumptions stated in Sections 
\ref{subsec:SB_post} and 
\ref{subsec:true_distribution},
and under the further assumption that
$$
\frac{C_n}{n}\gtrsim \dfrac{\left[\log\left(\frac{1}{\sigma_n}\right)+O\left(\frac{1}{n}
\log\left(\frac{1-\epsilon_n}{\epsilon_n}\right)\right)\right]}{\left(\frac{1}{\sigma_n^2}\right)}
\ \ \ \ (``\gtrsim" \ \ indicates \ \ ``\geq" \ \ as \ \ n\rightarrow\infty),$$
the following holds: 
\begin{eqnarray}
&& \sup_{|y|\leq a}\bigg\vert 
E\left(\hat f_{SB}(y\mid \Theta_{M_{n}},\sigma)\bigg\vert \bY_n\right)
- f_0(y) \bigg\vert\nonumber\\
&& \ \ \ \  
=O\left(M_{n} B_{M_{n}}+\left(1-\frac{1}{M_{n}}\right)^n \left(\frac{\alpha_n+M_{n}}{\alpha_n}\right)^{M_n}
+\epsilon_{M_n}^*+\sigma_n+|\hat k_n-k|\right),
\label{eq:sb_bias}
\end{eqnarray}
where 
$\epsilon_{M_n}^*=
\frac{\epsilon_n}{1-\epsilon_n}\exp\left(\frac{n(a+c_1)^2}{2(b_n)^2}\right)\frac{(\alpha_n+M_{n})^{M_n}}{\alpha_n^{M_n} H_0^{M_n}}$,
$B_{M_{n}}=\frac{(\alpha_n+M_{n})}{\alpha_n} \exp{\left(-\frac{c^2}{4\sigma_n^2}\right)}$, 
and $b_n$ is as defined in Theorem \ref{theorem:expec_ew}.
Also, for every $y$, 
$\theta^*(y)\in (-a-c, a+c)$, and $E^n_0$ denotes the expectation with respect to the true distibution
of $\bY_n$, given by $f_0(y)=\frac{\varphi(\theta^*(y),k)}{k}\phi\left(\frac{y-\theta^*(y)}{k}\right)
\mathbb I_{\left\{|y|\leq a\right\}}$.
The constant involved in the above order is independent of $\bY_n$.
\end{theorem}

\begin{proof}
See Section S-2.2.1 of the supplement.
The proof depends upon  
several lemmas, all of which are stated and proved
in Section S-2.2 of the supplement. 
\end{proof}
Several remarks regarding the above theorem follows.
\\[2mm]
{\it Remark 1:}
We will choose $\epsilon_n,\alpha_n,M_{n}$ such that the right hand side of
(\ref{eq:sb_bias}) goes to zero. 
In (\ref{eq:sb_bias}) the terms $\epsilon_{M_n}^*$, $M_nB_{M_n}$ and 
$\left(1-\frac{1}{M_{n}}\right)^n \left(\frac{\alpha_n+M_{n}}{\alpha_n}\right)^{M_n}$
are contributed by the orders of the posterior probabilities
$P(\sigma>\sigma_n\vert\bY_n)$, $P(Z\in R^*_1,\Theta_{M_n}\in E^c,\sigma\leq\sigma_n\vert\bY_n)$,
and $P(Z\in R^*_1,\Theta_{M_n}\in E,\sigma\leq\sigma_n\vert\bY_n)$, respectively. The formal
statements and proofs of these results are provided in the forms of Lemmas S-2.4, S-2.5, and S-2.6
of Section S-2.2 of the supplement; see also Lemma S-2.7.
\\[2mm]
{\it Remark 2:} Note that if $M_{n}< \sqrt{n}$, then it is easy to verify, using L'Hospital's rule, 
that the asymptotic
order of $P(Z\in R^*_1,\Theta_{M_n}\in E,\sigma\leq\sigma_n\vert\bY_n)$, given by
$\left(1-\frac{1}{M_{n}}\right)^n \left(\frac{\alpha_n+M_{n}}{\alpha_n}\right)^{M_n}$, tends to zero
as $n\rightarrow\infty$.  
Similarly, $P(Z\in R^*_1,\Theta_{M_n}\in E^c,\sigma\leq\sigma_n\vert\bY_n)$ can be made to tend
to zero by making $M_nB_{M_n}\rightarrow 0$. Combining these two results show that 
if the maximum number of components is
small compared to the data size, then, given an appropriate estimator $C_n/n$ 
of the true population variance $\sigma^2_T$, 
the probability that any mixture component will remain empty tends to zero as data size
increases.
On the other hand, as we show later in 
Section \ref{subsec:param_sb}  
if $M_{n}>n$, the probability that a mixture component will remain empty may converge to 1
as $n\rightarrow\infty$.
\\[2mm]
{\it Remark 3:}
It is important to make a few remarks regarding the choice of $c_1$. 
Firstly, note that the term $O\left(\frac{1}{n}\log\left(\frac{1-\epsilon_n}{\epsilon_n}\right)\right)$
appears because of the involvement of $O(\epsilon_n)$ and $O(1-\epsilon_n)$ in the proof of
Theorem \ref{theorem:sb_expectation}. 
Assuming that the limit of $O(1-\epsilon_n)/(1-\epsilon_n)$ exists as $n\rightarrow\infty$, one can easily verify that 
$O\left(\frac{1}{n}\log\left(\frac{1-\epsilon_n}{\epsilon_n}\right)\right)\sim 
\frac{1}{n}\log\left(\frac{1-\epsilon_n}{\epsilon_n}\right)$, so that 
$O\left(\frac{1}{n}\log\left(\frac{1-\epsilon_n}{\epsilon_n}\right)\right)$ is asymptotically
independent of the constants $c_2$ and $c_3$.
In other words, the required condition becomes
\[
\frac{C_n}{n}\gtrsim \sigma^2_n\log\left(\frac{1}{\sigma^2_n}\right)+
\frac{\sigma^2_n}{n}\log\left(\frac{1-\epsilon_n}{\epsilon_n}\right).
\]
Since $\sigma^2_n\rightarrow 0$, the first term $\sigma^2_n\log\left(\frac{1}{\sigma^2_n}\right)\rightarrow 0$
as $n\rightarrow\infty$.
Now assuming $\epsilon^*_{M_n}=r_n$, where $r_n\rightarrow 0$ as $n\rightarrow\infty$, we have
\[
\frac{\sigma^2_n}{n}\log\left(\frac{1-\epsilon_n}{\epsilon_n}\right)
=-\frac{\sigma^2_n}{n}\log (r_n)+\frac{\sigma^2_n}{b^2_n}\frac{(a+c_1)^2}{2}
+\frac{\sigma^2_n}{n}M_n\log\left(\frac{\alpha_n+M_n}{\alpha_n H_0}\right).
\]
For suitable choices of the sequences $\sigma_n$, $\alpha_n$ and $M_n$, the first and the third
terms of the right hand side of the above expression tend to zero. Indeed, as in Lemma \ref{lemma:compare1}
of Section \ref{sec:comparison_sb_ew}, if we assume $\alpha_n=n^{\omega}$, $M_n=n^b$, where $\omega<1$, $b<1$
and $\omega<b$, then the first and the third terms tend to zero if we choose $r_n=n^{-t}$ for $t\geq 1$
and $\sigma^2_n=n^{-b}$. Now, assume that $\sigma^2_n\sim b^2_n$. Then,
\[
 \frac{\sigma^2_n}{n}\log\left(\frac{1-\epsilon_n}{\epsilon_n}\right)
 \rightarrow \frac{(a+c_1)^2}{2}.
\]
Recalling that $0<\frac{C_n}{n}<4a^2$ for all $n$, we must have $\frac{(a+c_1)^2}{2}<4a^2$.
This holds if and only if $0<c_1<(2\sqrt{2}-1)a$. Hence, we must set $c_1\in\left(0,(2\sqrt{2}-1)a\right)$. 
This also implies that for consistency of $C_n/n$ we must have $\frac{a^2}{2}<\sigma^2_T<4a^2$.
Thus, we are interested in situations where $\sigma^2_T$, the variance of the true density $f_0$ 
is large for large enough $a$; that is, we are interested in situations where it is 
indeed a challenging task to learn about $f_0$. Consequently, we will make asymptotic comparisons of the SB method
with the EW method assuming this challenging set-up where $\frac{a^2}{2}<\sigma^2_T<4a^2$ holds.
Hence, for comparison purpose we set $0<c_1<(2\sqrt{2}-1)a$ for both EW and SB.
%




In order to make asymptotic comparisons between the models of EW and SB, first it is necessary
to ensure than both are consistent estimators of the same true density. The following theorem, proved
in Section S-2.3 of the supplement, show that this is indeed the case.
\begin{theorem}
\label{theorem:mu_theta}
The models of EW and SB converge to the same distribution. In other words, for every $y$, $\mu^*(y)=\theta^*(y)$,
where $\mu^*(y)$ is given in Theorem \ref{theorem:expec_ew} and $\theta^*(y)$ is given in
Theorem \ref{theorem:sb_expectation}.
\end{theorem}

\begin{proof}
See Section S-2.3 of the supplement.
\end{proof}

We strengthen our convergence results given by Theorems \ref{theorem:expec_ew} and \ref{theorem:sb_expectation}
by obtaining the orders of the $MISE$ of the models given by (\ref{eq:ew_prior_pred}) and 
(\ref{eq:sb_model}). As already mentioned, 
Theorems \ref{theorem:expec_ew} and \ref{theorem:sb_expectation} are directly
related to the bias of the $MISE$ which can be broken up into a variance part and a bias part.

\input{ew_post}


\input{sb_post}


\input{comparison2}



\input{simulation_study}

\input{param}

\input{pi_model}



\section*{{\bf Description of the supplement}}

Section S-1 contains proof of the result associated with Section 4.2,
Section S-2 contains proofs of the results presented in Section 5; the proofs of the results
provided in Section 6 are given in Section S-3, and Section S-4 contains proofs of the results
associated with Section 8. An overview of the asymptotic calculations associated with Section 9
is provided in Section S-5.
Finally, in Section S-6, the ``large $p$, small $n$" problem for both EW and SB is investigated.



\newpage

\input{supp}

\renewcommand\baselinestretch{1.3}
\normalsize
\bibliographystyle{ECA_jasa}
\bibliography{thesis_saby}

\end{document}

%% file: ew_post.tex
\section{{\bf MISE bounds for the competing models}}
\label{sec:mise_bounds}


\subsection{{\bf The main result for the EW model}}
\label{subsec:ew_post_rate}
\begin{theorem}
\label{theorem:mise_ew}
Under the assumptions of Theorem \ref{theorem:expec_ew}, 
\begin{equation}
E^n_0\left[MISE(EW)\right]=O\left(\left(\frac{\alpha_n}{\alpha_n+n}\right)^2+B_n+\epsilon_n^*+\sigma_n^2
+E^n_0(\hat k_n-k)^2\right).
\label{eq:expected_mise_order_ew}
\end{equation}
\end{theorem}
Note that, due to uniform integrability, $E^n_0|\hat k_n-k|^2\rightarrow 0$ as $n\rightarrow\infty$.

To prove Theorem \ref{theorem:mise_ew} we will break up $MISE$ into variance and bias parts, following
representation (\ref{eq:mise_bayesian_3}) of $MISE$, and will
obtain bounds for the variance and the bias parts separately. 
These bounds will be independent of both $y$ and $\bY_n$.

Note that
\begin{eqnarray}
&&Var(\hat f_{EW}(y\mid \Theta_n,\sigma)\vert \bY_n)=
\frac{1}{(\alpha_n+n)^2}\left[\sum_{i=1}^{n} Var\left(\frac{\varphi(\theta_i,\sigma+\hat k_n)}{(\sigma+\hat k_n)}
\phi\left(\frac{y-\theta_i}{\sigma+\hat k_n}\right)\bigg\vert \bY_n\right) \right.\nonumber\\
&&\ \ +\left.\sum_{i=1}^{n}\sum_{j=1, j\neq i}^{n}Cov\left(\frac{\varphi(\theta_i,\sigma+\hat k_n)}{(\sigma+\hat k_n)}
\phi\left(\frac{y-\theta_i}{\sigma+\hat k_n}\right), 
\frac{\varphi(\theta_j,\sigma+\hat k_n)}{(\sigma+\hat k_n)}
\phi\left(\frac{y-\theta_j}{\sigma+\hat k_n}\right)\bigg\vert \bY_n\right)\right]. \nonumber\\
\label{eq:mise_var1}
\end{eqnarray}

\subsubsection{{\bf Order of $Bias(\hat f_{EW}(y\mid \Theta_n,\sigma))$}}
\label{subsubsec:bias_hatf}

It follows from Theorem \ref{theorem:expec_ew} that
\begin{eqnarray}
&&E\left(\hat f_{EW}(y\mid \Theta_n,\sigma)\bigg\vert \bY_n\right)-f_0(y)\nonumber\\
&=& O\left(\frac{\alpha_n}{\alpha_n+n}+\frac{n}{\alpha_n+n}(B_n+\epsilon_n^*+\sigma_n+|\hat k_n-k|)\right). 
\label{eq:bias_ew_order}
\end{eqnarray}
%
%
We denote $\frac{\alpha_n}{\alpha_n+n}+\frac{n}{\alpha_n+n}(B_n+\epsilon_n^*+\sigma_n+|\hat k_n-k|)$ by $S^*_n$.

\subsubsection{{\bf Order of $Var\left(\frac{\varphi(\theta_i,\sigma+\hat k_n)}{(\sigma+\hat k_n)}
\phi\left(\frac{y-\theta_i}{\sigma+\hat k_n}\right)\bigg\vert \bY_n\right)$}}
\label{subsubsec:var_hatf}


\begin{lemma}
\label{lemma:var_ew}
\begin{equation}
Var\left(\frac{\varphi(\theta_i,\sigma+\hat k_n)}{(\sigma+\hat k_n)}
\phi\left(\frac{y-\theta_i}{\sigma+\hat k_n}\right)\bigg\vert \bY_n\right)
=O\left(B_n+\epsilon_n^*\right).
\label{eq:ew_var_order1}
\end{equation}
\end{lemma}

\begin{proof}
See Section S-3.1.1 of the supplement.
\end{proof}

\subsubsection{{\bf Order of the covariance term}}
\label{subsubsec:order_cov}

For $i\neq j$, let \\
$\xi_{in}=\left[\frac{\varphi(\theta_i,\sigma+\hat k_n)}{(\sigma+\hat k_n)}\phi\left(\frac{y-\theta_i}{\sigma+\hat k_n}\right)
-E\left(\frac{\varphi(\theta_i,\sigma+\hat k_n)}{(\sigma+\hat k_n)}
\phi\left(\frac{y-\theta_i}{\sigma+\hat k_n}\right)\bigg\vert \bY_n\right)\right]$ and \\
$\xi_{jn}=\left[\frac{\varphi(\theta_j,\sigma+\hat k_n)}{(\sigma+\hat k_n)}\phi\left(\frac{y-\theta_j}{\sigma+\hat k_n}\right)-
E\left(\frac{\varphi(\theta_j,\sigma+\hat k_n)}{(\sigma+\hat k_n)}
\phi\left(\frac{y-\theta_j}{\sigma+\hat k_n}\right)\bigg\vert \bY_n\right)\right]$.

\begin{eqnarray}
cov_{ij}
&=& Cov\left(\frac{\varphi(\theta_i,\sigma+\hat k_n)}{(\sigma+\hat k_n)}\phi\left(\frac{y-\theta_i}{\sigma+\hat k_n}\right), 
\frac{\varphi(\theta_j,\sigma+\hat k_n)}{(\sigma+\hat k_n)}
\phi\left(\frac{y-\theta_j}{\sigma+\hat k_n}\right)\bigg\vert \bY_n\right) \nonumber\\
&=& E\left(\left[\xi_{in}-E\left(\xi_{in}\bigg\vert \bY_n\right)\right]
\left[\xi_{jn}-E\left(\xi_{jn}\bigg\vert \bY_n\right)\right]\bigg\vert \bY_n\right).   \nonumber
\end{eqnarray}

\begin{lemma}
\label{lemma:cov}
\begin{equation}
cov_{ij}=O\left(B_n+\epsilon_n^*\right). \label{eq:cov_order_ew}
\end{equation}
\end{lemma}

\begin{proof}
Follows from Lemma \ref{lemma:var_ew} using the Cauchy-Schwartz inequaity.
\end{proof}

\subsubsection{{\bf Final calculations putting together the above results}}

We thus have,
\begin{eqnarray}
&&MISE(EW)  \nonumber\\
&=&O\left(\frac{1}{(\alpha_n+n)^2}\left[n(B_n+\epsilon_n^*)+n(n-1)(B_n+\epsilon_n^*)\right]+
\left(S_n^*\right)^2\right). \nonumber\\
\label{eq:mise_ew2}
\end{eqnarray}
Assuming $n$ to be large enough such that $\frac{n(n-1)}{(\alpha_n+n)^2}\approx 1$, the actual form of $MISE$ given in equation (\ref{eq:mise_ew2}) 
can be simplified further for comparison purpose. 
Note that if $\frac{\alpha_n}{\alpha_n+n}\rightarrow 1$, then
the conditional distribution based on the Polya urn scheme implies that $\theta_l$'s arise from 
$G_0$ only, which seems to be too restrictive an
assumption. Thus assuming $\frac{\alpha_n}{\alpha_n+n}\rightarrow 0$ seems more plausible as it entails 
a nonparametric set up.
We assume $\alpha_n=n^{\omega}$, $\omega<1$, so that $\frac{n}{\alpha_n+n}\approx 1$ for large $n$. So, (\ref{eq:mise_ew2})
boils down to
\begin{eqnarray}
&&MISE(EW)  \nonumber\\
&=&O\left(\frac{n}{(\alpha_n+n)^2}(B_n+\epsilon_n^*)+(B_n+\epsilon_n^*)+
\left(\frac{\alpha_n}{\alpha_n+n}+B_n+\epsilon_n^*+\sigma_n+|\hat k_n-k|\right)^2\right). \nonumber\\
\label{eq:mise_ew3}
\end{eqnarray}
We can further simplify this form by retaining only the higher order terms. 
Note that we have assumed that under certain conditions
$B_n$, $\epsilon_n^*$ and $\frac{\alpha_n}{\alpha_n+n}$ converge to 0, and hence 
$\frac{n}{(\alpha_n+n)^2}(B_n+\epsilon_n^*)\prec O\left(B_n+\epsilon_n^*\right)$.
In the third term of equation (\ref{eq:mise_ew3}) there are two extra terms, 
$\sigma_n$ and $\frac{\alpha_n}{\alpha_n+n}$ under the squared term. 
Adjusting for that term we write the simplified form of 
of $MISE$ as
\begin{equation}
MISE(EW)=O\left(\left(\frac{\alpha_n}{\alpha_n+n}\right)^2+B_n+\epsilon_n^*+\sigma_n^2+(\hat k_n-k)^2\right).
\label{eq:mise_order_ew}
\end{equation}

Note that the order remains unchanged after 
taking expectation
with respect to $E^n_0$. Only $|\hat k_n-k|^2$ in (\ref{eq:mise_order_ew}) is replaced with 
$E^n_0|\hat k_n-k|^2$, which converges to zero due to uniform integrability. 
In other words, Theorem \ref{theorem:mise_ew} follows.

%% file: sb_post.tex

\subsection{{\bf The main result for the SB model}}
\label{subsec:sb_model}

\begin{theorem}
\label{theorem:mise_sb}
Under the above assumptions of
Theorem \ref{theorem:sb_expectation},
\begin{equation}
E^n_0\left[MISE(SB)\right] = O\left(\left(1-\frac{1}{M_n}\right)^n
\left(\frac{\alpha_n+M_n}{\alpha_n}\right)^{M_n}+M_nB_{M_n}+\epsilon^*_{{M}_n}+\sigma_n^2+E^n_0(\hat k_n-k)^2\right). \nonumber\\
\label{eq:expected_mise_order_sb}
\end{equation}
\end{theorem}

\subsubsection{{\bf Bounds of $Var(\hat f_{SB}(y\mid \Theta_{M},\sigma))$}}
\label{subsubsec:Var_f1}

\begin{lemma}
\label{lemma:var_order_sb}
\begin{eqnarray}
&&\frac{1}{{M_n}^2}\sum_{i=1}^{M_n}Var\left(\frac{\varphi(\theta_i,\sigma+\hat k_n)}{(\sigma+\hat k_n)}
\phi\left(\frac{y-\theta_i}{\sigma+\hat k_n}\right)\bigg\vert \bY_n\right) \nonumber\\
&=& O\left(\frac{1}{M_n}\left(M_nB_{M_n}+\left(1-\frac{1}{M_n}\right)^n\left(\frac{\alpha_n+M_n}{\alpha_n}\right)^{M_n}
+\epsilon^*_{{M}_n}\right)\right)
\label{eq:sb_var_order}
\end{eqnarray}
\end{lemma}

\begin{proof}
See Section S-3.2.1 of the supplement.
\end{proof}

\subsubsection{{\bf Order of the covariance term}}
\label{subsubsec:order_cov_sb}

\begin{lemma}
\label{lemma:cov_order_sb}
\begin{eqnarray}
&&\frac{1}{M^2_n}\sum_{i=1}^{M_n}\sum_{j=1,j\neq i}Cov\left(\frac{\varphi(\theta_i,\sigma+\hat k_n)}{(\sigma+\hat k_n)}
\phi\left(\frac{y-\theta_i}{\sigma+\hat k_n}\right),
\frac{\varphi(\theta_j,\sigma+\hat k_n)}{(\sigma+\hat k_n)}
\phi\left(\frac{y-\theta_j}{\sigma+\hat k_n}\right)\bigg\vert \bY_n\right) \nonumber\\
&&\ \ =O\left(M_nB_{M_n}+
\left(1-\frac{1}{M_n}\right)^n\left(\frac{\alpha_n+M_n}{\alpha_n}\right)^{M_n}+\epsilon^*_{{M}_n}\right).
\label{eq:sb_cov_order}
\end{eqnarray}
\end{lemma}

\begin{proof}
Follows from Lemma \ref{lemma:var_order_sb} using the Cauchy-Schwartz inequality.
\end{proof}

\subsubsection{{\bf Bound for the bias term}}
\label{sb_bias}

The bias of the $MISE$, which we denote by $Bias(\hat f_{SB}(y\mid\Theta_{M_n},\sigma))$, is given by 
\begin{eqnarray}
\frac{1}{M_n}\sum_{j=1}^{M_n}E\left(\frac{\varphi(\theta_j,\sigma+\hat k_n)}{(\sigma+\hat k_n)}
\phi\left(\frac{y-\theta_j}{\sigma+\hat k_n}\right)
\bigg\vert \bY_n\right)- f_0(y).\nonumber
\end{eqnarray}
From Theorem \ref{theorem:sb_expectation} 
we have,
\begin{equation}
Bias(\hat f_{SB}(y\mid\Theta_{M_n},\sigma))^2=O\left(\left[M_nB_{M_n}+\left(1-\frac{1}{M_n}\right)^n\left(\frac{\alpha_n+M_n}{\alpha_n}\right)^{M_n}+\epsilon^*_{{M}_n}+\sigma_n+|\hat k_n-k|\right]^2\right).
\label{eq:sb_bias_order}
\end{equation}

Thus, the complete order of $MISE$ can be obtained by adding up these individual orders of (\ref{eq:sb_var_order}), (\ref{eq:sb_cov_order})
and (\ref{eq:sb_bias_order}), yielding
\begin{eqnarray}
&&MISE(SB)=O\left(\frac{1}{M_n}\left(M_nB_{M_n}+\left(1-\frac{1}{M_n}\right)^n\left(\frac{\alpha_n+M_n}{\alpha_n}\right)^{M_n}+\epsilon^*_{{M}_n}\right)\right) \nonumber\\
&&\ \ + O\left(M_nB_{M_n}+\left(1-\frac{1}{M_n}\right)^n\left(\frac{\alpha_n+M_n}{\alpha_n}\right)^{M_n}+
\epsilon^*_{{M}_n}\right)\nonumber\\
&&\ \ + O\left(\left[M_nB_{M_n}+\left(1-\frac{1}{M_n}\right)^n\left(\frac{\alpha_n+M_n}{\alpha_n}\right)^{M_n}+\epsilon^*_{{M}_n}+\sigma_n+|\hat k_n-k|\right]^2\right) \nonumber\\
\label{eq:mise_sb1}
\end{eqnarray}
\\

Appropriate choices of the sequences involved in $MISE(SB)$ guarantee that $M_nB_{M_n}\rightarrow 0$,  
$\left(1-\frac{1}{M_n}\right)^n\left(\frac{\alpha_n+M_n}{\alpha_n}\right)^{M_n}\rightarrow 0$ and
$\epsilon^*_{{M}_n}\rightarrow 0$. With these we have
\begin{eqnarray}
MISE(SB) &=& O\left(\left(1-\frac{1}{M_n}\right)^n
\left(\frac{\alpha_n+M_n}{\alpha_n}\right)^{M_n}+M_nB_{M_n}+\epsilon^*_{{M}_n}+\sigma_n^2+(\hat k_n-k)^2\right). \nonumber\\
\label{eq:mise_order_sb}
\end{eqnarray}
As before, the order remains unchanged after taking expectation with respect to $E^n_0$; only
$|\hat k_n-k|^2$ is to be replaced with $E^n_0|\hat k_n-k|^2$, thus proving Theorem \ref{theorem:mise_sb}.

%% file: comparison2.tex
\section{{\bf Comparison between MISE's of EW and SB}}
\label{sec:comparison_sb_ew}

As claimed in \ctn{Sabya10a} and \ctn{Sabya10} (see \ctn{Sabya13} for the complete details), 
the SB model is much more
efficient than the EW model in terms of computational complexity and ability to 
approximate the true underlying clustering or regression.
Here we investigate the conditions under which the model of SB beats that of EW in terms of $MISE$.
In particular, we provide conditions which guarantee that each term of the order of $MISE(EW)$
dominates the corresponding term of the order of $MISE(SB)$ (for any two sequences $\{a^{(1)}_n\}$ 
and $\{a^{(2)}_n\}$
we say that $a^{(1)}_n$ dominates $a^{(2)}_n$ if $a^{(2)}_n/a^{(1)}_n\rightarrow 0$ as $n\rightarrow\infty$).

For the purpose of comparison we will use the same values of $b_n$, $\sigma_n$, $\epsilon_n$, for all $n$, for both SB and EW model,
in a way such that both the $MISE$'s converge to 0.

\begin{lemma}
\label{lemma:compare1}
Let $\alpha_n=n^{\omega}$, $M_n=n^b$, where $\omega <1$, $b<1$, and $\omega <b$. Then,\\
$\frac{\epsilon^*_{M_n}}{\epsilon^*_n}\rightarrow 0$.
\end{lemma}
\begin{proof}
The proof follows from simple applications of L'Hospital's rule.
\end{proof}

\begin{lemma}
\label{lemma:compare2}
Let $M_n=n^b$ and $\alpha_n=n^{\omega}$, where $b>\omega$. Then $\frac{M_nB_{M_n}}{B_n}\rightarrow 0$ 
if $M_n\prec\sqrt{n}$. 
\end{lemma}
\begin{proof}
The proof follows from simple applications of L'Hospital's rule.
\end{proof}

\begin{lemma}
\label{lemma:compare3}
$r_1(n)=\dfrac{\left(1-\frac{1}{M_n}\right)^n 
\left(\frac{\alpha_n+M_n}{\alpha_n}\right)^{M_n}}{\left(\frac{\alpha_n}{\alpha_n+n}\right)^2}\rightarrow 0$,
if $\frac{1}{2}>b>\omega$.
\end{lemma}
\begin{proof}
The proof follows from simple applications of L'Hospital's rule.
\end{proof}

Hence, combining the results of Lemma \ref{lemma:compare1} to Lemma \ref{lemma:compare3} we conclude that 
$MISE(SB)$ converges to $0$ at a faster rate than $MISE(EW)$, provided that we choose $M_n$ and $\alpha_n$
as required by Lemmas \ref{lemma:compare1}--\ref{lemma:compare3}.


\subsection{{\bf Asymptotic choices of the prior parameters}}
\label{subsec:asymptotic_choice}

Asymptotic choices of $\alpha_n$ and $M_n$ are provided by Lemmas \ref{lemma:compare1}--\ref{lemma:compare3}. 
We also need to choose $\sigma_n$ and $\epsilon_n$ appropriately
for complete prior specifications of the Bayesian frameworks of EW and SB. 
Below we provide choices of these parameters and compare the $MISE$ rates of EW and SB for these choices.

Let $\alpha_n = n^{\omega}$, $0< \omega< 1$, $\sigma_n^2=\frac{1}{n^t}$, $t>0$ and let $\epsilon_n$ 
be chosen so that $\epsilon_n^*=\frac{1}{n^r}$, 
$r>0$. 
Then the order of $MISE(EW)$ becomes
\begin{eqnarray}
\frac{1}{\left(1+n^{1-\omega}\right)^2}+\left((1+n^{1-\omega})e^{-\frac{c^2n^t}{6}}\right)+\frac{1}{n^r}+\frac{1}{n^t}. 
\label{eq:mise_ew_opt1}
\end{eqnarray}
By setting $r$, $t$ and $\omega$ appropriately, we can make the convergence rate of $MISE(EW)$
of the order $n^{\omega-1}$, where $0<\omega<1$. Also we can choose $r$, $t$ and $\omega$ such that the
additional conditions in Lemmas \ref{lemma:compare1}---\ref{lemma:compare3}
are satisfied, so that $MISE(SB)$ will be smaller than $MISE(EW)$. 

With the above choices of $\sigma_n$ and $\epsilon_n$, choice of $\alpha_n$ can also be made by minimizing
the order of $MISE(EW)$ with respect to $\alpha_n$ over $0\leq \alpha_n\leq \infty$, yielding
$\alpha_n=n\left(\dfrac{1}{1-\frac{e^{-\frac{c^2n^t}{12}}}{2^{1/3}}}-1\right)$. 
The order of $MISE(EW)$ in that case becomes 
\begin{eqnarray}
\left(\frac{1}{2^{1/3}}e^{-\frac{c^2n^t}{12}}\right)^2+2^{1/3}e^{-\frac{c^2n^t}{6}}+\frac{1}{n^r}+\frac{1}{n^t}. 
\label{eq:mise_ew_opt2}
\end{eqnarray}
However, note that the optimum order of $MISE(EW)$ given by (\ref{eq:mise_ew_opt2}) is attained 
when $\alpha_n$ is a decreasing function of $n$. This is not the same condition under which we have 
proved better performance of the SB model over
the EW model. So it is of interest to study whether under this new condition also 
the SB model outperforms the EW model. 
Using L'Hospital's rule it can be shown, plugging in 
$\alpha_n=n\left(\dfrac{1}{1-\frac{e^{-\frac{c^2n^t}{12}}}{2^{1/3}}}-1\right)$ 
in Lemmas \ref{lemma:compare1}--\ref{lemma:compare3},
that the corresponding ratios still converge to 0 as $n\rightarrow \infty$.
Hence, the SB model again outperforms EW. 

It is also possible, in principle, to obtain $\alpha_n$ by minimizing the order of $MISE(SB)$. However,
in this case closed form solution does not seem to be available, and numerical methods may be necessary.
In any case, it is clear that SB will outperform EW in terms of $MISE$ even for this choice.


%% file: simulation_study.tex
\section{{\bf Bayesian $MISE$ based comparison between density estimators of SB and EW using simulation study}}
\label{sec:simulation_study}

It is demonstrated in \ctn{Bhattacharya08} with simulation experiments and with three real data sets
that the SB-based default posterior predictive density provides better fit to the observed data than the EW-based
default posterior predictive density, and moreover, the SB model emphatically outperforms EW
in terms of pseudo-Bayes factor.

We now compare the performances of the density estimators of the SB and the EW model with respect to Bayesian $MISE$
under simulation experiments.
We also consider the corresponding density estimators when the scales are different and are also part of the Dirichlet
process prior as in the original papers of SB and EW. 

It is to be noted that since the classical kernel density estimators ignore uncertainty about the parameters, 
they are not comparable with the Bayesian density estimators. For the same reason the Bayesian 
$MISE$ is an inappropriate measure for such estimators -- ignoring parameter uncertainty would make 
Bayesian $MISE$ misleadingly small. Hence, we do not consider the kernel density estimators for
formal comparison with the Bayesian density estimators using Bayesian $MISE$. 

Specifically, we compare the following density estimators:
\begin{itemize}
\item[(EW-1):] The Bayesian density estimator
$$\hat f_{EW}(y\mid \Theta_n,\sigma) = \frac{\alpha_n}{\alpha_n+n}A_n(y)
+\frac{1}{\alpha_n+n}\sum_{i=1}^{n}\frac{\varphi(\theta_i,\sigma+\hat k_n)}{(\sigma+\hat k_n)}\phi\left(
\frac{y-\theta_i}{\sigma+\hat k_n}\right)\mathbb I_{\{|y|\leq a\}},$$ provided by (\ref{eq:ew_prior_pred}).
We assume that $\tau=\frac{1}{\sigma^2}$ and since $\sigma$ gives almost point mass to zero for large $n$,
we set $\sigma=0.001$ in our examples. We assume that
$\theta_i\stackrel{iid}{\sim}G$; $i=1,\ldots,n$; $G\sim DP(\alpha_nG_0)$, and under $G_0$,
$\theta_i\sim N\left(\theta_0,\frac{\psi}{\tau}\right)\mathbb I_{\{|\theta_i|\leq a+c\}}$. 
Following \ctn{Bhattacharya08} and \ctn{Escobar95}
se set $\theta_0=5.02$ and $\psi=33.3$. We set $a=1000$ and $c=1000$.
The choice of $\hat k_n$ is elucidated in Section \ref{subsubsec:hat_k_n}.
\item[(SB-1):] The Bayesian density estimator
$$\hat f_{SB}(y\mid \Theta_{M_{n}},\sigma) = \frac{1}{M_{n}}\sum_{i=1}^{M_{n}}
\frac{\varphi(\theta_i,\sigma+\hat k_n)}{(\sigma+\hat k_n)}\phi\left(\frac{y-\theta_i}{\sigma+\hat k_n}\right)
\mathbb I_{\{|y|\leq a\}},$$ provided by (\ref{eq:sb_model}). We assume the same choices as in EW-1, with
$n$ replaced with $M_n$.
\item[(EW-2):] The Bayesian density estimator
$$\tilde f_{EW}(y\mid \Theta_n) = \frac{\alpha_n}{\alpha_n+n}
\int\frac{\varphi(\theta,\sigma)}{\sigma}\phi\left(\frac{y-\theta}{\sigma}\right)dG_0(\theta,\sigma)
+\frac{1}{\alpha_n+n}\sum_{i=1}^{n}\frac{\varphi(\theta_i,\sigma_i)}{\sigma_i}
\phi\left( \frac{y-\theta_i}{\sigma_i}\right)\mathbb I_{\{|y|\leq a\}},$$ 
where we now assume that $(\theta_i,\sigma_i)\stackrel{iid}{\sim}G$; $i=1,\ldots,n$; $G\sim DP(\alpha_nG_0)$.
In the above, $\Theta_n=\left\{(\theta_i,\sigma_i):i=1,\ldots,n\right\}$.
We set $\tau_i=\frac{1}{\sigma^2_i}$ and assume that under $G_0$, 
$\tau_i\sim Gamma(\frac{s}{2},\frac{S}{2})$ and $[\theta_i|\tau_i]\sim 
N\left(\theta_0,\frac{\psi}{\tau_i}\right)\mathbb I_{\{|\theta_i|\leq a+c\}}$.
We assume the same choices of the parameters common with EW-1, and for $s$ and $S$, we set $s=4$ and 
$S=2\times\left(\frac{0.2}{0.573}\right)$ following \ctn{Bhattacharya08}.
\item[(SB-2)] The Bayesian density estimator
$$\tilde f_{SB}(y\mid \Theta_{M_{n}}) = \frac{1}{M_{n}}\sum_{i=1}^{M_{n}}
\frac{\varphi(\theta_i,\sigma_i)}{\sigma_i}\phi\left(\frac{y-\theta_i}{\sigma_i}\right)
\mathbb I_{\{|y|\leq a\}},$$ where $(\theta_i,\sigma_i)\stackrel{iid}{\sim}G$; $i=1,\ldots,M_n$; $G\sim DP(\alpha_nG_0)$,
and $\Theta_{M_n}=\left\{(\theta_i,\sigma_i):i=1,\ldots,M_n\right\}$.
As in the case of EW-2, we set $\tau_i=\frac{1}{\sigma^2_i}$ and assume that under $G_0$,
$\tau_i\sim Gamma(\frac{s}{2},\frac{S}{2})$ and 
$[\theta_i|\tau_i]\sim N\left(\theta_0,\frac{\psi}{\tau_i}\right)\mathbb I_{\{|\theta_i|\leq a+c\}}$.
We assume the same choices as in EW-2, with $n$ replaced with $M_n$.
\end{itemize}

\subsection{{\bf Simulation design and methods}}
\label{subsec:simulation_design_methods}

For our simulation experiments, we draw $100$ datasets each consisting of $n=100$ data points from 
several different distributions and compare the Bayesian $MISE$ of the competing density estimators in each case,
for each of the $100$ datasets. We also compare the expected Bayesian $MISE$ of the density estimators 
obtained by averaging the Bayesian $MISE$ values over the $100$ datasets. 
For each of the $100$ datasets we approximate the Bayesian $MISE$ of the Bayesian density estimators 
by averaging over the output of standard Gibbs sampling from the posterior distributions and direct Monte Carlo 
draws from the true distribution. For the Gibbs sampling we discard the first $3\times 10^5$ iterations as burn-in
and store one sample in $150$ iterations in the next $15\times 10^5$ iterations to store $10000$ Gibbs samples
for inferential purpose. In a VMWare machine with 2.8 CPU GHz and 200 TB memory, 
SB-1, SB-2, EW-1 and EW-2 take, respectively, about 1 minute, 8 minutes,
23 minutes and 38 minutes on the average, for a single run of the aforementioned Gibbs sampling.

\subsubsection{{\bf Choices of true, data-generating densities}}
\label{subsubsec:true_densities}

We consider as true density $f_0$ the following distributions: $N(\mu_0,\sigma^2_0)$ with $(\mu_0=0,\sigma^2_0=200)$, 
$Cauchy(\mu_0,\sigma^2_0)$ with $(\mu_0=0,\sigma^2=5)$, Student's $t_{df;\mu_0,\sigma^2_0}$, with $df=5$ degrees of freedom, 
$\mu_0=0$ and $\sigma^2_0=5$, $Gamma(\alpha_0,\beta_0)$ with $\alpha_0=2$, $\beta_0=0.5$ (mean $\alpha_0/\beta_0$ and variance
$\alpha_0/\beta^2_0$) and $Beta(\alpha_0,\beta_0)$ with $\alpha_0=4$ and $\beta_0=2$. 
Note that the $Gamma$ and the $Beta$ densities are right skewed and left skewed, respectively, while the others are symmetric 
around zero with different scales. The $5$ degrees of freedom of the Student's $t$ distribution prevents the
underlying $t$ distribution from being close to either normal or Cauchy. In other words, we evaluate
performances of the competing density estimators for a reasonably wide range of true distributions, including the Cauchy
distribution whose moments do not exist. 

It is important to observe that among the above choices of the true densities, only 
the $N(0,200)$ density is approximately of the form (\ref{eq:true_density_1}) (this holds approximately since
our choice $a=1000$ is sufficiently large). Thus, as per our theoretical results, 
although in this case the Bayesian density estimators SB-1 and EW-1 are expected to perform reasonably well, and in particular,
SB-1 is expected to outperform EW-1, there is no theoretical basis to expect good performance of SB-1 and EW-1
for the other choices of the non-normal true densities, given that only a single bandwidth is considered
by SB-1 and EW-1. On the other hand, since SB-2 and EW-2 make use of multiple bandwidths, better performances
may be expected of these density estimators in practice. 
However, our aim is to investigate, using simulation studies, if SB-1 and EW-1, 
in particular SB-1, can yield reasonable performance for wider range of densities in practice 
than that considered for our asymptotic theory. 
For all the density estimators and for the true normal and Cauchy densities we choose $a=1000$. 
For the other true densities we do not enforce any
truncation constraint.

\subsubsection{{\bf Choices of $\alpha_n$ and $M_n$}}
\label{subsubsec:alpha_M}

Following Lemma \ref{lemma:compare3} in Section \ref{sec:comparison_sb_ew}, we set $\alpha_n=n^{\omega}$,
$M_n=n^b$, with $\omega<b<1/2$. In particular, we set $\omega=0.3$ and $b=0.4$, so that $n^{\omega}\approx 4$
and $n^b\approx 7$. Hence, we set $\alpha_n=4$ and $M_n=7$.

\subsubsection{{\bf Consistent estimation of $k$}}
\label{subsubsec:hat_k_n}
To estimate $k$ consistently we first observe that in (\ref{eq:true_density_1}), for $a=\infty$, 
$k$ is nothing but $E\left[Var(Y|\theta)\right]$, 
where both the expectation and the variance parts are with respect to the true model $f_0$. Now observe that
$E\left[Var(Y|\theta)\right]$ can be viewed as the average of the cluster-wise variances of $Y$,
$\theta$ being the cluster centers, since under $f_0$, $[Y|\theta]$ is approximately
normal with mean $\theta$ and variance $k^2$. In other words, assuming that $\theta_j\stackrel{iid}{\sim}F_0$;
$j=1,\ldots,m$ and $[Y_{ij}|\theta_j]\stackrel{iid}{\sim}N\left(\theta_j,k^2\right)$; $i=1,\ldots,n$; $j=1,\ldots,m$, 
a clustering of $\{Y_{ij};i=1,\ldots,n;j=1,\ldots,m\}$ would reveal $m$ clusters, with each empirically estimated
within-cluster variance being a consistent estimator of $k^2$, as $n\rightarrow\infty$. The average of the 
within-cluster variances is thus a consistent estimator of $E\left[Var(Y|\theta)\right]$, as $n\rightarrow\infty$, for $m\geq 1$.
For finite $a>0$, $E\left[Var(Y|\theta_j)\right]$ is given by 
$k^2\left[1-\frac{\beta_j\phi(\beta_j)-\alpha_j\phi(\alpha_j)}{\Phi(\beta_j)-\Phi(\alpha_j)}-
\frac{\phi(\beta_j)-\phi(\alpha_j)}{\Phi(\beta_j)-\Phi(\alpha_j)}^2\right]$, where $\alpha_j=\frac{-a-\theta_j}{k}$
and $\beta_j=\frac{a-\theta_j}{k}$. Also, $E(Y|\theta_j)=\theta_j
-k\left[\frac{\phi(\beta_j)-\phi(\alpha_j)}{\Phi(\beta_j)-\Phi(\alpha_j)}\right]$, when $a$ is finite. 
From the expressions of $E(Y|\theta_j)$ and $Var(Y|\theta_j)$, $\hat\theta_j$ and $k^2$ can be estimated
consistently by substituting the cluster-wise mean and variance in the place of $E(Y|\theta_j)$ and $Var(Y|\theta_j)$
respectively, and solving for $\theta_j$ and $k^2$. 
Since $|Y_{ij}|<a$ for $i=1,\ldots,n$; $j=1,\ldots,m$, for all $n$ and $m$, it follows that $\hat k_n$ is uniformly bounded,
so that $(\hat k_n-k)^2$ is uniformly integrable, as required for the development of our asymptotic theory.

To obtain $\hat k_n$ in our examples, since $a=1000$ is adequately large, we first ignore the truncation 
constraint for simplicity. Then we cluster the observed data of size $n=100$ into $K=2,3,4,5$ 
clusters using the K-means algorithm and for each value of $K$ we compute the variance of the within-cluster variances. 
We then select that $K=\hat K$ for which the variance of the within cluster variances is minimized, which ensures that 
the cluster variances are approximately the same. We then set $\hat k^2_n$ to be the average
of the within cluster variances associated with $K=\hat K$. 
Note that we restrict the number of clusters to at most $5$ so that each cluster may have a reasonable number
of observations to reliably compute the cluster-wise variances.

\subsection{{\bf Results of our simulation experiments}}
\label{subsec:simulation_results}

Table \ref{table:MISE} shows the expected Bayesian $MISE$ values for the competing Bayesian density 
estimators for five different choices of the true distribution, with various degrees of skewness and different 
scales and variabilities. Observe that both SB-1 and EW-1 perform much better than either of SB-2 and EW-2
when the true density is normal, and also expectedly yield the best performances among all the choices of the true densities. 
Observe that only except in the case of the $Gamma$ density, SB-1 outperforms EW-1 for all the other choices of the true density.
In fact, SB-1 beats both EW-2 and SB-2 not only when the true density is normal, 
but also when the true density is Cauchy.
On the other hand, SB-2 outperforms EW-2 when the true densities are normal and $Beta$, and in the rest of the
cases, EW-2 performs somewhat better.

Figure \ref{fig:mise_comparison} depicts the comparisons between SB-1, EW-1, SB-2 and EW-2 for the five true densities
in terms of the data-wise Bayesian $MISE$'s. Panel (a) of the figure shows that when $f_0$ is the normal density,
SB-1 is almost uniformly better than the rest in terms of Bayesian $MISE$. When the true density is Cauchy, panel (b) 
shows that SB-1 and EW-1 perform almost equivalently, and have smaller variance with respect to the data compared to the rest. 
Indeed, for most of the data sets, SB-1 performs the best, resulting in the overall best performance among all the
density estimators. Panel (c) shows that when the true density if $t$, both SB-2 and EW-2 are close to each other
in terms of Bayesian $MISE$ for most of the data sets and perform better than
SB-1 and EW-1; SB-1 performs almost uniformly better than EW-1. For the Gamma distribution, panel (d) shows that 
EW-2 performs better than SB-2, which uniformly dominates both EW-1 and SB-1 in terms of Bayesian $MISE$. Also, the variability
of Bayesian $MISE$ for EW-1 is significantly higher than those of the remaining density estimators. Panel (e) shows
that both SB-1 and SB-2 significantly outperform both EW-1 and EW-2, and that SB-2 performs the best for most of the
data sets. 

Figure \ref{fig:density_comparison} shows the posterior predictive density estimates for sample data sets from the
five true densities. It is generally observed in panels (a) -- (e) that EW-2 and SB-2 are more inclined towards
capturing the observed histogram rather than the true density. This makes these density estimators less smooth than desired.
In contrast, SB-1 and EW-1 are smoother, and seem to be better suited for density estimation whenever
the true density is of the form (\ref{eq:true_density_1}). For the details, note that when the normal density is true,
SB-1 and EW-1 visually outperform SB-2 and EW-2, which is consistent with the Bayesian $MISE$ results; moreover,
SB-2 clearly beats EW-2. For the Cauchy density, panel (b) shows that SB-1 performs the best, whereas EW-1
seem to deviate the most from the true density, which is again consistent with our Bayesian $MISE$ results.
When the true density is $t$ or $Gamma$, as depicted in panels (c) and (d), both SB-1 and EW-1 seem to be 
inadequate for estimating
the true density or the observed histogram. In comparison, SB-2 and EW-2, although not sufficiently smooth
for estimating the true density adequately, perform better that SB-1 and EW-1. As before, these are reflected
in our numerical results on Bayesian $MISE$. Panel (e) shows that in the case of the $Beta$ density, EW-2 performs
the worst, while not much difference is visually revealed among the other density estimators. However, our 
Bayesian $MISE$ computation shows that SB-2 performs the best in this case, closely followed by SB-1, while
EW-1 and EW-2 perform significantly poorly in comparison. 

Overall it may be argued that the methods based on SB emphatically outperform those based on EW
particularly when the implementation time is also taken in consideration. Indeed, as reported
in Section \ref{subsec:simulation_design_methods}, SB-1 and SB-2 take significantly less computational time
compared to EW-1 and EW-2.

\begin{table}
\caption{Expected Bayesian $MISE$ comparison between density estimators}
\label{table:MISE}
\begin{center}
\begin{tabular}{|c||cc||cc|}\hline
True distribution & EW-1 & SB-1 & EW-2 & SB-2\\
\hline
$N(0,200)\mathbb I_{\left\{|y|\leq 1000\right\}}$ & 0.000092 & 0.000010 & 0.000082 & 0.000061\\
$Cauchy(0,5)\mathbb I_{\left\{|y|\leq 1000\right\}}$ & 0.000493 & 0.000211 & 0.000278 & 0.000245\\
$Student's~ t_{5;0,5}$ & 0.001402 & 0.000859 & 0.000338 & 0.000405\\
$Gamma(2,0.5)$ & 0.008891 & 0.011571 & 0.001469 & 0.002146\\
$Beta(4,2)$ & 0.659581 & 0.157082 & 0.216738 & 0.110191\\
\hline
\end{tabular}
\end{center}
\end{table}

\newpage

\begin{figure}
\centering
\subfigure[$f_0(x)\equiv N(0,200)\mathbb I_{\left\{|x|\leq 1000\right\}}$.]{ \label{fig:normal}
\includegraphics[width=6.5cm,height=5.4cm]{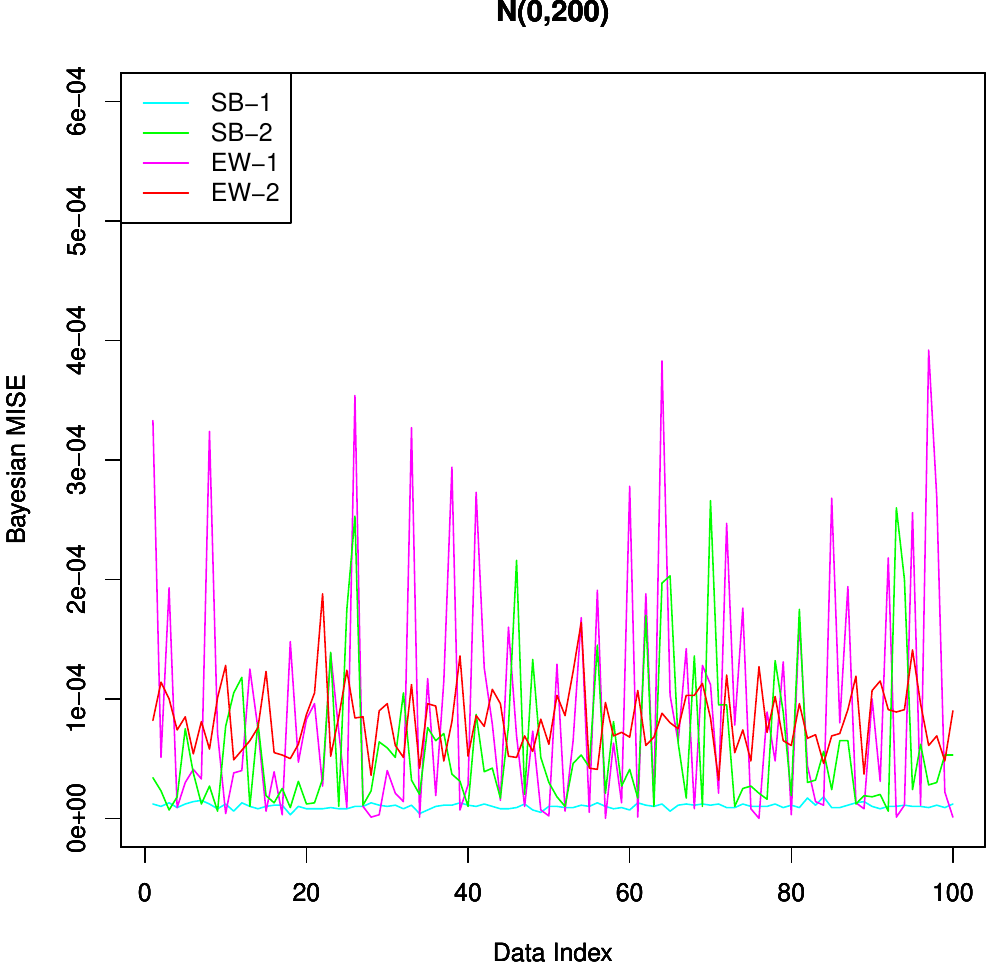}}
\hspace{2mm}
\subfigure[$f_0(x)\equiv Cauchy(0,5)\mathbb I_{\left\{|x|\leq 1000\right\}}$.]{ \label{fig:cauchy}
\includegraphics[width=6.5cm,height=5.3cm]{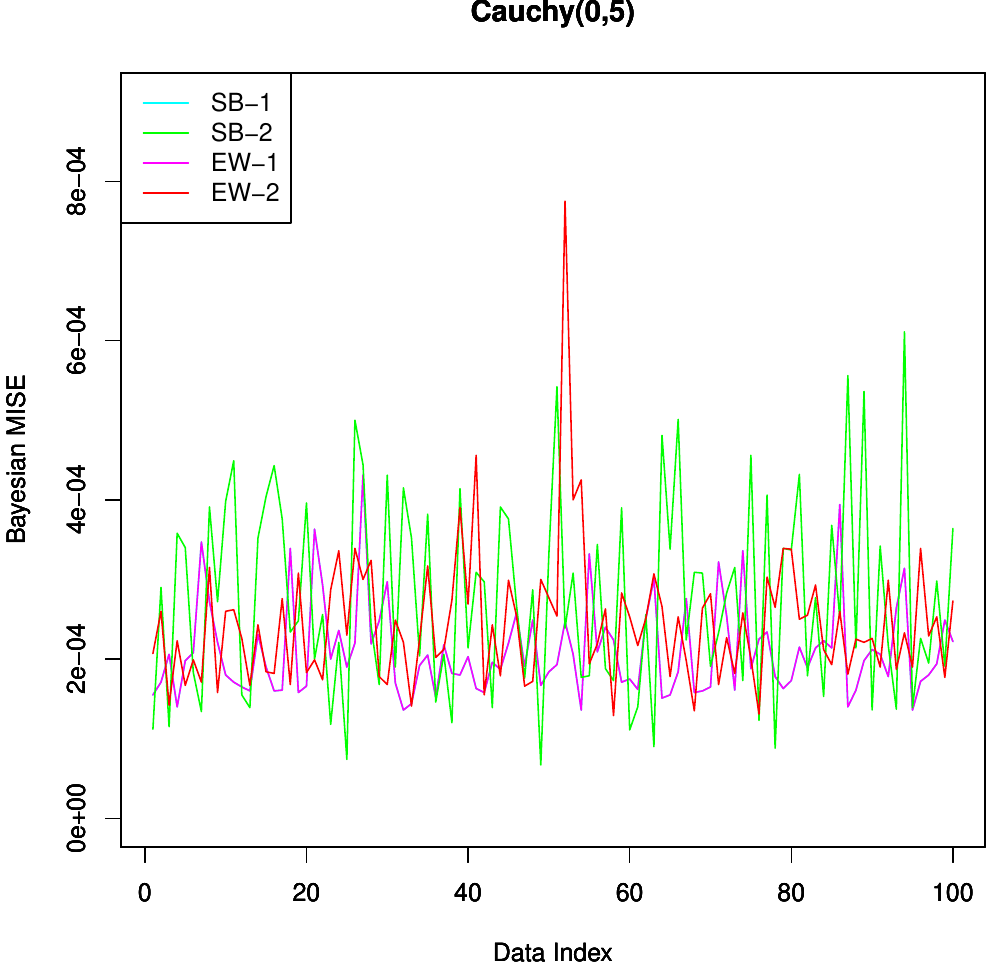}}
\vspace{2mm}\\
\subfigure[$f_0\equiv t(5;0,5)$.]{ \label{fig:t}
\includegraphics[width=6.5cm,height=5.3cm]{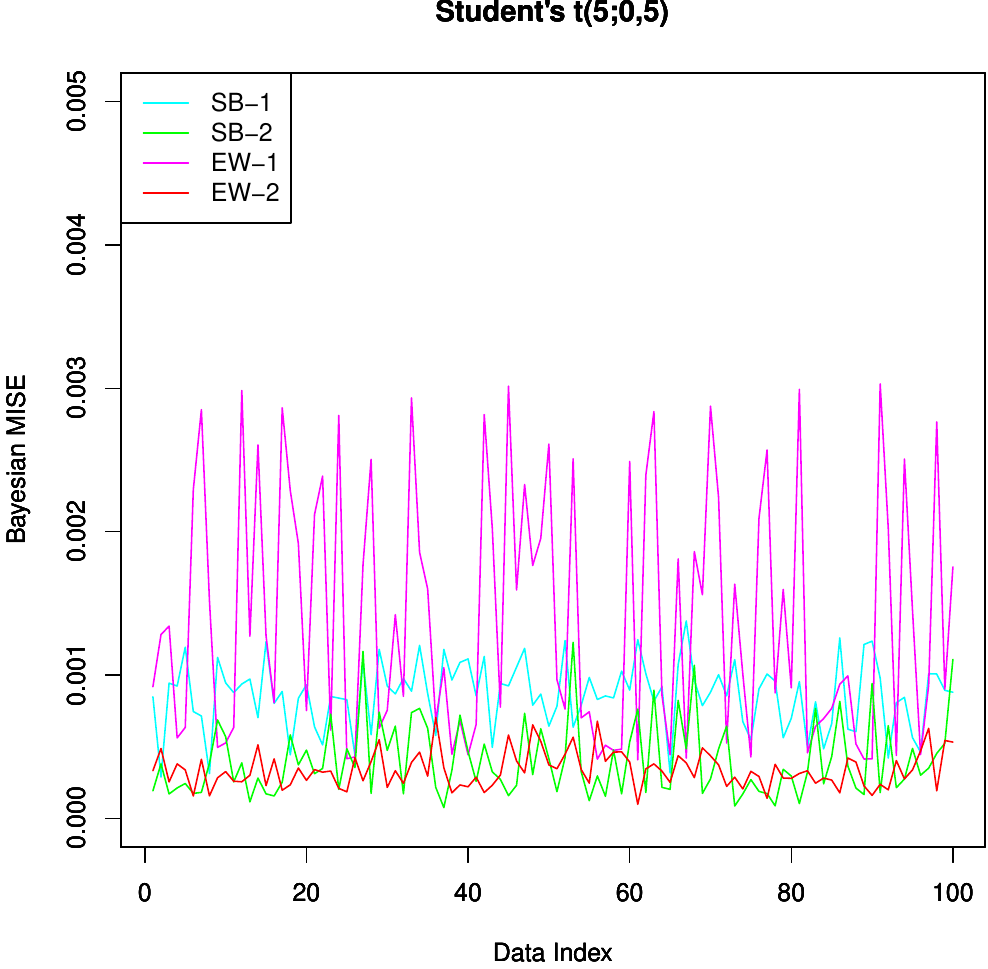}}
\hspace{2mm}
\subfigure[$f_0\equiv Gamma(2,0.5)$.]{ \label{fig:gamma}
\includegraphics[width=6.5cm,height=5.3cm]{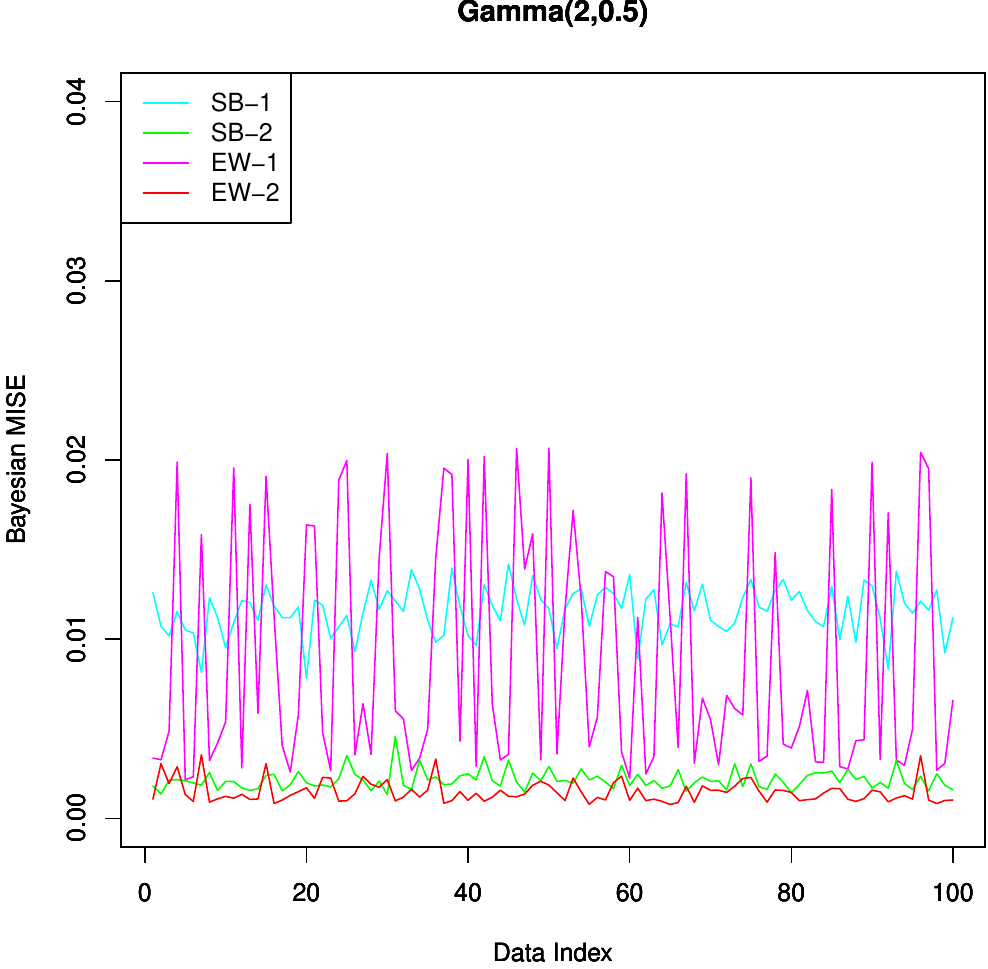}}
\vspace{2mm}\\
\subfigure[$f_0\equiv Beta(4,2)$.]{ \label{fig:beta}
\includegraphics[width=6.5cm,height=5.3cm]{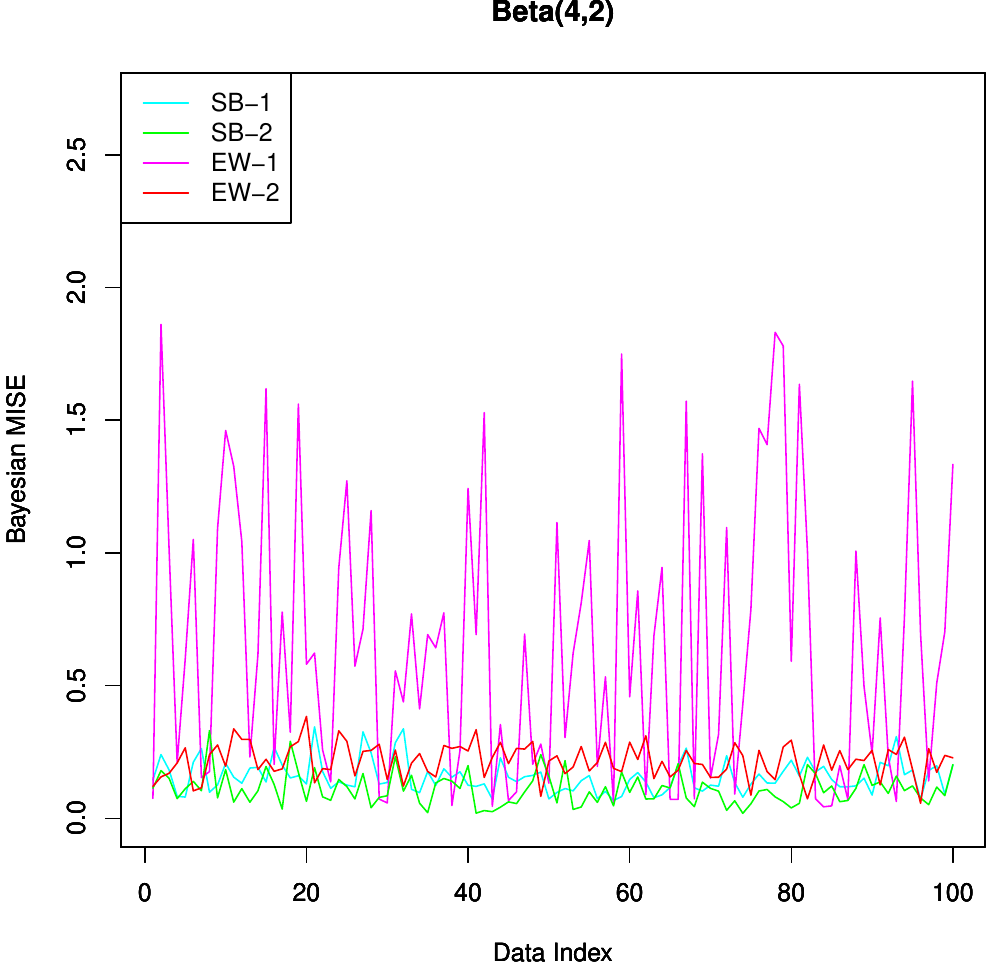}}
\caption{{\bf Data-wise Bayesian $MISE$'s for EW and SB based density estimators for various choices of the
true density $f_0$.}} 
\label{fig:mise_comparison}
\end{figure}

\begin{figure}
\centering
\subfigure[$f_0(x)\equiv N(0,200)\mathbb I_{\left\{|x|\leq 1000\right\}}$.]{ \label{fig:compare_normal}
\includegraphics[width=6.5cm,height=5.4cm]{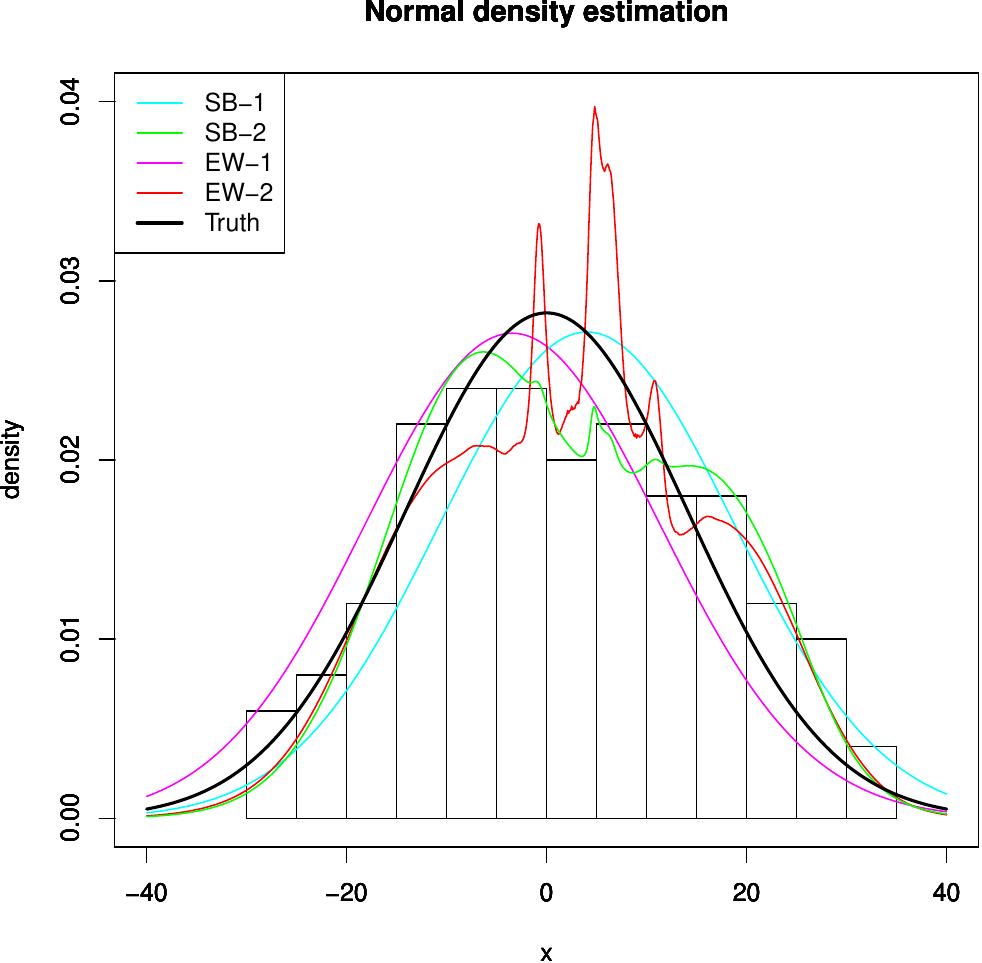}}
\hspace{2mm}
\subfigure[$f_0(x)\equiv Cauchy(0,5)\mathbb I_{\left\{|x|\leq 1000\right\}}$.]{ \label{fig:compare_cauchy}
\includegraphics[width=6.5cm,height=5.3cm]{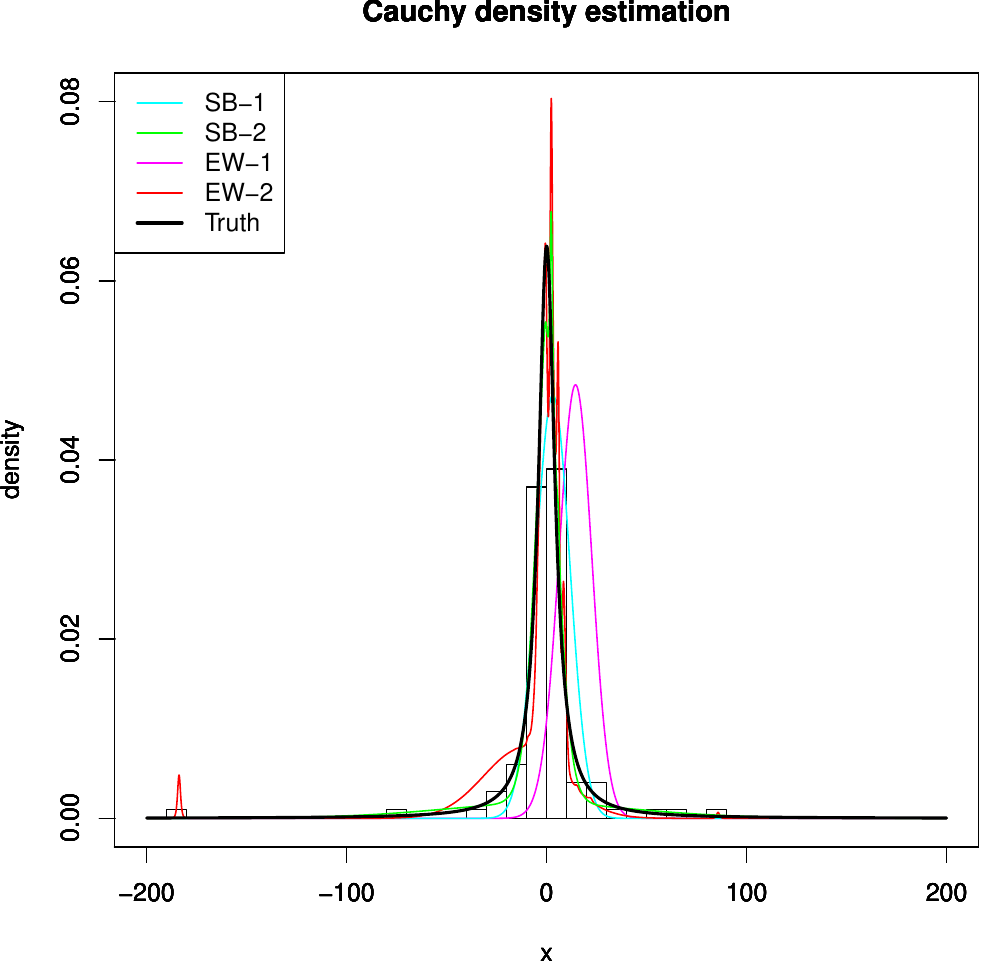}}
\vspace{2mm}\\
\subfigure[$f_0\equiv t(5;0,5)$.]{ \label{fig:compare_t}
\includegraphics[width=6.5cm,height=5.3cm]{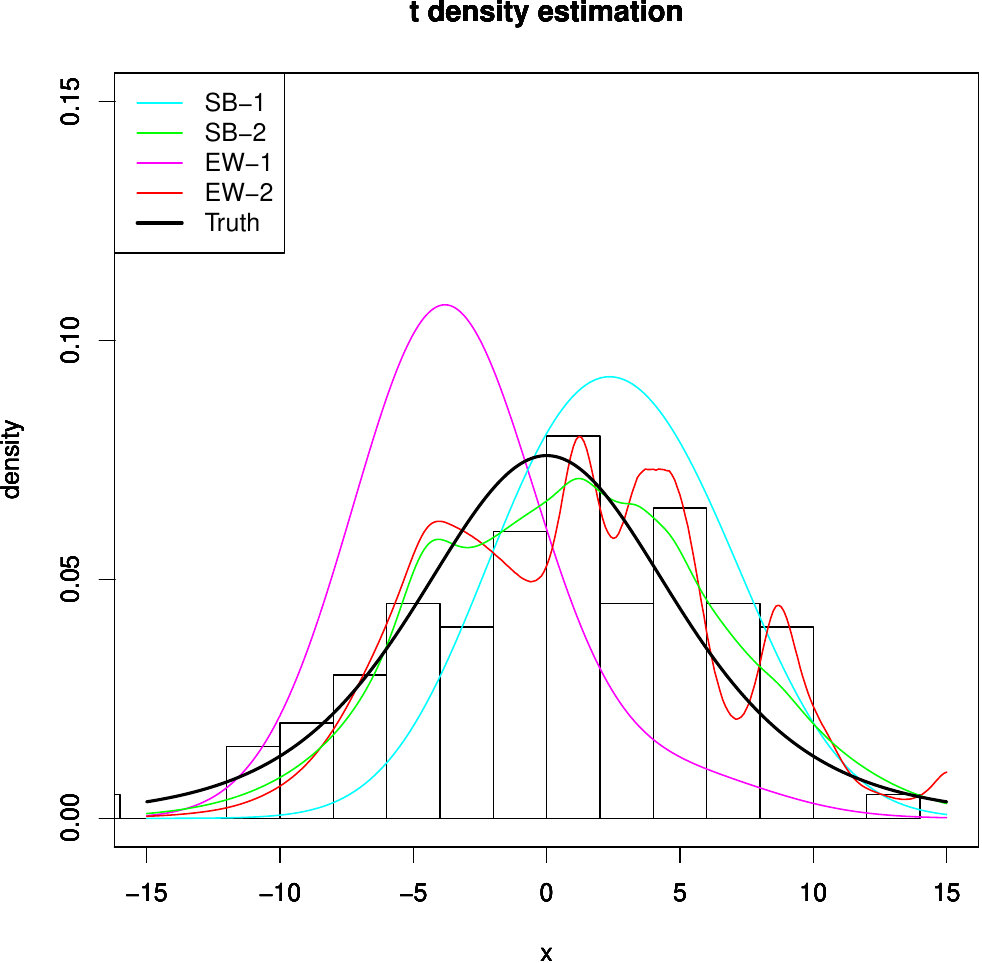}}
\hspace{2mm}
\subfigure[$f_0\equiv Gamma(2,0.5)$.]{ \label{fig:compare_gamma}
\includegraphics[width=6.5cm,height=5.3cm]{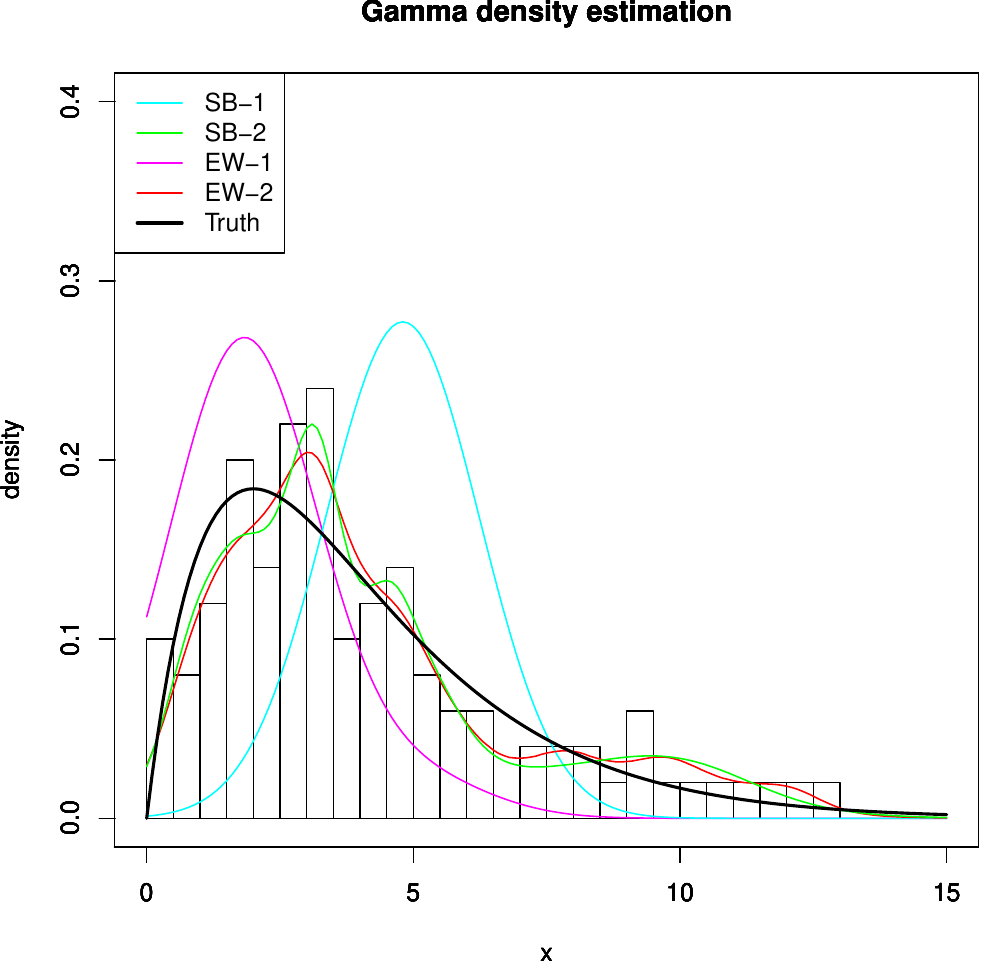}}
\vspace{2mm}\\
\subfigure[$f_0\equiv Beta(4,2)$.]{ \label{fig:compare_beta}
\includegraphics[width=6.5cm,height=5.3cm]{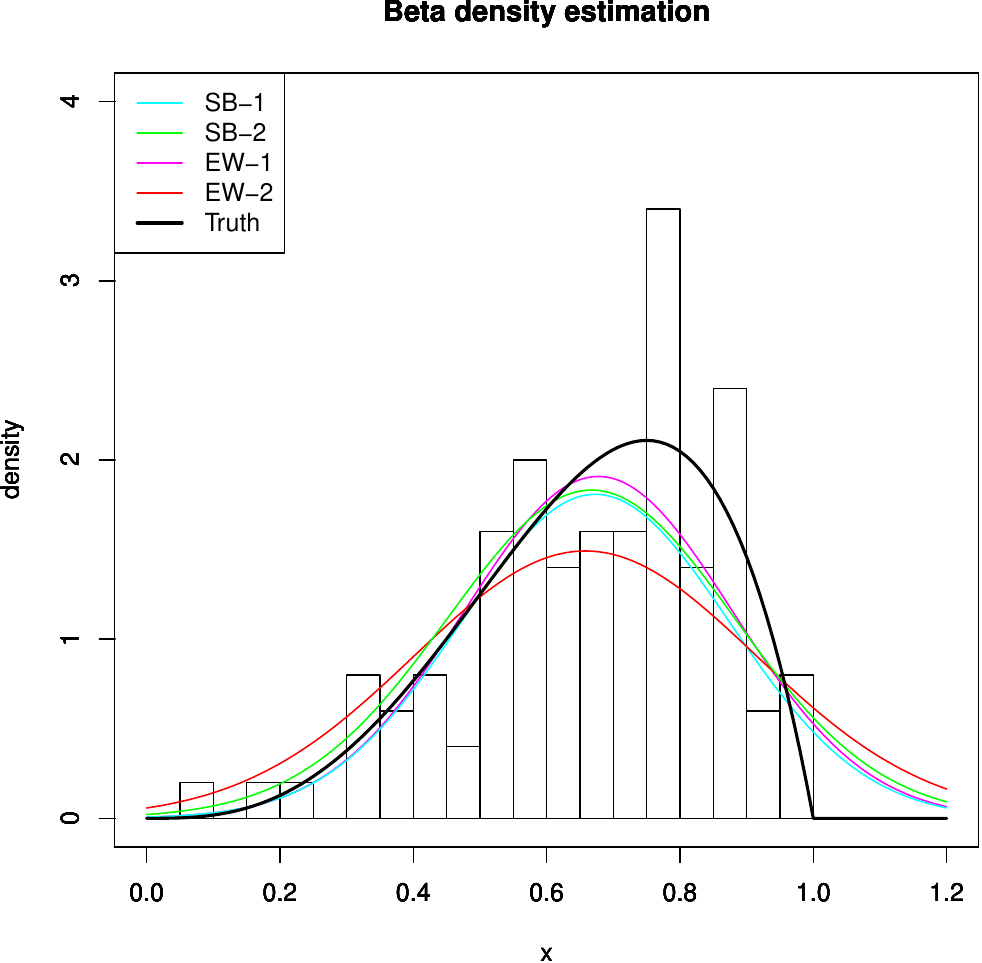}}
\caption{{\bf Visual comparison between the posterior predictive Bayesian density estimators for EW and SB for 
various choices of the true density $f_0$.}} 
\label{fig:density_comparison}
\end{figure}

%% file: param.tex
\section{{\bf Convergence to wrong model}}
\label{sec:param_model}

We now investigate conditions under which the models of EW and SB converge to wrong models, that is, to models which
did not generate the data. It is perhaps easy to anticipate that if $\alpha_n$ is made to grow at a rate faster than $n$, then
the density estimator of EW would converge to the convolution of the kernel and $G_0$, irrespective of the true, data-generating
model. 
We show in this section that indeed the simple condition of letting $\alpha_n$ grow faster than $n$ is enough to 
derail the EW model. On the other hand, much stronger conditions are necessary to get the SB model
to converge to the wrong model.

\subsection{{\bf EW model}}
\label{subsec:param_ew}

\begin{theorem}
\label{theorem:param_ew}
Suppose that $\alpha_n\succ O(n)$. Then 
\begin{eqnarray}
E^n_0\left[E\left(\hat f_{EW}(y\mid \Theta_n,\sigma)\right)\right]\rightarrow 
\int_{\theta_{n+1}}\frac{\varphi(\theta_{n+1},k)}{k\sqrt{2\pi}}e^{-\frac{\left(y-\theta_{n+1}\right)^2}{2k^2}}
dG_0(\theta_{n+1}).
\label{eq:param_ew}
\end{eqnarray}
\end{theorem}

\begin{proof}
See Section S-4.1 of the supplement.
\end{proof}

Thus, the condition $\alpha_n\succ O(n)$ is enough to mislead the EW model, taking it to a 
wrong model. An intuitive reason for this negative result is that as $\alpha_n$ increases at a rate
faster than $n$, then in (\ref{eq:ew_prior_pred}), the weight $\frac{\alpha_n}{\alpha_n+n}$ 
of the first term $A_n$ tends to 1 as $n\rightarrow\infty$. Since $A_n$ is the term associated
with $G_0$, and since this term gets full weight as $n\rightarrow\infty$, it is not unexpected
that convergence to the true model can not be achieved in this case.

It may be interesting to compare the issue involved with this inconsistency problem with the famous example 
of \ctn{Diaconis86a}, \ctn{Diaconis86b}. 
The latter's example concern estimation of an unknown 
location parameter $\theta$, 
with the following set-up: for $i=1,\ldots,n$, $x_i=\theta+\epsilon_i$; $\epsilon_i\stackrel{iid}{\sim}F$;
$F\sim DP(\alpha G_0)$. Assuming that $x_1,\ldots,x_n$ are all distinct, and that $G_0$ has density $g_0$ 
it follows (see \ctn{Korwar73}) that the likelihood of the location parameter is given, almost surely, by
$\prod_{i=1}^ng_0(x_i-\theta)$. Now, if the data-generating distribution is very different from $g_0$,
then it is not surprising that the posterior mean of $\theta$ may be inconsistent. Indeed, \ctn{Diaconis86a},
\ctn{Diaconis86b} chose the data-generating distribution to be particularly incompatible with the
base measure $G_0$.
In other words, the example of \ctn{Diaconis86a}, \ctn{Diaconis86b} relies upon the ``no-ties" assumption
within the observed data, so that the parametric likelihood, based on $g_0$, is obtained. In contrast,
we investigate the situation when the precision parameter $\alpha_n$ increases with the sample size $n$
at a fast rate, but without ruling out the possibility of ties within the data. 

In the next section we will observe that, unlike in EW, the condition $\alpha_n\succ O(n)$ does not alone 
guarantee inconsistency in the SB case.

\subsection{{\bf SB model}}
\label{subsec:param_sb}

If there are more ``distinct" mixture components in the SB model than
sample size $n$, then at least one component will be empty. If this phenomenon persists for large $n$
then the posterior probabilities of the empty components, being the same as their
prior probailities, can not tend to zero as $n\rightarrow\infty$. 
More distinct components than $n$ can be enforced by setting $M_n>n$ and letting 
$\alpha_n$ increase at a rate much faster than $M_n$.
Since marginally the empty components
arise from $G_0$, in this case the 
model will converge to $\int\frac{\varphi(\theta,k)}{k}\phi(\frac{y-\theta}{k})dG_0(\theta)$.    

The above issue should not be confused with the results of \ctn{Rousseau11}
who show that if the fitted mixture model contains more components than the true model, the latter 
also assumed to be a mixture
of the same form, then the extra components of the fitted model asymptotically die out under suitable
prior assumptions. Indeed, although \ctn{Rousseau11} assumed that the number of components in the fitted
mixture is more than that in the true mixture, unlike us they did not let the former grow with the sample
size $n$. Also, our assumptions regarding the true density is much more general than the finite mixture
assumption of \ctn{Rousseau11}.

Now we elaborate the issue of non-convergence of the SB model in more detail.
Recall that no additional condition on $\alpha_n$ is necessary to ensure asymptotic 
convergence of $MISE(SB)$ given by
(\ref{eq:mise_order_sb}); it is only necessary to ensure that $M_n\prec O(\sqrt{n})$ and that
$\sigma_n$ and $\epsilon_n$ are appropriately chosen. 
So, unlike in the case of EW, even if $\alpha_n\succ O(M_n)$ or $\alpha_n\succ O(n)$, 
the SB model can still converge to the 
true distribution by setting $M_n\prec O(\sqrt{n})$. 

However, for $M_n\succ O(n)$, the order
of the posterior probability
$P(Z\in R_1^*, \Theta_{M_n}\in E, \sigma\leq \sigma_n\vert \bY_n)$, given by
$\left(1-\frac{1}{M_n}\right)^n\left(\frac{\alpha_n+M_n}{\alpha_n}\right)^{M_n}$, does not converge to 0. 
Since the probability that a component will remain empty is $P(Z\in R_1^*\vert \bY_n)$, which
is bounded below by $P(Z\in R_1^*, \Theta_{M_n}\in E, \sigma\leq \sigma_n\vert \bY_n)$, with
probability tending to 1 as $n\rightarrow\infty$ a component will remain empty if the latter
probability tends to 1.
In order to investigate conditions which ensure this, we make use of Lemma S-4.1, formally stated and proved in
Section S-4.2 of the supplement. Assuming that $C_n=O\left(\frac{1}{r^s_n n^2}\right)$; $s>2$  
(the implication of which is elucidated in Remark 1 below), the lemma gives an asymptotic lower bound 
of the form $\left(1-\frac{1}{M_n}\right)^n\left(\frac{\alpha_n+M_n}{\alpha_n}\right)^{M_n}$
for the posterior probability
$P(Z\in R_1^*, \Theta_{M_n}\in E, \sigma\leq \sigma_n\vert \bY_n)$.

Using L' Hospital's rule it is easy to check that for $M_n=n^b$, $\alpha_n=n^{\omega}$, 
$\omega>1$, $b>1$, such that $\omega>2b$,
$\left(\frac{\alpha_n}{\alpha_n+M_n}\right)^{M_n} \left(1-\frac{1}{M_n}\right)^{n}\rightarrow 1$. 
Thus for $M_n> n$, $\alpha_n> n$, the posterior probability that a mixture
component remains empty converges to 1 as $n\rightarrow\infty$. 
Since $P(Z\in R_1^*, \Theta_{M_n}\in E, \sigma\leq \sigma_n\vert\bY_n)$ converges to 1
as $n\rightarrow\infty$, the factor associated with this posterior probability is the only contributing term  
in $E\left(\frac{\varphi(\theta_i,\sigma+\hat k_n)}{\left(\sigma+\hat k_n\right)}
\phi\left(\frac{y-\theta_i}{\sigma+\hat k_n}\right)\bigg\vert \bY_n\right)$, 
for large $n$. 
Formally, we have the following theorem.

\begin{theorem}
\label{theorem:converge_param}
Assume the conditions of Lemma S-4.1. 
Further assume that $M=n^b$, $\alpha_n=n^{\omega}$, 
$\omega>1$, $b>1$, such that $\omega>2b$.
Then, for the SB model it holds that
\begin{equation}
E^n_0\left[E\left(\hat f_{SB}(y\mid\Theta_{M_n},\sigma)\bigg\vert \bY_n\right)\right]\rightarrow 
\int_{\theta_{i}}\frac{\varphi(\theta_i,k)}{k\sqrt{2\pi}}e^{-\frac{\left(y-\theta_{i}\right)^2}{2k^2}}dG_0(\theta_{i}).
\label{eq:sb_wrong}
\end{equation}
\end{theorem}

\begin{proof}
See Section S-4.2.1 of the supplement.
\end{proof}

Thus, while the EW model can converge to the wrong model if only $\alpha_n\succ O(n)$ is assumed, much
stronger restrictions
are necessary to get the SB model to deviate from the true model.
%
%
The following remarks may be noted.
\\[2mm]
{\it Remark 1:} Note that Lemma S-4.1 requires $C_n=O\left(\frac{1}{r^s_n n^2}\right)$; $s>2$, 
such that $\frac{1}{r^{s-1}_nn^2}\rightarrow\infty$ and $\frac{1}{r^{s-2}_n}\rightarrow\infty$.
These conditions are satisfied, for example, for $r_n=\frac{1}{n^t\tilde a^{\frac{1}{s}}}$,
with $0<t<1$, $s=\frac{3}{t}$ and $\frac{a}{2}<\tilde a<4a$. These choices also ensure that $\frac{a}{2}<\frac{C_n}{n}<4a$.
\\[2mm]
{\it Remark 2:}  
However, all choices of $r_n$ such that $\frac{1}{r^{s-1}_nn^2}\rightarrow\infty$ and $\frac{1}{r^{s-2}_n}\rightarrow\infty$
need not ensure that $\frac{a}{2}<\frac{C_n}{n}<4a$. For instance, the same choices as above except that $s>\frac{3}{t}$
satisfies these the first two requirements but $\frac{C_n}{n}\rightarrow\infty$ as $n\rightarrow\infty$.
This also contradicts the compact support assumption of the true distribution.
\\[2mm]
{\it Remark 3:} Remark 2 shows that even if $\alpha_n\succ O(n)$ and $M_n\succ O(n)$, the SB model 
does not necesarily converge to the wrong model, whereas EW converges to the wrong model just if
$\alpha_n\succ O(n)$. In this regard as well, the SB model appears to be superior compared to the EW model.   
%
%
%

%% file: pi_model.tex
\section{{\bf Modified SB Model}}
\label{sec:modified_sb_model}

A slightly modified version of SB model is as follows:
\begin{equation}
\hat f^*_{SB}(y\mid \Theta_{M}, \Pi,\sigma) = \sum_{i=1}^{M}\pi_i\frac{\varphi(\theta_i,\sigma+k)}{(\sigma+k)}
\phi\left(\frac{y-\theta_i}{\sigma+k}\right),
\label{eq:sb_model_modified}
\end{equation} 
where $\sum_{i=1}^{M} \pi_i=1$. We assume that
$\Pi=(\pi_1, \ldots, \pi_M)\sim Dirichlet(\beta_1, \ldots, \beta_M)$, $\beta_i> 0$ and is independent of $\Theta_M$ and $\sigma$.
The assumptions of Dirichlet process prior on $\Theta_M$ and the prior structure of $\sigma$ remain same as before.
The previous form of the SB model (\ref{eq:sb_model}) is a special case (discrete version) of this model with $\pi_i=\frac{1}{M}$ for each $i$.

Due to discreteness of the Dirichlet process prior, the parameters $\theta_{i}$ are coincident with positive probability.
As a result, (\ref{eq:sb_model_modified}) reduces to the form
\begin{equation}
\hat f^*_{SB}(y\mid \Theta_{M},\Pi,\sigma) = \sum_{i=1}^{M^*}p_i\frac{\varphi(\theta^*_i,\sigma+k)}{(\sigma+k)}
\phi\left(\frac{y-\theta^*_i}{\sigma+k}\right),
\label{eq:mult_normixture2}
\end{equation}
where $\left\{\theta^*_1,\ldots,\theta^*_{M^*}\right\}$ are $M^*$ distinct components in $\bTheta_M$ with $\theta^*_i$ occuring $M_i$ times, 
and $p_i=\sum_{j=1}^{M_i}\pi_j$. 
In contrast to the previous form of the SB model (\ref{eq:sb_model}) where the mixing probabilities
are of the form $M_i/M$, here the mixing probabilities $p_i$ are continuous.

The asymptotic calculations associated with the modified SB model are almost the same as in the case of the 
SB model in Section \ref{subsec:sb_model}. 
Indeed, this modified version of SB's model converges to the same
distribution where the EW model and the previous version of the SB model also converge.
Moreover, the order of $MISE$ for this model remains exactly the same as that of the
previous version of the SB model. In Section S-5 of the supplement we 
provide a brief overview of the steps involved in the asymptotic calculations.

%% file: supp.tex
\begin{center}
{\bf \LARGE Supplementary Material} 
\end{center}

Throughout, we refer to our main manuscript as MB.

\section{{\bf Proofs of results associated with Section 4 of MB}}

\subsection{{\bf Proof of the result presented in Section 4.2 of MB}}

\begin{lemma}
\label{lemma:lemma_sigma_n}
Under the data generating true density $f_0$,
$\frac{1}{n}\sum_{j=1}^{M_n}n_j\left(\bar Y_j - \bar Y\right)^2\rightarrow 0$, a.s.
if $1<M_n \prec O(n)$ (for any two sequences $a^{(1)}_n$ and $a^{(2)}_n$ we say $a^{(1)}_n\prec a^{(2)}_n$ if
 $\frac{a^{(1)}_n}{ a^{(2)}_n}\rightarrow 0$).
\end{lemma}

\begin{proof}
Let $C_{n}^{'} = \dfrac{\sum_{j=1}^{M_n}n_j\left(\bar Y_j - \bar Y\right)^2}{\sum_{j=1}^{M_n}n_j}$. 
We recall that $\bY_{n}$
form a triangular array, 
and the $n$-th row of that array is summarized by the statistic $C_{n}^{'}$.
Since the random variables of a particular row of that array are independent of the random variables of 
the other rows, 
$C'_{n}$ are independent among themselves. \\

Suppose $\mu_{T}$ is the true population mean and $\sigma_{T}^{2}$ is the true population variance 
(both of which are assumed to be finite). Since all $Y_i$'s are from same true density $f_0$, 
under $f_0$, $E^{f_0}(Y_i) = \mu_{T}$ and 
$V^{f_0}(Y_i) = \sigma_{T}^{2}$.
\begin{eqnarray}
&&E^{f_0}\left(\dfrac{\sum_{j=1}^{M_n}n_j\left(\bar Y_j - \bar Y\right)^2}{\sum_{j=1}^{M_n}n_j}\right) \nonumber\\
&=& E_{z}\left[E_{Y\vert z}^{f_0}\left(\dfrac{\sum_{j=1}^{M_n}n_j\left(\bar Y_j - \bar Y\right)^2}{\sum_{j=1}^{M_n}n_j}\right)\right].
\label{expec_sb_sigma_n}
\end{eqnarray}

Note that
\begin{eqnarray}
&&E_{Y\vert z}^{f_0}\left(\dfrac{\sum_{j=1}^{M_n}n_j\left(\bar Y_j - \bar Y\right)^2}{\sum_{j=1}^{M_n}n_j}\right)\nonumber\\
&=& \frac{1}{n}E_{Y\vert z}^{f_0}\left(\sum_{j=1}^{M_n}n_j(\bar Y_j-\mu_{T})^2 - n(\bar Y-\mu_{T})^2\right)\nonumber\\
&=& \frac{M_n-1}{n}\sigma^{2}_{T} = \mu_{n}^{*}, \label{eq:mean_cond}
\end{eqnarray}
noting the fact that since, given that $Y\sim f_0$, $Y$ and $Z$ are independent, 
$V_{Y\vert z}^{f_0}(Y_i) = \sigma^{2}_{T}$. 
Hence, 
\begin{eqnarray}
E^{f_0}\left(\dfrac{\sum_{j=1}^{M_n}n_j\left(\bar Y_j - \bar Y\right)^2}{\sum_{j=1}^{M_n}n_j}\right) 
= \frac{M_n-1}{n}\sigma^{2}_{T} = \mu_{n}^{*}.
\label{eq:mean_sigma_n_lemma}
\end{eqnarray}
Note that if $M_n > 1$ for all $n$, then $\mu_{n}^{*} > 0$ and if $M_n \prec O(n)$,
then $\mu_{n}^{*}\rightarrow 0$. \\

Similarly we can split the variance term as
\begin{eqnarray}
&&V^{f_0}\left(\dfrac{\sum_{j=1}^{M_n}n_j\left(\bar Y_j - \bar Y\right)^2}{\sum_{j=1}^{M_n}n_j}\right) \nonumber\\
&=& V_{z}\left[\frac{1}{n}E_{Y\vert z}^{f_0}\left(\sum_{j=1}^{M}n_j(\bar Y_j-\mu_{T})^2 - n(\bar Y-\mu_{T})^2\right)\right]\nonumber\\
&+& E_{z}\left[V_{Y\vert z}^{f_0}\left(\frac{1}{n}\sum_{j=1}^{M}n_j(\bar Y_j-\mu_{T})^2 - n(\bar Y-\mu_{T})^2\right)\right]. 
\label{eq:var_cond}
\end{eqnarray}
From (\ref{eq:mean_cond}) we have $E_{Y\vert z}^{f_0}\left(\sum_{j=1}^{M}n_j(\bar Y_j-\mu_{T})^2 - n(\bar Y-\mu_{T})^2\right)$
is free of $z$. So the first term of (\ref{eq:var_cond}) is 0. Easy calculations shows that
the order the second term of (\ref{eq:var_cond}) is $\frac{2^n}{n\times (M_{n})^{n}}$.\\
Thus
\begin{eqnarray}
V^{f_0}\left(\dfrac{\sum_{j=1}^{M_n}n_j\left(\bar Y_j - \bar Y\right)^2}{\sum_{j=1}^{M_n}n_j}\right) = 
O\left(\frac{2^n}{n\times (M_{n})^{n}}\right).
\label{eq:var_sigma_n_lemma}
\end{eqnarray}
Note that for $M_n \geq 2$, $V^{f_0}\left(\dfrac{\sum_{j=1}^{M_n}n_j\left(\bar Y_j - \bar Y\right)^2}{\sum_{j=1}^{M_n}n_j}\right)$ 
converges to 0. 
Hence, under $f_0$, 
$\bigg\vert\dfrac{\sum_{j=1}^{M_n}n_j\left(\bar Y_j - \bar Y\right)^2}{\sum_{j=1}^{M_n}n_j} - \mu_{n}^{*}\bigg\vert\rightarrow 0$, 
in probability, for $M_n > 1$ and $M_n \prec O(n)$.

Also we note that, under $f_0$,
\begin{eqnarray}
P\left(\left[\dfrac{\sum_{j=1}^{M_n}n_j\left(\bar Y_j - \bar Y\right)^2}{\sum_{j=1}^{M_n}n_j}-\mu_n^{*}\right]^2 > 
\epsilon\right)&\leq&
\frac{1}{\epsilon}E^{f_0}\left(\left[\dfrac{\sum_{j=1}^{M_n}n_j\left(\bar Y_j - \bar Y\right)^2}{\sum_{j=1}^{M_n}n_j}-
\mu_n^{*}\right]^2\right) \nonumber\\
&=& O\left(\frac{2^n}{n\times (M_{n})^{n}}\right). \nonumber
\end{eqnarray}

Thus, for $M_n > 2$, 
\begin{eqnarray}
\sum_{n=1}^{\infty}P\left(\left[\dfrac{\sum_{j=1}^{M_n}n_j\left(\bar Y_j - \bar Y\right)^2}{\sum_{j=1}^{M_n}n_j}-
\mu_n^{*}\right]^2 > 
\epsilon\right) 
< \infty. \nonumber
\end{eqnarray}
Hence we conclude that, under $f_0$, 
$\bigg\vert \dfrac{\sum_{j=1}^{M_n}n_j\left(\bar Y_j - \bar Y\right)^2}{\sum_{j=1}^{M_n}n_j}-
\mu_n^{*} \bigg\vert\rightarrow 0$, a.s., for $M_n > 1$ and $M_n \prec O(n)$. Also we have $\mu_n^{*}\rightarrow 0$, 
a.s. under the
same set of conditions.  
Combining these two results we have
that, under $f_0$, $\dfrac{\sum_{j=1}^{M_n}n_j\left(\bar Y_j - \bar Y\right)^2}{\sum_{j=1}^{M_n}n_j}\rightarrow 0$, a.s., for 
$M_n > 2$ and $M_n \prec O(n)$. Note that, since we assumed $M_n$ to increase with $n$, the condition
$M_n>2$ holds at least after some initial values of $n$. 

\end{proof}

\section{{\bf Proofs of results associated with Section 5 of MB}}

\subsection{{\bf Proofs of results on the EW model}}
\label{subsec:ew_posterior_mean}

\begin{lemma}
\label{lemma:epsilon_ew}
Let $\{b_n\}$ and $\{\sigma_n\}$ be sequences of positive numbers such that   
$\sigma_n\rightarrow 0$, $0<b_n<\sigma_n$ for all $n$, 
$P(\sigma> \sigma_n)/P(b_n<\sigma\leq\sigma_n) 
= O\left(\frac{\epsilon_n}{1-\epsilon_n}\right)$, 
for some sequence of positive constants 
$\epsilon_n >0$ such that $\epsilon_n=o(1)$.
If $|Y_i|<a;i=1,\ldots,n$, then 
\begin{equation}
P(\sigma> \sigma_n\vert \bY_n) = O\left(\epsilon_n^*\right) 
\label{eq:third_part}
\end{equation}
\end{lemma}

\begin{proof}
$P(\sigma> \sigma_n\vert \bY_n) = \dfrac{\int_{\sigma_n}^{\infty}
\int_{\Theta_n} \prod_{j=1}^{n}\frac{1}{\sigma}\phi\left(\frac{Y_j-\theta_j}{\sigma}\right) dG_n(\sigma)
dH_n\left(\Theta_n\right)}{\int_{\sigma}\int_{\Theta_n}\prod_{j=1}^{n}\frac{1}{\sigma}
\phi\left(\frac{Y_j-\theta_j}{\sigma}\right) dG_n(\sigma)dH_n\left(\Theta_n\right)} = \frac{N}{D}$, where $H_n\left(\Theta_n\right)$
is the joint distribution of $\Theta_n$.

Denote $L(\Theta_n, Y) = \prod_{j=1}^{n}\frac{1}{\sigma}\phi\left(\frac{Y_j-\theta_j}{\sigma}\right)$.

\begin{eqnarray}
D&\geq& \int_{\sigma\in (b_n, \sigma_n)}\int_{\Theta_n\in E_n}L(\Theta_n, Y)dG_n(\sigma)dH_n\left(\Theta_n\right)  \nonumber\\
&=& \dfrac{e^{-\left\{\frac{\sum_{i=1}^{n}(Y_i-\theta_i^*)^2}{2{b_n}^2}\right\}}}{(b_n)^n}P(b_n < \sigma \leq \sigma_n)P\left(\Theta_n\in E_n\right),
\nonumber
\end{eqnarray}
 
where $E_n=\{\theta_j\in [-c_1, c_1], i=1,\ldots,n\}$, and $\theta_i^*\in (-c_1, c_1)$, $c_1 > 0$ being a constant. 

Now $|Y_i|<a, |\theta_i^*|<c_1 \Rightarrow (Y_i-\theta_i^*)^2< (a+c_1)^2,\forall i
\Rightarrow \sum_{i=1}^n(Y_i-\theta_i^*)^2< n(a+c_1)^2$. 
Again, from properties of the Polya urn, $H_n\left(\Theta_n\right)\geq \prod_{i=1}^n \frac{\alpha_n}{\alpha_n+n}G_0(\theta_i)\Rightarrow 
P\left(\Theta_n\in E_n\right)\geq (\frac{\alpha_n}{\alpha_n+n})^n H_0^n$.
\\
Now, the function $\frac{1}{\sigma}\exp\{-\frac{n(a+c_1)^2}{2\sigma^2}\}$ is increasing on $b_n<\sigma<\sigma_n$ for
$\sigma_n<\sqrt{n}(2a+c)$. 
Hence, 
\begin{equation}
D\geq \dfrac{e^{\left(\frac{-n(a+c_1)^2}{2{b_n}^2}\right)}}{(b_n)^n}P(b_n < \sigma \leq \sigma_n)\left(\frac{\alpha_n}{\alpha_n+n}\right)^n H_0^n.
\label{eq:deno}
\end{equation}

For the numerator observe that for $\sigma>\sigma_n$, $\dfrac{e^{-\frac{\sum_{i=1}^n(Y_i-\theta_i)^2}{2\sigma^2}}}{\sigma^n}
\leq \frac{1}{(\sigma_n)^n}$. This implies
\begin{eqnarray}
N= \int_{\Theta_n}\int_{{\sigma> \sigma_n}}L(\Theta_n, Y)dG_n(\sigma)dH_n\left(\Theta_n\right)
&\leq& \frac{1}{(\sigma_n)^n}P(\sigma> \sigma_n) \nonumber\\
&=& \frac{O(\epsilon_n)}{(\sigma_n)^n}. 
\label{eq:num}
\end{eqnarray}
Inequalities (\ref{eq:deno}) and (\ref{eq:num}) together imply that
\begin{eqnarray}
P(\sigma> \sigma_n\vert \bY_n) \leq \frac{P(\sigma > \sigma_n)}{P(b_n< \sigma \leq \sigma_n)}\frac{{b_n}^n}{(\sigma_n)^n}
e^{\frac{n(a+c_1)^2}{2{b_n}^2}}\frac{(\alpha_n+n)^n}{(\alpha_n)^n H_0^n} = A_n^{*}, \mbox{say}. \nonumber
\end{eqnarray}

By the assumptions of the lemma, 
$P(\sigma> \sigma_n)/P(b_n<\sigma<\sigma_n) 
= O\left(\frac{\epsilon_n}{1-\epsilon_n}\right)$, 
where $\epsilon_n=o(1)$. 
Thus
$A_n^{*} = O\left(\epsilon_n^*\right)$, where
$\epsilon_n^*=\frac{\epsilon_n}{1-\epsilon_n}e^{\frac{n(a+c_1)^2}{2{b_n}^2}}\frac{(\alpha_n+n)^n}{(\alpha_n)^n H_0^n}$. 
This completes the proof.
\end{proof}

\begin{lemma}
\label{lemma:b_n_ew}
Under the same assumptions as Lemma \ref{lemma:epsilon_ew} and $0 <\sigma_n< a$, the following holds:
\begin{equation}
P\left(\theta_i\in [-a-c, a+c]^{c}\cap\mathbb S, \sigma\leq \sigma_n\vert \bY_n\right) \leq
\frac{C_0}{e^{-1/2}\delta}B_n,
\label{eq:eq:lemma_2}
\end{equation}
where 
$C_0 = \sup_{\sigma}\{\sigma^{-1}\exp(-c^2/4\sigma^2)\}$, and $\delta$ is the lower bound
of the density of $G_0$ on $[-a-c,a+c]$. 
\end{lemma}

\begin{proof} 
This proof is similar to that of Lemma 11 of \ctn{Ghosal07}.
\end{proof}

\begin{lemma}
\label{lemma:A_n}
$\frac{\alpha_n}{\alpha_n+n}A_n=O\left(\frac{\alpha_n}{\alpha_n+n}\right)$, almost surely.
\end{lemma}
\begin{proof} 

Note that, by the mean value theorem for integrals, $\varphi(\theta,\sigma+\hat k_n)
=\left[\int_{-a}^a\frac{1}{\sigma+\hat k_n}\phi\left(\frac{y-\theta}{\sigma+\hat k_n}\right)dy\right]^{-1}
=\left[\frac{2a}{\sigma+\hat k_n}\phi\left(\frac{y^*(\theta,\sigma+\hat k_n)-\theta}{\sigma+\hat k_n}\right)\right]^{-1}$,
where $y^*(\theta,\sigma+\hat k_n)\in[-a,a]$.
Hence, for $y\in [-a,a]$, 
\begin{align}
\frac{\varphi(\theta,\sigma+\hat k_n)}{(\sigma+\hat k_n)\sqrt{2\pi}}
e^{-\frac{\left(y-\theta\right)^2}{2(\sigma+\hat k_n)^2}}\mathbb I_{\{|y|\leq a\}}
&=\frac{1}{2a}\exp\left(\frac{{y^*(\theta,\sigma+\hat k_n)}^2-y^2}{\left(\sigma+\hat k_n\right)^2}\right)
\times\exp\left(2\theta\frac{\left(y-y^*(\theta,\sigma+\hat k_n)\right)}{\left(\sigma+\hat k_n\right)^2}\right)\notag\\
&\leq\frac{1}{2a}\exp\left(\frac{a^2+4a\times\underset{\theta\in\mathbb S}{\sup}~|\theta|}{\eta^2}\right)
=H_{1}~\mbox{(say)},
\label{eq:H_1}
\end{align}
where we recall that $\eta$ is a lower bound of the sequence $\hat k_n$, for $n\geq 1$, for almost all sequences
$\hat k_n$.

Hence, for all $y\in [-a,a]$,
\begin{eqnarray}
A_n(y) &=& \int_{\mathbb S}\frac{\varphi(\theta,\sigma+\hat k_n)}{(\sigma+\hat k_n)\sqrt{2\pi}}
e^{-\frac{\left(y-\theta\right)^2}{2(\sigma+\hat k_n)^2}}dG_0(\theta)\nonumber\\
&\leq& H_{1} \int_{\mathbb S}dG_0(\theta)\nonumber\\
&=& H_{1}.
\label{eq:A}
\end{eqnarray}
Thus $A_n=O(1)$, almost surely, and 
\begin{equation}
\frac{\alpha_n}{\alpha_n+n}A_n \stackrel{a.s.}{=} O\left(\frac{\alpha_n}{\alpha_n+n}\right). \label{eq:order_A}
\end{equation}
\end{proof}
\noindent {\it Remark 1:
For $\alpha_n=O(n^{\omega}); 0<\omega <1$, $\frac{\alpha_n}{\alpha_n+n}\rightarrow 0$, and so,
$A_n=o(1)$, almost surely.
}

\subsubsection{{\bf Proof of Theorem 5.1}}
\label{subsubsec:proof_theorem_5.1}

\begin{proof}

Note that
\begin{eqnarray}
&&E\left(\frac{\varphi(\theta_i,\sigma+\hat k_n)}{(\sigma+\hat k_n)}\phi\left(\frac{y-\theta_i}{\sigma+\hat k_n}\right)
\bigg\vert \bY_n\right) \nonumber\\
&=& \frac{1}{D}\int_{\sigma}\int_{\Theta_n}\frac{\varphi(\theta_i,\sigma+\hat k_n)}{(\sigma+\hat k_n)}
\phi\left(\frac{y-\theta_i}{\sigma+\hat k_n}\right)L(\Theta_n,\bY_n)dH(\Theta_n)dG_n(\sigma)  \nonumber\\ 
&=& J_1+J_2+J_3,  \nonumber\\
\label{eq:lemma3_expec}
\end{eqnarray}
where \\
$
J_1=\frac{1}{D}\int_{R_1}\frac{\varphi(\theta_i,\sigma+\hat k_n)}{(\sigma+\hat k_n)}
\phi\left(\frac{y-\theta_i}{\sigma+\hat k_n}\right)L(\Theta_n,\bY_n)dH(\Theta_n)dG_n(\sigma),$\\
$J_2=\frac{1}{D}\int_{R_2}\frac{\varphi(\theta_i,\sigma+\hat k_n)}{(\sigma+\hat k_n)}
\phi\left(\frac{y-\theta_i}{\sigma+\hat k_n}\right)L(\Theta_n,\bY_n)dH(\Theta_n)dG_n(\sigma),$\\
$J_3=\frac{1}{D}\int_{R_3}\frac{\varphi(\theta_i,\sigma+\hat k_n)}{(\sigma+\hat k_n)}
\phi\left(\frac{y-\theta_i}{\sigma+\hat k_n}\right)L(\Theta_n,\bY_n)dH(\Theta_n)dG_n(\sigma),$\\
$
R_1=\{\theta_i\in [-a-c, a+c], \sigma\leq \sigma_n\},$\\
$R_2=\{\theta_i\in [-a-c, a+c]^c\cap\mathbb S, \sigma\leq \sigma_n\},$\\
$R_3=\{\sigma> \sigma_n\}.$\\

Then it follows from 
Lemma \ref{lemma:epsilon_ew} that
\begin{equation}
J_3\leq H_{1} P\left(R_3\big\vert \bY_n\right)\leq H_{1}\epsilon_n^*,
\label{eq:lemma3_part3}
\end{equation}
\\[5mm]
Using 
Lemma \ref{lemma:b_n_ew} we obtain
\begin{equation}
J_2\leq H_{1} P\left(R_2\big\vert \bY_n\right)\leq H_{1} B_n. \label{eq:lemma3_part2}
\end{equation}

\begin{equation}
J_1=\frac{\varphi(\mu^*_n(y),\sigma+\hat k_n)}{(\upsilon_n(y)+\hat k_n)}
\phi\left(\frac{y-\mu_n^*(y)}{\upsilon_n(y)+\hat k_n}\right)
\left(1-P\left(R_2\big\vert \bY_n\right)-P\left(R_3\big\vert \bY_n\right)\right),
\label{eq:lemma3_part1}
\end{equation}
where, for every $y$, $\mu_n^*(y)\in (-a-c,a+c)$ and $\upsilon_n(y)\in (0,\sigma_n)$, applying 
the general mean value theorem ($GMVT$).

Let us choose $\epsilon_n$ and $\sigma_n$ in a way such that $\epsilon_n^*$ and $B_n$ converge to zero as $n\rightarrow \infty$.
Now we note that, in $J_{1}$, the range of $\theta_i$ remains the same for all $n$. It is the range of $\sigma$ that varies with $n$ and the point
$\upsilon_n(y)$ varies with $n$. Moreover, $\hat k_n$ also changes with $n$. These have the effect of 
varying $\mu^*_n(y)$ since the kernel depends upon both $\theta_i$, $\sigma$ and $\hat k_n$. 
This implies that $\mu_n^*(y)$ and $\upsilon_n(y)$ depend upon $\sigma_n$ and $\hat k_n$, so that 
$\mu_n^*(y)=\mu(\sigma_n,\hat k_n,y)$,
$\upsilon_n(y)=\varPsi(\sigma_n,\hat k_n,y)$, such that $\varPsi(0,k,y)=0$ (note that $\upsilon_n(y)< \sigma_n$ and 
$\sigma_n\rightarrow 0$). We also assume that $\mu(0,k,y)= \mu^*(y)$.
We assume that $\mu$ and $\varPsi$ are continuously partially differentiable with respect to the
first two arguments at least once; indeed we can
choose $\mu$ and $\varPsi$ to be smooth functions such that $\mu_n^*(y)=\mu(\sigma_n,\hat k_n,y)$, $\mu(0,k,y)= \mu^*(y)$, 
$\upsilon_n(y)=\varPsi(\sigma_n,\hat k_n,y)$, and $\varPsi(0,k,y)=0$ .

Then we obtain the following Taylor's series expansion, letting $x=\sigma_n$ and $\tilde k=\hat k_n$, and expanding around
$x=0$ and $\tilde k = k$:
\begin{eqnarray}
&&\frac{\varphi(\mu(\sigma_n,\hat k_n,y),\varPsi(\sigma_n,\hat k_n,y)+\hat k_n)}
{(\varPsi(\sigma_n,\hat k_n,y)+\hat k_n)}\phi\left(\frac{y-\mu(\sigma_n,\hat k_n,y)}
{\varPsi(\sigma_n,\hat k_n,y)+\hat k_n}\right) \nonumber\\
&=& \frac{\varphi(\mu(0,k,y),k)}{k}\phi\left(\frac{y-\mu(0,k,y)}{k}\right)\nonumber\\
&&\qquad+ 
\sigma_n\frac{\partial}{\partial \sigma_n}\left(\frac{\varphi(\mu(\sigma_n,\hat k_n,y),
\varPsi(\sigma_n,\hat k_n,y)+\hat k_n)}
{(\varPsi(\sigma_n,\hat k_n,y)+\hat k_n)}\phi\left(\frac{y-\mu(\sigma_n,\hat k_n,y)}
{\varPsi(\sigma_n,\hat k_n,y)+\hat k_n}\right)\right)\bigg\vert_{\sigma_n=\sigma^*_n,\hat k_n=\hat k^*_n}\notag\\
&&\qquad +
(\hat k_n-k)\frac{\partial}{\partial \hat k_n}\left(\frac{\varphi(\mu(\sigma_n,\hat k_n,y),
\varPsi(\sigma_n,\hat k_n,y)+\hat k_n)}
{(\varPsi(\sigma_n,\hat k_n,y)+\hat k_n)}\phi\left(\frac{y-\mu(\sigma_n,\hat k_n,y)}
{\varPsi(\sigma_n,\hat k_n,y)+\hat k_n}\right)\right)\bigg\vert_{\sigma_n=\sigma^*_n,\hat k_n=\hat k^*_n},
\nonumber
\end{eqnarray}
where $\sigma^*_n$ lies between $0$ and $\sigma_n$, and $\hat k^*_n$ lies between $\hat k_n$ and $k$.
Noting that the terms are bounded for any
$y\in [-a,a]$, we have, almost surely,
\begin{equation}
\sup_{|y|\leq a}\bigg\vert\frac{\varphi(\mu_n^*(y),\upsilon_n(y)+\hat k_n)}{(\upsilon_n(y)+\hat k_n)}
\phi\left(\frac{y-\mu_n^*(y)}{\upsilon_n(y)+\hat k_n}\right)-
\frac{\varphi(\mu(0,k,y),k)}{k}\phi\left(\frac{y-\mu(0,k,y)}{k}\right)\bigg\vert=O(\sigma_n+|\hat k_n-k|).  
\label{eq:expansion_lemma3}
\end{equation}
Thus, for any $y\in[-a,a]$, we have that 
$\frac{\varphi(\mu_n^*(y),\upsilon_n(y)+\hat k_n)}{(\upsilon_n(y)+\hat k_n)}
\phi\left(\frac{y-\mu_n^*(y)}{\upsilon_n(y)+\hat k_n}\right)\rightarrow 
\frac{\varphi(\mu^*(y),k)}{k}
\phi\left(\frac{y-\mu^*(y)}{k}\right)$, almost surely.

It follows from (\ref{eq:expansion_lemma3}) that
\begin{equation}
\sup_{|y|\leq a}\bigg\vert J_1-\frac{\varphi(\mu^*(y),k)}{k}\phi\left(\frac{y-\mu^*(y)}{k}\right)\bigg\vert=
O\left(B_n+\epsilon_n^*+\sigma_n+|\hat k_n-k|\right). \label{eq:obs_bias}
\end{equation}

Finally, since $B_n$ and $\epsilon_n^*$ can be made to converge to 0, $\sigma_n\rightarrow 0$
and $|\hat k_n-k|\stackrel{a.s.}{\rightarrow} 0$, $J_2, J_3\rightarrow 0$, 
$J_1\rightarrow \frac{\varphi(\mu^*(y),k)}{k}
\phi\left(\frac{y-\mu^*(y)}{k}\right)$.
Also it holds that
\begin{eqnarray}
\int_y J_1 dy &=&\frac{1}{D}\int_{|y|\leq a}\int_{R_1}\frac{\varphi(\theta_i,\sigma_k)}{(\sigma+k)}
\phi\left(\frac{y-\theta_i}{\sigma+k}\right)
L(\Theta_n,\bY_n)dH(\Theta_n)dG_n(\sigma)dy\nonumber\\
&=& \ \ 1-P\left(R_2\big\vert \bY_n\right)-P\left(R_3\big\vert \bY_n\right)\nonumber\\
&\rightarrow& 1,~\mbox{almost surely}.\label{eq:J_1_integral_1}
\end{eqnarray}
Again, since we have proved above that 
$J_1\rightarrow \frac{\varphi(\mu^*(y),k)}{k}
\phi\left(\frac{y-\mu^*(y)}{k}\right)$, if we can show that for all $n$, 
$J_1$ given by (\ref{eq:lemma3_part1}) is bounded above by an integrable function $h(y)$ for every $y\in[-a,a]$,
then it follows from the dominated convergence
theorem (DCT) that
\begin{equation}
\int_{|y|\leq a} J_1dy\rightarrow \int_{|y|\leq a} \frac{\varphi(\mu^*(y),k)}{k}\phi\left(\frac{y-\mu^*(y)}{k}\right)dy.
\label{eq:J_1_integral_2}
\end{equation}
But DCT clearly holds since
$J_1\leq \frac{\varphi(\mu_n^*(y),\upsilon_n(y)+\hat k_n)}{\upsilon_n(y)+\hat k_n}
\phi\left(\frac{y-\mu_n^*(y)}{\upsilon_n(y)+\hat k_n}\right)$, and since both $y$ and $\mu^*_n(y)$ are in compact sets
and $\hat k_n\geq \eta>0$ for all $n\geq 1$.


It follows from (\ref{eq:J_1_integral_1}) and (\ref{eq:J_1_integral_2}) that
\begin{equation}
\int_{|y|\leq a}\frac{\varphi(\mu^*(y),k)}{k}\phi\left(\frac{y-\mu^*(y)}{k}\right)dy=1,
\label{eq:J_1_integral_3}
\end{equation}
showing that $f_0(y)=\frac{\varphi(\mu^*(y),k)}{k}\phi\left(\frac{y-\mu^*(y)}{k}\right)\mathbb I_{\{|y|\leq a\}}$ is a density.
Indeed, we consider $f_0$ as the true data-generating density.

Now note that,
\begin{eqnarray}
E\left(\hat f_{EW}(y\mid \Theta_n,\sigma)\bigg\vert \bY_n\right)=O\left(\frac{\alpha_n}{\alpha_n+n}\right)
+\frac{1}{\alpha_n+n}\sum_{i=1}^{n}E\left(\frac{\varphi(\theta_i,\sigma+k)}{(\sigma+k)}
\phi\left(\frac{y-\theta_i}{\sigma+k}\right)\bigg\vert \bY_n\right), \nonumber
\end{eqnarray}
where the first term is thanks to (\ref{eq:order_A}).
Further note that, 
\begin{equation}
\frac{1}{\alpha_n+n}\sum_{i=1}^{n}E\left(\frac{\varphi(\theta_i,\sigma+k)}{(\sigma+k)}
\phi\left(\frac{y-\theta_i}{\sigma+k}\right)\bigg\vert \bY_n\right)=\frac{n}{\alpha_n+n}\times (J_1+J_2+J_3). 
\label{eq:expec_split}
\end{equation}

\begin{eqnarray}
&&E\left(\hat f_{EW}(y\mid \Theta_n,\sigma)\bigg\vert \bY_n\right)
-f_0(y)\nonumber\\
&&\ \ =O\left(\frac{\alpha_n}{\alpha_n+n}\right)-\frac{\alpha_n}{\alpha_n+n}f_0(y)\nonumber\\
&&\quad\quad +\frac{n}{\alpha_n+n}(J_1-f_0(y))+\frac{n}{\alpha_n+n}(J_2+J_3)\nonumber
\end{eqnarray}
Since 
$f_0(y)$ is uniformly bounded, we have
\begin{eqnarray}
&&\bigg\vert E\left(\hat f_{EW}(y\mid \Theta_n,\sigma)\bigg\vert \bY_n\right)
-f_0(y)\bigg\vert\nonumber\\
&&\ \ \leq O\left(\frac{\alpha_n}{\alpha_n+n}\right)
+\frac{n}{\alpha_n+n}\left(\vert J_2\vert+\vert J_3\vert\right)\nonumber\\
&&\quad\quad\quad+\frac{n}{\alpha_n+n}\vert J_1-f_0(y)\vert\nonumber
\end{eqnarray}
It follows from 
(\ref{eq:lemma3_part3}), (\ref{eq:lemma3_part2}), and (\ref{eq:obs_bias})
and the fact that the orders of $J_1,J_2,J_3$ are independent of $y$, that 
\begin{eqnarray}
&&\sup_{|y|\leq a}\bigg\vert E\left(\hat f_{EW}(y\mid \Theta_n,\sigma)\bigg\vert \bY_n\right)
-f_0(y)\bigg\vert\nonumber\\
&&\quad\quad =O\left(\frac{\alpha_n}{\alpha_n+n}+\frac{n}{\alpha_n+n}(B_n+\epsilon^*_n+\sigma_n+|\hat k_n-k|)\right),\nonumber
\end{eqnarray}
proving the theorem.

\end{proof}

\subsection{{\bf Proofs of results on the SB model}}
\label{subsec:sb_posterior_mean}

\begin{lemma}
Under the same assumptions as in Lemma \ref{lemma:epsilon_ew} and Lemma \ref{lemma:b_n_ew},  
\label{lemma:sb_epsilon}
$$P\left(\sigma> \sigma_n\vert \bY_n\right)=O(\epsilon_{M_n}^*),$$ where \\
$\epsilon_{M_n}^*=
\frac{\epsilon_n}{1-\epsilon_n}\exp\left(\frac{n(a+c_1)^2}{2(b_n)^2}\right)\frac{(\alpha_n+M_{n})^{M_n}}{\alpha_n^{M_n} H_0^{M_n}}$. 
\end{lemma}

\begin{proof}
$P\left(\sigma> \sigma_n\vert \bY_n\right) = \dfrac{\sum_{z}\int_{\sigma_n}^{\infty}
\int_{\Theta_{M_n}}  L(\Theta_{M_n}, z, \bY_n)dH(\Theta_{M_n})dG_n(\sigma)}{\sum_{z}\int_{0}^{\infty}
\int_{\Theta_{M_n}}L(\Theta_{M_n}, z, \bY_n)dH(\Theta_{M_n})dG_n(\sigma)} = \frac{N}{D}$, where $H(\Theta_{M_n})$ is the joint distribution of $\Theta_{M_n}$ and
\begin{eqnarray}
L(\Theta_{M_n}, z, \bY_n)=\prod_{j=1}^{M_n}\frac{1}{\sigma^{n_j}}e^{-\frac{1}{2}\sum_{t:z_t=j}\left(\frac{Y_{t}-\theta_j}{\sigma}\right)^2},
\nonumber
\end{eqnarray}
where $n_j=\#\{t:z_t=j\}$ (for any set $\mathbb A$, $\#\mathbb A$ denotes the cardinality of 
the set $\mathbb A$).
Let $E^{*}=$\{all $\theta_{l}\in \Theta_{M_n}$ are in $[-c_1, c_1]$\}.

Then,
\begin{eqnarray}
&&\int_{0}^{\infty}\int_{\Theta_{M_n}}L(\Theta_{M_n}, z, \bY_n)dH(\Theta_{M_n})dG_n(\sigma) \nonumber\\
&\geq& \int_{{\sigma\in (b_n, \sigma_n)}}\int_{{\Theta_{M_n}\in E^{*}}}L(\Theta_{M_n}, z, \bY_n)dH(\Theta_{M_n})dG_n(\sigma). \nonumber\\
\label{eq:lemma_sb_ep}
\end{eqnarray}

Now note that 
\begin{eqnarray}
&&|Y_{t}|<a, |\theta_j^*|< c_1 \Rightarrow (Y_{t}-\theta_j^*)^2< (a+c_1)^2,  \nonumber\\
&\Rightarrow& \sum_{j=1}^{M_n}\sum_{t:z_t=j}\left(Y_{t}-\theta_j^*\right)^2<
 n(a+c_1)^2.
\nonumber
\end{eqnarray}

Again, from the Polya urn scheme we have $H(\Theta_{M_n})\geq \prod_{j=1}^{M_n} \frac{\alpha_n}{\alpha_n+M_n}G_0(\theta_j)$ 
which implies
$P\left(\Theta_{M_n}\in E^{*}\right)\geq 
\left(\frac{\alpha_n}{\alpha_n+M_n}\right)^{M_n} H_0^{M_n}$, where $H_0 = \int_{-c_1}^{c_1} G_0(\theta)d\theta$.
Hence, for $\sigma_n<\sqrt{n}(a+c_1)$,
\begin{equation}
D\geq M_n^n\dfrac{\exp\left(\frac{-n(a+c_1)^2}{2(b_n)^2}\right)}{b_n}P(b_n<\sigma\leq \sigma_n)\left(\frac{\alpha_n}{\alpha_n+M_n}\right)^{M_n} H_0^{M_n}.
\label{eq:deno_sb}
\end{equation}

Again, in the same way as Lemma \ref{lemma:epsilon_ew}, 
$N\leq M_n^n\frac{1}{(\sigma_n)^n}P(\sigma> \sigma_n)$.
Also, by our assumption, 
$P(\sigma> \sigma_n)/P(b_n<\sigma\leq \sigma_n)=O\left(\frac{\epsilon_n}{1-\epsilon_n}\right)$, so
it follows that, 
\begin{equation}
P\left(\sigma> \sigma_n\vert \bY_n\right)=O(\epsilon_{M_n}^*),
\nonumber
\end{equation}
where
\begin{equation}
\epsilon^*_{M_n}=\frac{\epsilon_n}{1-\epsilon_n}\exp\left(\frac{n(a+c_1)^2}{2(b_n)^2}\right)
\frac{(\alpha_n+M_n)^{M_n}}{\alpha_n^{M_n} H_0^{M_n}}. \nonumber\\
\end{equation}
Hence, the proof follows.
\end{proof}

\begin{lemma}
\label{lemma:sb_b_n}
Under the same assumptions as in Lemma \ref{lemma:b_n_ew}, 
$P(Z\in R_1^*, \Theta_{M_{n}}\in E^c, \sigma\leq \sigma_n\vert \bY_n)= O((M_{n}-1)B_{M_{n}})$, where 
$B_{M_{n}}=\frac{(\alpha_n+M_{n})}{\alpha_n} e^{\left(-\frac{c^2}{4\sigma_n^2}\right)}$.\\
\end{lemma}

\begin{proof} 
Clearly, $E^c$ =\{at least one $\theta_k$ in the likelihood is in $[-a-c, a+c]^c\cap\mathbb S$\}. We have
\begin{eqnarray}
&&P(Z\in R_1^*, \Theta_{M_n}\in E^c, \sigma\leq \sigma_n\vert \bY_n) \nonumber\\
&=&\dfrac{\sum_{j=1}^{(M_n-1)}\sum_{l=1}^{j}\sum_{z\in V_j}\int_{\sigma\leq \sigma_n}\int_{W_l}L(\Theta_{M_n}, z, \bY_n)dG_n(\sigma)dH(\Theta_{M_n})}{
\sum_{z}\int_{\sigma}\int_{\Theta_{M_n}} L(\Theta_{M_n}, z, \bY_n) dG_n(\sigma)dH(\Theta_{M_n})} \nonumber\\
&\leq& \dfrac{\sum_{j=1}^{(M_n-1)}\sum_{l=1}^{j}\sum_{z\in V_j}\int_{\sigma\leq \sigma_n}\int_{W_l}L(\Theta_{M_n}, z, \bY_n) dG_n(\sigma)
dH(\Theta_{M_n})}{\sum_{j=1}^{(M_n-1)}\sum_{z\in V_j}\int_{\sigma<\mathcal K} \int_{\Theta_{M_n}} L(\Theta_{M_n}, z, \bY_n) dG_n(\sigma)dH(\Theta_{M_n})}. \nonumber\\
\label{eq:less_sb}
\end{eqnarray}
In the denominator of the last step, $\mathcal K$ is such that $0<\sigma_n < \mathcal K$ for all $n$
(since $\sigma_n\rightarrow 0$ as $n\rightarrow\infty$, $\sigma_n$ must be bounded).

Note that 
\begin{equation}
L(\Theta_{M_n}, z, \bY_n)=\prod_{j=1}^{M_n}\frac{1}{\sigma^{n_j}}e^{-\frac{1}{2}\sum_{t:z_t=j}\left(\frac{Y_{t}-\theta_j}{\sigma}\right)^2}
=\prod_{j=1}^{M_n}\frac{1}{\sigma^{n_j}}e^{-\frac{1}{2}\sum_{t:z_t=j}\left(\frac{Y_{t}-\bar Y_j}{\sigma}\right)^2}
e^{-\frac{n_j}{2}\left(\frac{\bar Y_j-\theta_j}{\sigma}\right)^2},
\label{eq:L_split}
\end{equation}
where $n_j=\#\{l:z_l=j\}$ and $\bar Y_j=\frac{\sum_{l:z_l=j}Y_l}{n_j}$.\\

Let $H_{j}(\theta_j\mid \Theta_{-jM_n})$ be the conditional distribution of $\theta_j$ given $\Theta_{-jM_n}$
and $H_{-j}(\Theta_{-jM_n})$ the joint distribution of $\Theta_{-jM_n}$, where $\Theta_{-jM_n}=\Theta_{M_n}\setminus \theta_j$. 
Since
\begin{equation}
H_{j}(\theta_j\mid \Theta_{-jM_n})=\frac{\alpha_n}{\alpha_n+M_n-1}G_0(\theta_j)+\frac{1}{\alpha_n+M_n-1}\sum_{l=1, l\neq j}^{M_n}\delta_{\theta_l},
\label{eq:sb_dirichlet}
\end{equation}
and $\sigma<\mathcal K$ in the integral associated with the denominator, 
we have in the denominator for each $z\in V_{j}$, 
\begin{eqnarray}
&&\int_{\theta_j} e^{-\frac{n_j}{2}\left(\frac{\bar Y_j-\theta_j}{\sigma}\right)^2}dH_j(\theta_j\mid \Theta_{-jM_n}) \nonumber\\
&\geq& \frac{\alpha_n}{\alpha_n+M_n} 
\int_{\theta_j\in[\bar Y_j-\frac{\sigma}{n_j^{1/2}},{\bar Y_j+\frac{\sigma}{n_j^{1/2}}}]\cap\mathbb S}
\exp\left(-\frac{n_j(\bar Y_j-\theta_j)^2}{2\sigma^2}\right)dG_0(\theta_j)  \nonumber\\
&\geq& \frac{\alpha_n}{\alpha_n+M_n}e^{-1/2} \frac{\sigma}{n_j^{1/2}}\delta,  
\label{eq:lemma8_d1_sb}
\end{eqnarray}
where $\delta$ is the lower bound of the density of $G_0$ on 
$[\bar Y_j-\frac{\sigma}{n^{1/2}_j},\bar Y_j+\frac{\sigma}{n^{1/2}_j}]\cap\mathbb S$
(we assume that the density of $G_0$ is strictly positive in neighborhoods of $\bar Y_j$, for each $j$;
since $\sigma<\mathcal K$ the neighborhoods must be bounded so that the lower bound of the density 
on such neighborhoods can be assumed to be bounded away from zero).

Thus for each $z\in V_{j}$ we have,
\begin{eqnarray}
&&\int_{\sigma<\mathcal K} \int_{\Theta_{M_n}} L(\Theta_{M_n}, z, \bY_n) dG_n(\sigma)dH(\Theta_{M_n})  \nonumber\\
&\geq& \frac{\alpha_n}{(\alpha_n+M_n)n_j^{1/2}}e^{-1/2}\delta \nonumber\\
&&\ \ \times\int_{0}^{\mathcal K}\int_{\Theta_{-jM_n}} \left[\frac{1}{\sigma^{(n-1)}} \prod_{l=1}^{M_n} e^{-\frac{1}{2}\sum_{t:z_t=l}
\left(\frac{Y_{t}-\bar Y_l}{\sigma}\right)^2} \right. \nonumber\\
&&\hspace{25mm}\left. \prod_{l=1, l\neq j}^{M_n}e^{-\frac{n_l}{2}\left(\frac{\bar Y_l-\theta_l}{\sigma}\right)^2}
dH_{-j}(\Theta_{-jM_n})dG_n(\sigma)\right]. \nonumber\\
&&=\frac{\alpha_n}{(\alpha_n+M_n)n_j^{1/2}}e^{-1/2}\delta\times \zeta_n(j,z),\ \ \mbox{(say)},
\end{eqnarray}
where
\begin{eqnarray}
&&\zeta_n(j,z)=\int_{0}^{\mathcal K}\int_{\Theta_{-jM_n}} \left[\frac{1}{\sigma^{(n-1)}} \prod_{l=1}^{M_n} e^{-\frac{1}{2}\sum_{t:z_t=l}
\left(\frac{Y_{t}-\bar Y_l}{\sigma}\right)^2} \right. \nonumber\\
&&\hspace{25mm}\left. \prod_{l=1, l\neq j}^{M_n}e^{-\frac{n_l}{2}\left(\frac{\bar Y_l-\theta_l}{\sigma}\right)^2}
dH_{-j}(\Theta_{-jM_n})dG_n(\sigma)\right]. \nonumber\\
\end{eqnarray}

To obtain an upper bound for the numerator we note that for each $z\in V_{j}$ and $j=1(1)M_n$, $|\bar Y_j|<a$ (since each $|Y_l|<a$, $l=1, \ldots, n$) and 
$\theta_j\in [-a-c, a+c]^c\cap\mathbb S$. Since $\sigma<\sigma_n$ for the integral associated with the numerator, 
and $n_j\geq 1$, this implies
\begin{eqnarray}
&&\frac{1}{\sigma}\exp\left(-n_j(\bar Y_j-\theta_j)^2/2\sigma^2\right) \nonumber\\
&\leq& \frac{1}{n_j^{1/2}}\frac{n_j^{1/2}}{\sigma}\exp\left(-n_jc^2/4\sigma^2\right) \exp(-c^2/4\sigma_n^2)  \nonumber\\
&\leq& \frac{A_1^*}{n_j^{1/2}}\exp(-c^2/4\sigma_n^2),
\nonumber
\end{eqnarray}
where $A_1^*=\sup_{\sigma,n_j}\left\{\frac{n_j^{1/2}}{\sigma}\exp\left(-\frac{n_jc^2}{4\sigma^2}\right)\right\}$.
It is easy to check that $A_1^*$ is free of $n$. Thus for each $z\in V_j$, 
\begin{eqnarray}
&&\int_{\sigma\leq \sigma_n}\int_{W_l}L(\Theta_{M_n}, z, \bY_n)dG_n(\sigma)dH(\Theta_{M_n}) \nonumber\\
&\leq& \frac{A_1^*}{n_j^{1/2}}\exp(-c^2/4\sigma_n^2) \nonumber\\
&&\ \ \times\int_{0}^{\sigma_n}\int_{\Theta_{-jM_n}}\left[\frac{1}{\sigma^{(n-1)}}\prod_{l=1}^{M_n}e^{-\frac{1}{2}\sum_{t:z_t=l}\left(\frac{Y_{t}-\bar Y_l}{\sigma}
\right)^2} \right. \nonumber\\
&&\hspace{30mm}\times\left.\prod_{l\neq j}e^{-\frac{n_l}{2}\left(\frac{\bar Y_l-\theta_l}{\sigma}\right)^2}
dH_{-j}(\Theta_{-jM_n})dG_n(\sigma)\right]
\nonumber\\
&\leq& A_1^*\exp(-c^2/4\sigma_n^2)\times\zeta_n(j,z).\nonumber
\end{eqnarray}

As a result, using (\ref{eq:less_sb}), we see that
\begin{eqnarray}
&&P(Z\in R_1^*, \Theta_{M_n}\in E^c, \sigma\leq \sigma_n\vert \bY_n) \nonumber\\
&\leq&
\dfrac{\sum_{j=1}^{(M_n-1)}\sum_{l=1}^{j}\sum_{z\in V_j}\int_{\sigma\leq \sigma_n}\int_{W_l}L(\Theta_{M_n}, z, \bY_n)dG_n(\sigma)
dH(\Theta_{M_n})}{\sum_{j=1}^{(M_n-1)}\sum_{z\in V_j}\int_{\sigma<\mathcal K} \int_{\Theta_{M_n}} L(\Theta_{M_n},z,\bY_n) dG_n(\sigma)dH(\Theta_{M_n})} \nonumber\\
&\leq& \dfrac{A_1^*\exp(-c^2/4\sigma_n^2)\times\sum_{j=1}^{(M_n-1)}\sum_{l=1}^{j}\sum_{z\in V_j}\zeta_n(j,z)}
{\frac{\alpha_n}{\alpha_n+M_n}e^{-1/2}\delta\times\sum_{j=1}^{(M_n-1)}\sum_{z\in V_j}\zeta_n(j,z)}\nonumber\\
&\leq& \dfrac{A_1^*\exp(-c^2/4\sigma_n^2)\times\sum_{j=1}^{(M_n-1)}\sum_{l=1}^{(M_n-1)}\sum_{z\in V_j}\zeta_n(j,z)}
{\frac{\alpha_n}{\alpha_n+M_n}e^{-1/2}\delta\times\sum_{j=1}^{(M_n-1)}\sum_{z\in V_j}\zeta_n(j,z)}\nonumber\\
&=& \frac{(M_n-1)A_1^*\exp(-c^2/4\sigma_n^2)}{\frac{\alpha_n}{\alpha_n+M_n}e^{-1/2}\delta}, 
\label{eq:sb_lemma81}
\end{eqnarray}
proving the lemma.

\end{proof}

\begin{lemma}
\label{lemma:sb_lemma9}
Let 
\begin{equation}
C_n\gtrsim \dfrac{n\left[\log\left(\frac{1}{\sigma_n}\right)+O\left(\frac{1}{n}
\log\left(\frac{1-\epsilon_n}{\epsilon_n}\right)\right)\right]}{\left(\frac{1}{\sigma_n^2}\right)}, 
\label{eq:C_n_sb_lemma9}
\end{equation}
where ``$\gtrsim$" stands for ``$\geq$" as $n\rightarrow\infty$.
\\[2mm]
Then,
\begin{equation}
P(Z\in R_1^*, \Theta_{M_{n}}\in E, \sigma\leq \sigma_n\vert \bY_n) = 
O\left(\left(1-\frac{1}{M_{n}}\right)^n \left(\frac{\alpha_n+M_{n}}{\alpha_n}\right)^{M_n}\right).
\label{eq:lemma9}
\end{equation}
\end{lemma}

\begin{proof}

Note that,
\begin{eqnarray}
&&P(Z\in R_1^*, \Theta_{M_n}\in E, \sigma\leq \sigma_n\vert \bY_n) \nonumber\\
&=&\dfrac{\sum_{z\in R_1^*}\int_{\sigma\leq \sigma_n}\int_{\Theta_{M_n}\in E}\int L(\Theta_{M_n}, z, \bY_n)dG_n(\sigma)dH(\Theta_{M_n})}{
\sum_{z}\int_{\sigma}\int_{\Theta_{M_n}} L(\Theta_{M_n}, z, \bY_n) dG_n(\sigma)dH(\Theta_{M_n})} \nonumber\\
&=& \frac{N}{D}. \label{eq:sb_lemma9}
\end{eqnarray}


Since
\begin{eqnarray}
L(\Theta_{M_n}, z, \bY_n)&=&\frac{1}{\sigma^n}e^{-\frac{\sum_{j=1}^{M_n} \sum_{t:z_t=j}(Y_{t}-\bar Y_j)^2}{2\sigma^2}}
\times e^{-\frac{\sum_{j=1}^{M_n}n_j(\bar Y_j-\theta_j)^2}{2\sigma^2}}
\nonumber\\
&& \ \ \leq \frac{1}{\sigma^n}e^{-\frac{\sum_{j=1}^{M_n} \sum_{t:z_t=j}(Y_{t}-\bar Y_j)^2}{2\sigma^2}},\nonumber 
\end{eqnarray}
it follows that

\begin{eqnarray}
&&\int_{\Theta_{M_n}\in E}\int_{0}^{\sigma_n}L(\Theta_{M_n}, z, \bY_n)
dH(\Theta_{M_n})dG_n(\sigma) \nonumber\\
&& \ \ \leq \int_{\theta_{1}\in [-a-c, a+c]}\int_{\Theta_{-1M_n}}\int_{0}^{\sigma_n}
\frac{1}{\sigma^n}e^{\frac{-\sum_{j=1}^{M_n}\sum_{t:z_t=j}(Y_{t}-\bar Y_j)^2}{2\sigma^2}}
dH(\Theta_{M_n})dG_n(\sigma) \nonumber\\
&&\ \ \leq A_1^* \int_{\theta_{1}\in [-a-c, a+c]}\int_{\Theta_{-1M_n}}\int_{0}^{\sigma_n}dH(\Theta_{M_n})dG_n(\sigma) \nonumber\\
&&\ \ = A_1^* G_0([-a-c, a+c])O(1-\epsilon_n),\label{eq:N_equation}
\end{eqnarray}
where 
$A_1^*=\sup_{\{\sigma\in (0,\sigma_n)\}}\left\{\frac{1}{\sigma^n}e^{-\frac{C_n^{(1)}}{2\sigma^2}}\right\}
= \left(\frac{1}{\sigma_n}\right)^ne^{-\frac{C_n^{(1)}}{2\sigma_n^2}}$, \\
$C_n^{(1)}=\inf_{\{z\in R_1^*\}}\left(\sum_{j=1}^{M_n}\sum_{t:z_t=j}(Y_{t}-\bar Y_j)^2\right)$.
\\[2mm]
Clearly, for each $z\in R_1^{*}$ each term in $N$ is bounded above by 
\begin{eqnarray}
N^*=\left(\frac{1}{\sigma_n}\right)^ne^{-
\frac{C_n^{(1)}}{2\sigma_n^2}}\times G_0([-a-c, a+c])\times O(1-\epsilon_n).
\nonumber 
\end{eqnarray}
Hence 
\begin{equation}
N\leq \left(M_n-1\right)^n N^{*}.
\label{eq:N_SB1}
\end{equation}


Let 
$C_n^{(2)}=\sup_{z\in R_1^*}\sum_{j=1}^{M_n}\sum_{t:z_t=j}(Y_{t}-\bar Y_j)^2$. 
Now, assuming that $k_n$ is a sequence diverging to $\infty$ and denoting 
$R^{*}$=\{$\theta_1\in [\bar Y_1-k_n, \bar Y_1+k_n]\cap\mathbb S,
\ldots,\theta_d\in [\bar Y_d-k_n, \bar Y_d+k_n]\cap\mathbb S$, rest $\theta_{l}$'s are in $\mathbb S$, 
$nk_n\leq\sigma\leq 2nk_n\}$, where $1\leq d\leq M_n$,

\begin{eqnarray}
D&=&\int_{\Theta_{M_n}}\int_{0}^{\infty}\frac{1}{\sigma^n}
\left[e^{-\dfrac{\sum_{j=1}^{M_n}\sum_{t:z_t=j}(Y_{t}-\bar Y_j)^2}{2\sigma^2}} \right. \nonumber\\
&&\left.\hspace{25mm}\times \ e^{-\dfrac{\sum_{j=1}^{M_n}n_j(\bar Y_j-\theta_j)^2}{2\sigma^2}}
dH(\Theta_{M_n})dG_n(\sigma) \right] \nonumber\\
&\geq& \int_{R^{*}}\left[\frac{1}{\sigma^n}
e^{-\dfrac{\sum_{j=1}^{M_n}\sum_{t:z_t=j}(Y_{t}-\bar Y_j)^2}{2\sigma^2}} \right. \nonumber\\
&&\left.\hspace{10mm}\times \ e^{-\dfrac{\sum_{j=1}^{M_n}n_j(\bar Y_j-\theta_j)^2}{2\sigma^2}}
dH(\Theta_{M_n})dG_n(\sigma)\right] \nonumber\\
&\geq& \inf_{\{z, \sigma\in [nk_n,2nk_n]\}}\left(\frac{1}{\sigma^n}
e^{-\frac{\sum_{j=1}^{M_n}\sum_{t:z_t=j}(Y_{t}-\bar Y_j)^2}{2\sigma^2}}\right) \nonumber\\
&&\ \ \times\int_{R^{*}} e^{-\dfrac{\sum_{j=1}^{d}n_j(\bar Y_j-\theta_j)^2}{2\sigma^2}}\times dH(\Theta_{M_n})dG_n(\sigma) \nonumber\\
&\geq & \left(\frac{1}{2nk_n}\right)^n e^{-\frac{C_n^{(2)}}{8n^2k_n^2}}\times 
\int_{R^{*}}e^{-\frac{nk_n^2}{2\sigma^2}}\times 
dH(\Theta_{M_n})dG_n(\sigma) \nonumber\\
&\geq& \left(\frac{1}{2nk_n}\right)^n e^{-\frac{C_n^{(2)}}{8n^2k_n^2}}\times e^{-\frac{1}{2n}}\times 
\int_{R^{*}}dH(\Theta_{M_n})dG_n(\sigma), \nonumber
\end{eqnarray}
because $\sigma\geq nk_n\Rightarrow e^{-\frac{nk_n^2}{2\sigma^2}}\geq e^{-\frac{1}{2n}}$. Thus,
\begin{eqnarray}
D &\geq& M_n^n\left(\frac{1}{2nk_n}\right)^n e^{-\frac{C_n^{(2)}}{8n^2k_n^2}}\times 
\prod_{j=1}^{d}G_0([\bar Y_j-k_n, \bar Y_j+k_n]\cap\mathbb S)\times
e^{-\frac{1}{2n}}\nonumber\\
&&\times\left(\frac{\alpha_n}{\alpha_n+M_n}\right)^{M_n} \times O(\epsilon_n), 
\label{eq:D_equation}
\end{eqnarray}
assuming that $\int_{nk_n}^{2nk_n}dG_n(\sigma)=O(\epsilon_n)$ as well.\\

Inequalities (\ref{eq:N_SB1}) and (\ref{eq:D_equation}) imply that 
$\frac{N}{D}$ is of the order
\begin{eqnarray}
&&\dfrac{\left(M_n-1\right)^n\left(\frac{1}{\sigma_n}\right)^ne^{-\frac{C_n^{(1)}}{2\sigma_n^2}}\times G_0([-a-c, a+c]) \times O(1-\epsilon_n)}{
M_n^n \left(\frac{1}{2nk_n}\right)^n
e^{-\frac{C_n^{(2)}}{8n^2k_n^2}} e^{-\frac{1}{2n}} \left(\frac{\alpha_n}{\alpha_n+M_n}\right)^{M_n}
\prod_{j=1}^{d}G_0([\bar Y_j-k_n, \bar Y_j+k_n]\cap\mathbb S)\times O(\epsilon_n)}, \nonumber\\
\label{eq:lemma9_ratio}
\end{eqnarray}
where $\prod_{j=1}^{d}G_0([\bar Y_j-k_n, \bar Y_j+k_n]\cap\mathbb S)\rightarrow 1$ as $n\rightarrow\infty$.

As shown in Section 5 of MB, $\sum_{j=1}^{M_n}\sum_{t:z_t=j}(Y_{t}-\bar Y_j)^2/n$ converges to $\sigma_{T}^2$ almost
surely. That is, the quantity $\sum_{j=1}^{M_n}\sum_{t:z_t=j}(Y_{t}-\bar Y_j)^2$ is asymptotically independent
of $z$.
Hence, as $n\rightarrow\infty$, it holds, almost surely, that 
$C_n^{(1)}\sim C_n^{(2)}\sim C_n$. We now investigate the appropriate order of $C_n$ such that
\begin{equation}
\left(\frac{1}{\sigma_n}\right)^n e^{-\frac{C_n}{2\sigma_n^2}} G_0([-a-c, a+c]) O(1-\epsilon_n)\lesssim
\left(\frac{1}{2nk_n}\right)^n e^{-\frac{C_n}{8n^2k_n^2}}
O(\epsilon_n)
\label{eq:C_n_1}
\end{equation}
holds for large $n$.

Taking logarithm of both sides of (\ref{eq:C_n_1}) yields
\begin{eqnarray}
&&n\log\left(\frac{1}{\sigma_n}\right)-\frac{C_n}{2\sigma_n^2}+O\left(\log\left(\frac{1-\epsilon_n}{\epsilon_n}\right)\right)+
\log\left(H_0\right)  \nonumber\\
&\lesssim& n\log\left(\frac{1}{2nk_n}\right)-\frac{C_n}{8n^2k_n^2} \nonumber\\
&\Leftrightarrow& C_n\left(\frac{1}{2\sigma_n^2}-\frac{1}{8n^2 k_n^2}\right)\gtrsim n\log\left(\frac{1}{\sigma_n}\right)-
n\log\left(\frac{1}{2nk_n}\right)\nonumber\\
&&\hspace{5cm} +O\left(\log\left(\frac{1-\epsilon_n}{\epsilon_n}\right)\right)+\log\left(H_0\right) \nonumber\\
&\Leftrightarrow& C_n\gtrsim \dfrac{n\left[\log\left(\frac{1}{\sigma_n}\right)+\log(2nk_n)\right]+O\left(\log\left(\frac{1-\epsilon_n}{\epsilon_n}\right)\right)}{
\left(\frac{1}{2\sigma_n^2}-\frac{1}{8n^2 k_n^2}\right)} \nonumber\\
&&\hspace{5cm}+\dfrac{\log\left(H_0\right)}{\left(\frac{1}{2\sigma_n^2}-\frac{1}{8n^2 k_n^2}\right)} \nonumber\\
&\Leftrightarrow& C_n\gtrsim \dfrac{n\left[\log\left(\frac{1}{\sigma_n}\right)+O\left(\frac{1}{n}
\log\left(\frac{1-\epsilon_n}{\epsilon_n}\right)\right)\right]}{\left(\frac{1}{\sigma_n^2}\right)}.
\label{eq:C_n_lemma9}
\end{eqnarray}


Thus, (\ref{eq:C_n_1}) holds if $C_n$ is of the form (\ref{eq:C_n_sb_lemma9}). 
%
%
Hence, 
(\ref{eq:lemma9}) 
holds under the additional assumption 
(\ref{eq:C_n_sb_lemma9}).

\end{proof} 

\begin{lemma}
\label{lemma:sb_lemma10} 
$P(Z\in (R_1^*)^c, \theta_i\in [-a-c,a+c]^c\cap\mathbb S, \sigma\leq \sigma_n\vert \bY_n)=O\left(B_{M_{n}}\right)$, 
where $B_{M_{n}}$ is defined in Lemma \ref{lemma:sb_b_n}.\\
\end{lemma}

\begin{proof}
When $z\in (R_1^*)^c$, then $\theta_i$ is present in the likelihood and hence in $\Theta_{z}$. 
Thus the same calculations associated with Lemma \ref{lemma:sb_b_n}, now only with $\theta_i$, guarantees the result.
\end{proof}

\subsubsection{{\bf Proof of Theorem 5.2}}
\label{subsubsec:proof_theorem_5.2}

\begin{proof}


\begin{eqnarray}
&&E\left(\frac{\varphi(\theta_i,\sigma+\hat k_n)}{(\sigma+\hat k_n)}\phi\left(\frac{y-\theta_i}{\sigma+\hat k_n}\right)\bigg\vert \bY_n\right) 
\nonumber\\
&=&\dfrac{\sum_{z}\int_{\Theta_{M_n}}\int_{\sigma}\frac{\varphi(\theta_i,\sigma+\hat k_n)}{(\sigma+\hat k_n)}
\phi\left(\frac{y-\theta_i}{\sigma+\hat k_n}\right) 
L(\Theta_{M_n},z,\bY_n)dH(\Theta_{M_n})dG_n(\sigma)}{\sum_{z}\int_{\Theta_{M_n}}
\int_{\sigma}L(\Theta_{M_n},z,\bY_n)dH(\Theta_{M_n})dG_n(\sigma)} \nonumber\\
&=& \frac{N}{D}, 
\label{eq:lemma5_E1}
\end{eqnarray}
where $L(\Theta_{M_n},z,\bY_n) = \prod_{j=1}^{M_n}e^{-\frac{\sum_{t:z_t=j}(Y_{t}-\bar Y_j)^2}{2\sigma^2}}e^{-\frac{n_j(\bar Y_j-\theta_j)^2}{2\sigma^2}}$ is the likelihood of
$\Theta_{M_n}$, $n_j=$ $\#\{i: z_i = j\}$, $\bar Y_j = \frac{1}{n_j}\sum_{t:z_t=j}Y_{t}$.\\

To simplify the calculations we can split the set of all of $z$'s into $R_1^*$ and $(R_1^*)^c$; the cardinality of the set of
all $z$-vectors satisfying 
these conditions are $(M_n-1)^n$ and  $M_n^n-(M_n-1)^n$, 
respectively. Denote $I_1=\{\Theta_{M_n}\in E^c, \sigma\leq \sigma_n\}$, $I_2=\{\Theta_{M_n}\in E, \sigma\leq \sigma_n\}$, 
$I_3=\{\theta_i\in [-a-c,a+c]^c\cap\mathbb S, \sigma\leq \sigma_n\}$, 
$I_4=\{\theta_i\in [-a-c,a+c], \sigma\leq \sigma_n\}$, $I_5=\{\sigma> \sigma_n\}$,
where $E$ has been defined in Section 7 of MB. 
Note that
\begin{eqnarray}
\{R_1^{*}\cap I_1\}\cup \{R_1^{*}\cap I_2\} &=& R_1^{*}\cap \{\sigma\leq \sigma_n\}; \nonumber\\
\{(R_1^{*})^{c}\cap I_3\}\cup \{(R_1^{*})^{c}\cap I_4\} &=& (R_1^{*})^{c}\cap \{\sigma\leq \sigma_n\}; \nonumber\\
\left(\{R_1^{*}\cap I_1\}\cup \{R_1^{*}\cap I_2\}\right)\cup 
\left(\{(R_1^{*})^{c}\cap I_3\}\cup \{(R_1^{*})^{c}\cap I_4\}\right) &=& \{\sigma\leq \sigma_n\}. \nonumber
\end{eqnarray}

We write
\begin{eqnarray}
E\left(\frac{\varphi(\theta_i,\sigma+\hat k_n)}{(\sigma+\hat k_n)}
\phi\left(\frac{y-\theta_i}{\sigma+\hat k_n}\right)\bigg\vert \bY_n\right) 
= S_1+S_2+S_3+S_4+S_5, \label{eq:expec_sum}
\end{eqnarray}
where\\
$
S_1=\frac{1}{D}\sum_{R_1^*}\int_{I_1}\frac{\varphi(\theta_i,\sigma+\hat k_n)}{(\sigma+\hat k_n)}\phi\left(\frac{y-\theta_i}{\sigma+\hat k_n}\right)L(\Theta_{M_n},z, \bY_n)dH(\Theta_{M_n})dG_n(\sigma), \\
S_2 = \frac{1}{D}\sum_{R_1^*}\int_{I_2}\frac{\varphi(\theta_i,\sigma+\hat k_n)}{(\sigma+\hat k_n)}\phi\left(\frac{y-\theta_i}{\sigma+\hat k_n}\right)L(\Theta_{M_n},z, \bY_n)dH(\Theta_{M_n})dG_n(\sigma),\\
S_3 = \frac{1}{D}\sum_{(R_1^*)^c}\int_{I_3}\frac{\varphi(\theta_i,\sigma+\hat k_n)}{(\sigma+\hat k_n)}\phi\left(\frac{y-\theta_i}{\sigma+\hat k_n}\right)L(\Theta_{M_n},z, \bY_n)dH(\Theta_{M_n})dG_n(\sigma),\\
S_4 = \frac{1}{D}\sum_{(R_1^*)^c}\int_{I_4}\frac{\varphi(\theta_i,\sigma+\hat k_n)}{(\sigma+\hat k_n)}\phi\left(\frac{y-\theta_i}{\sigma+\hat k_n}\right)L(\Theta_{M_n},z, \bY_n)dH(\Theta_{M_n})dG_n(\sigma),\\
S_5 = \frac{1}{D}\sum_{z}\int_{I_5}\frac{\varphi(\theta_i,\sigma+\hat k_n)}{(\sigma+\hat k_n)}\phi\left(\frac{y-\theta_i}{\sigma+\hat k_n}\right)L(\Theta_{M_n},z, \bY_n)dH(\Theta_{M_n})dG_n(\sigma). \\
$
\\

Also let\\
$
P_1 = P(Z\in R_1^*, \Theta_{M_n}\in E^c, \sigma\leq \sigma_n\vert \bY_n),\\
P_2 = P(Z\in R_1^*, \Theta_{M_n}\in E, \sigma\leq \sigma_n\vert \bY_n),\\
P_3 = P(Z\in (R_1^*)^c, \theta_i\in [-a-c,a+c]^c\cap\mathbb S, \sigma\leq \sigma_n\vert \bY_n),\\
P_4=  P(Z\in (R_1^*)^c, \theta_i\in [-a-c,a+c], \sigma\leq \sigma_n\vert \bY_n),\\
P_5 = P(\sigma\geq \sigma_n\vert \bY_n).
$
\\

Recalling that $H_1$ is given by (\ref{eq:H_1}), the upper bounds of the terms $S_1,\ldots,S_5$ are given as follows.
\begin{equation}
S_1\leq H_1P(Z\in R_1^*, \Theta_{M_n}\in E^c, \sigma\leq \sigma_n\vert \bY_n)\leq H_1(M_n-1)B_{M_n}, \label{eq:sum_1_bound}
\end{equation}
from Lemma \ref{lemma:sb_b_n}.

\begin{equation}
S_2\leq H_1P(Z\in R_1^*, \Theta_{M_n}\in E, \sigma\leq \sigma_n\vert \bY_n)\leq H_1\left(1-\frac{1}{M_n}\right)^n 
\left(\frac{\alpha_n+M_n}{\alpha_n}\right)^{M_n}, \label{eq:sum_2_bound}
\end{equation}
from Lemma \ref{lemma:sb_lemma9}.

\begin{equation}
S_3\leq H_1 P(Z\in (R_1^*)^c, \theta_i\in [-a-c, a+c]^c\cap\mathbb S, \sigma\leq \sigma_n\mid \bY_n)\leq H_1B_{M_n}, \label{eq:sum_3_bound}
\end{equation}
from Lemma \ref{lemma:sb_lemma10}.

\begin{equation}
S_5 \leq H_1 P(Z\in (R_1^*)^c, \sigma> \sigma_n\mid \bY_n)\leq H_1\epsilon_{M_n}^*, \label{eq:sum_5_bound}
\end{equation}
from Lemma \ref{lemma:sb_epsilon}.

\begin{eqnarray}
S_4 &=& \frac{1}{D}\int_{I_4}\frac{\varphi(\theta_i,\sigma+\hat k_n)}{(\sigma+\hat k_n)}
\phi\left(\frac{y-\theta_i}{\sigma+\hat k_n}\right)\sum_{z\in (R_1^*)^c}L(\Theta_{M_n},z, \bY_n)
dH(\Theta_{M_n})dG_n(\sigma) \nonumber\\
&=& \frac{1}{D}\frac{\varphi(\theta_n^*(y),\sigma_n^*(y)+\hat k_n)}{(\sigma_n^*(y)+\hat k_n)}\phi\left(\frac{y-\theta_n^*(y)}{\sigma_n^*(y)+\hat k_n}\right) \nonumber\\
&&\ \ \times \sum_{z\in (R_1^*)^c}\int\int_{\theta_i\in [-a-c, a+c]}
\int_{0}^{\sigma_n}L(\Theta_{M_n},z,\bY_n)dH(\Theta_{M_n})dG_n(\sigma)  \label{eq:step_1}\\ 
&=& \frac{\varphi(\theta_n^*(y),\sigma_n^*(y)+\hat k_n)}{(\sigma_n^*(y)+\hat k_n)}\phi\left(\frac{y-\theta_n^*(y)}{\sigma_n^*(y)+\hat k_n}\right) \nonumber\\
&&\ \ \times P(Z\in (R_1^*)^c, \theta_i\in [-a-c, a+c],\sigma\leq \sigma_n\mid \bY_n) \nonumber\\
&=& \frac{\varphi(\theta_n^*(y),\sigma_n^*(y)+\hat k_n)}{(\sigma_n^*(y)+\hat k_n)}
\phi\left(\frac{y-\theta_n^*(y)}{\sigma_n^*(y)+\hat k_n}\right)(1-P_1-P_2-P_3-P_5),  \nonumber\\
\label{eq:sum_4}
\end{eqnarray}
where (\ref{eq:step_1}) is obtained by using $GMVT$, $\theta_n^*(y)\in (-a-c, a+c)$, and
$\sigma_n^*(y)\in (0,\sigma_n)$. \\
The integration and summation 
can be interchanged since the number of terms under summation is finite for a particular value of $n$.\\

Note that equations (\ref{eq:sum_1_bound})--(\ref{eq:sum_5_bound}), and $P_1, P_2, P_3, P_5$ converge to zero under proper
conditions. In particular, $P_1$ converges to 0 if $\sigma_n$ is chosen to be sufficiently small. 
Also $b_n$ can also chosen to be very small such that it satisfies $b^2_n<\sigma^2_n$ for all $n$. 

%

These choices 
get $P_3$ to converge to 0 and
$P_5$ converges to zero if $\frac{\epsilon_n}{1-\epsilon_n}\prec 
\left(\exp\left(\frac{n(2a+c)^2}{2(b_n)^2}\right)\frac{(\alpha_n+M_n)^{M_n}}{(\alpha_n)^{M_n} H_0^{M_n}}\right)^{-1}$.\\

$P_2$ converges to zero if $M_n\prec \sqrt{n}$, however, the form of the bound  (\ref{eq:lemma9}) 
given by Lemma \ref{lemma:sb_lemma9} is 
valid if
(\ref{eq:C_n_sb_lemma9}) holds.

Now note that $S_1+S_2+S_3+S_5 \leq H_1\left[P_1+P_2+P_3+P_5\right]$. Since under the specified assumptions $P_1, P_2, P_3, P_5$
converge to 0, as $n\rightarrow \infty$, the sum also goes to 0, as $n\rightarrow \infty$. Thus in $S_4$, the term $(1-P_1-P_2-P_3-P_5)\rightarrow 1$ as
$n\rightarrow \infty$. Uniform convergence of 
$\frac{\varphi(\theta_n^*(y),\sigma_n^*(y)+\hat k_n)}{(\sigma_n^*(y)+\hat k_n)}
\phi\left(\frac{y-\theta_n^*(y)}{\sigma_n^*(y)+\hat k_n}\right)$  to  
$\frac{\varphi(\theta^*(y),k)}{k}\phi\left(\frac{y-\theta^*(y)}{k}\right)$
can be proved in exactly the same way using Taylor's series expansion as done in the case of the EW model. 
In particular, it holds that 
\begin{equation}
\sup_{|y|\leq a}\bigg\vert\frac{\varphi(\theta_n^*(y),\sigma_n^*(y)+\hat k_n)}{(\sigma_n^*(y)+\hat k_n)}
\phi\left(\frac{y-\theta_n^*(y)}{\sigma_n^*(y)+\hat k_n}\right)-\frac{\varphi(\theta^*(y),k)}{k}
\phi\left(\frac{y-\theta^*(y)}{k}\right)
\bigg\vert
= O(\sigma_n+|\hat k_n-k|). \label{eq:expec_order1}
\end{equation}

We also conclude that 
\begin{eqnarray}
&&E\left(\frac{\varphi(\theta_i,\sigma+\hat k_n)}{(\sigma+\hat k_n)}
\phi\left(\frac{y-\theta_i}{\sigma+\hat k_n}\right)\bigg\vert \bY_n\right)-\frac{\varphi(\theta^*(y),k)}{k}
\phi\left(\frac{y-\theta^*(y)}{k}\right) \nonumber\\
&=&S_1+S_2+S_3+S_5\nonumber\\
&&\ \ +\frac{\varphi(\theta_n^*(y),\sigma_n^*(y)+\hat k_n)}{(\sigma_n^*(y)+\hat k_n)}
\phi\left(\frac{y-\theta_n^*(y)}{\sigma_n^*(y)+\hat k_n}\right)\left(1-P_1-P_2-P_3-P_5\right)-
\frac{\varphi(\theta^*(y),k)}{k}\phi\left(\frac{y-\theta^*(y)}{k}\right) \nonumber\\
&=&S_1+S_2+S_3+S_5-\frac{\varphi(\theta_n^*(y),\sigma_n^*(y)+\hat k_n)}{(\sigma_n^*(y)+\hat k_n)}
\phi\left(\frac{y-\theta_n^*(y)}{\sigma_n^*(y)+\hat k_n}\right)\left(P_1+P_2+P_3+P_5\right)\nonumber\\
&&\ \ +\frac{\varphi(\theta_n^*(y),\sigma_n^*(y)+\hat k_n)}{(\sigma_n^*(y)+\hat k_n)}
\phi\left(\frac{y-\theta_n^*(y)}{\sigma_n^*(y)+\hat k_n}\right)-
\frac{\varphi(\theta^*(y),k)}{k}\phi\left(\frac{y-\theta^*(y)}{k}\right) \nonumber\\
&=&O\left((M_n-1)B_{M_n}\right)+O\left(\left(1-\frac{1}{M_n}\right)^{n} 
\left(\frac{\alpha_n+M_n}{\alpha_n}\right)^{M_n}\right) \nonumber\\
&&\ \ +O\left(B_{M_n}\right)+O\left(\epsilon_{M_n}^*\right)+O(\sigma_n+|\hat k_n-k|) \nonumber\\
&=&O\left(M_n B_{M_n}+\left(1-\frac{1}{M_n}\right)^n \left(\frac{\alpha_n+M_n}{\alpha_n}\right)^{M_n}
+\epsilon_{M_n}^*+\sigma_n+|\hat k_n-k|\right). \nonumber\\
\label{eq:sb_expect}
\end{eqnarray}
It follows that
\begin{eqnarray}
&&\sup_{|y|\leq a}\bigg\vert E\left(\hat f_{SB}(y\mid \Theta_{M_n},\sigma)\mid\bY_n\right)
-\frac{\varphi(\theta^*(y),k)}{k}\phi\left(\frac{y-\theta^*(y)}{k}\right)\bigg\vert\nonumber\\
&&\ \ \leq\sup_{|y|\leq a}\bigg\vert E\left(\frac{\varphi(\theta_i,\hat k_n)}{(\sigma+\hat k_n)}
\phi\left(\frac{y-\theta_i}{\sigma+\hat k_n}\right)\bigg\vert \bY_n\right)-\frac{\varphi(\theta^*(y),k)}{k}
\phi\left(\frac{y-\theta^*(y)}{k}\right)\bigg\vert \nonumber\\
&&\ \ =O\left(M_n B_{M_n}+\left(1-\frac{1}{M_n}\right)^n \left(\frac{\alpha_n+M_n}{\alpha_n}\right)^{M_n}+\epsilon_{M_n}^*
+\sigma_n+|\hat k_n-k|\right), \nonumber\\
\end{eqnarray}
proving the theorem.
\end{proof}

\subsection{{\bf Proof of Theorem 5.3}}
\begin{proof}
Recall that
$J_1=\frac{1}{D_1}\int_{I_4}\frac{\varphi(\theta_i,\sigma+\hat k_n)}{(\sigma+\hat k_n)}
\phi\left(\frac{y-\theta_i}{\sigma+\hat k_n}\right)L(\Theta_n,\bY_n)dH(\Theta_n)dG_n(\sigma),$ and\\
$S_4 = \frac{1}{D_2}\int_{I_4}\frac{\varphi(\theta_i,\sigma+\hat k_n)}{(\sigma+\hat k_n)}
\phi\left(\frac{y-\theta_i}{\sigma+\hat k_n}\right)\sum_{z\in (R_1^*)^c}L(\Theta_{M_n},z, \bY_n)
dH(\Theta_{M_n})dG_n(\sigma)$, where $D_1$ and $D_2$ denote the normalizing constants of the posteriors corresponding to 
the EW and the SB models, respectively.
\\
Let $L=\max(M_n,n)$. Then,
\\
$\abs{J_1-S_4}$
\begin{eqnarray}
&=&\left\vert\int_{I_4}\left[\frac{1}{D_1}\frac{\varphi(\theta_i,\sigma+\hat k_n)}{(\sigma+\hat k_n)}
\phi\left(\frac{y-\theta_i}{\sigma+\hat k_n}\right)
\times L(\Theta_n,\bY_n)\right.\right.\nonumber\\
&& \ \ \left.\left.-\frac{1}{D_2}\frac{\varphi(\theta_i,\sigma+\hat k_n)}{(\sigma+\hat k_n)}
\phi\left(\frac{y-\theta_i}{\sigma+\hat k_n}\right)
\times\sum_{z\in (R_1^*)^c}L(\Theta_{M_n},z,\bY_n)\right]dH(\Theta_{L})dG_n(\sigma)\right\vert\nonumber\\
&=& \frac{\varphi(\theta_{1,n},\sigma_{1,n}+\hat k_n)}{(\sigma_{1,n}(y)+\hat k_n)}
\phi\left(\frac{y-\theta_{1,n}(y)}{\sigma_{1,n}(y)+\hat k_n}\right)
\times \left\vert P\left(I_4\vert\bY_n\right)-P\left((R^*_1)^c,I_4\vert\bY_n\right)\right\vert\nonumber\\
\label{eq:same_dist1}\\
&&\ \ \rightarrow 0.
\label{eq:same_dist2}
\end{eqnarray}
Step (\ref{eq:same_dist1}) follows using $GMVT$, where the notation have the usual meanings, 
and step (\ref{eq:same_dist2}) follows because
the first factor remains bounded and the second factor goes to zero 
(since $P\left(I_4\vert\bY_n\right)\rightarrow 1$, and 
$P\left((R^*_1)^c,I_4\vert\bY_n\right)\rightarrow 1$).
In other words, $J_1$ and $S_4$ converge to the same model. Hence, we must have $\mu^*(y)=\theta^*(y)$.
\end{proof}

\section{{\bf Proofs of results associated with Section 6 of MB}} 

\subsection{{\bf Proofs of results on the EW model}}

\subsubsection{{\bf Proof of Lemma 6.2}}

\begin{proof}
Note that 
\begin{eqnarray}
&&Var\left(\frac{\varphi(\theta_i,\sigma+\hat k_n)}{(\sigma+\hat k_n)}
\phi\left(\frac{y-\theta_i}{\sigma+\hat k_n}\right)\bigg\vert \bY_n\right)  \nonumber\\
&&= E\left(\frac{\varphi(\theta_i,\sigma+\hat k_n)}{(\sigma+\hat k_n)}
\phi\left(\frac{y-\theta_i}{\sigma+\hat k_n}\right)-E\left(\frac{\varphi(\theta_i,\sigma+\hat k_n)}{(\sigma+\hat k_n)}
\phi\left(\frac{y-\theta_i}{\sigma+\hat k_n}\right)\bigg\vert \bY_n\right)\bigg\vert \bY_n\right)^2
\nonumber\\
&&=J_1^{'}+J_2^{'}+J_3^{'}, \label{eq:var_ew_split}
\end{eqnarray}
where\\
$
J_1^{'}=\frac{1}{D}\int_{R_1}\xi_{in}^2\times L(\Theta_n,\bY_n)dH(\Theta_n)dG_n(\sigma),$\\
$J_2^{'}=\frac{1}{D}\int_{R_2}\xi_{in}^2\times L(\Theta_n,\bY_n)dH(\Theta_n)dG_n(\sigma),$\\
$J_3^{'}=\frac{1}{D}\int_{R_3}\xi_{in}^2\times L(\Theta_n,\bY_n)dH(\Theta_n)dG_n(\sigma),$\\
$\xi_{in}=\frac{\varphi(\theta_i,\sigma+\hat k_n)}{(\sigma+\hat k_n)}
\phi\left(\frac{y-\theta_i}{\sigma+\hat k_n}\right)
-E\left(\frac{\varphi(\theta_i,\sigma+\hat k_n)}{(\sigma+\hat k_n)}
\phi\left(\frac{y-\theta_i}{\sigma+\hat k_n}\right)\bigg\vert \bY_n\right).$\\

Clearly,
\begin{equation}
J_2^{'}\leq 4H_{1}^2\times B_n,  \label{eq:var_part2}
\end{equation}
and 
\begin{equation}
J_3^{'}\leq 4H_{1}^2\times \epsilon_n^*,  \label{eq:var_part3}
\end{equation}
where $H_1$ is given by (\ref{eq:H_1}).
As in the proof of Theorem 5.1 (see Section \ref{subsubsec:proof_theorem_5.1}), here also we set
\\
$R_1=\{\theta_i\in [-a-c, a+c], \sigma\leq \sigma_n\},$\\
$R_2=\{\theta_i\in [-a-c, a+c]^c\cap\mathbb S, \sigma\leq \sigma_n\},$\\
$R_3=\{\sigma> \sigma_n\}.$\\

Letting $P_1=P\left(R_1\big\vert \bY_n\right), P_2=P\left(R_2\big\vert \bY_n\right), P_3=P\left(R_3\big\vert \bY_n\right)$, we concentrate on the term
\begin{eqnarray}
&&J_1^{'}=\frac{1}{D}\int\int_{\theta_i\in [-a-c,a+c]}\int_{\sigma\leq \sigma_n} \xi_{in}^2\times L(\Theta_n,\bY_n)dH(\Theta_n)dG_n(\sigma)\nonumber\\
&&=\left[\frac{\varphi(m_n(y),\tau_n(y)+\hat k_n)}{(\tau_n(y)+\hat k_n)}
\phi\left(\frac{y-m_n(y)}{\tau_n(y)+\hat k_n}\right)-E\left(\frac{\varphi(\theta_i,\sigma+\hat k_n)}{(\sigma+\hat k_n)}
\phi\left(\frac{y-\theta_i}{\sigma+\hat k_n}\right)\bigg\vert \bY_n\right)\right] \nonumber\\
&&\ \ \times\frac{1}{D}\int\int_{\theta_i\in [-a-c,a+c]}\int_{\sigma\leq \sigma_n} 
\xi_{in}\times L(\Theta_n,\bY_n)dH(\Theta_n)dG_n(\sigma)\nonumber\\
\label{eq:ew_var1}
\end{eqnarray}
applying $GMVT$, where, for every $y$, $m_n(y)\in (-a-c,a+c)$, and $\tau_n(y)\in (0,\sigma_n)$.\\

Now we consider the following term:
\begin{eqnarray}
&&\frac{1}{D}\int\int_{\theta_i\in [-a-c,a+c]}\int_{\sigma\leq \sigma_n} \xi_{in}\times L(\Theta_n,\bY_n)dH(\Theta_n)dG_n(\sigma)\nonumber\\
&&\ \ =\frac{1}{D}\int\int_{\theta_i\in [-a-c,a+c]}\int_{\sigma\leq \sigma_n} 
\frac{\varphi(\theta_i,\sigma+\hat k_n)}{(\sigma+\hat k_n)}\phi\left(\frac{y-\theta_i}{\sigma+\hat k_n}\right)
L(\Theta_n,\bY_n)dH(\Theta_n)dG_n(\sigma) \nonumber\\
&&\ \ \ \ -E\left(\frac{\varphi(\theta_i,\sigma+\hat k_n)}{(\sigma+\hat k_n)}
\phi\left(\frac{y-\theta_i}{\sigma+\hat k_n}\right)\bigg\vert \bY_n\right)
P\left(\theta_i\in [-a-c,a+c], \sigma\leq \sigma_n\bigg\vert \bY_n\right) \nonumber\\
&&=J_{1}^{''}+J_{2}^{''}. \nonumber\\
\label{ew_var2}
\end{eqnarray}

For the part $J_{1}^{''}$, we note that 
\begin{eqnarray}
&&\frac{1}{D}\int\int_{\theta_i\in [-a-c,a+c]}\int_{\sigma\leq \sigma_n} 
\frac{\varphi(\theta_i,\sigma+\hat k_n)}{(\sigma+\hat k_n)}\phi\left(\frac{y-\theta_i}{\sigma+\hat k_n}\right)
L(\Theta_n,\bY_n)dH(\Theta_n)dG_n(\sigma) \nonumber\\  
&=& \frac{\varphi(\mu_n^*(y),\upsilon_n(y)+\hat k_n)}{(\upsilon_n(y)+\hat k_n)}
\phi\left(\frac{y-\mu_n^*(y)}{\upsilon_n(y)+\hat k_n}\right)(1-P_2-P_3).  \nonumber\\
\label{eq;ew_var3}
\end{eqnarray}

From (\ref{eq:lemma3_part1}) and following Theorem 5.1 of MB 
it follows that
\begin{eqnarray}
&&E\left(\frac{\varphi(\theta_i,\sigma+\hat k_n)}{(\sigma+\hat k_n)}
\phi\left(\frac{y-\theta_i}{\sigma+\hat k_n}\right)\bigg\vert \bY_n\right)  \nonumber\\
&=&  J_2+J_3+\frac{\varphi(\mu_n^*(y),\upsilon_n(y)+k)}{(\upsilon_n(y)+k)}
\phi\left(\frac{y-\mu_n^*(y)}{\upsilon_n(y)+k}\right)\left(1-P_2-P_3\right). 
\label{eq:ew_var4}
\end{eqnarray}

Thus,
\begin{eqnarray}
&&J_{1}^{''}+J_{2}^{''}  \nonumber\\
&=& \frac{\varphi(\mu_n^*(y),\upsilon_n(y)+\hat k_n)}{(\upsilon_n(y)+\hat k_n)}
\phi\left(\frac{y-\mu_n^*(y)}{\upsilon_n(y)+\hat k_n}\right)
\left[(1-P_2-P_3)-(1-P_2-P_3)^2\right] \nonumber\\
&&\ \ -(J_2+J_3)(1-P_2-P_3)\nonumber\\
&\leq& \frac{\varphi(\mu_n^*(y),\upsilon_n(y)+\hat k_n)}{(\upsilon_n(y)+\hat k_n)}
\phi\left(\frac{y-\mu_n^*(y)}{\upsilon_n(y)+\hat k_n}\right)(P_2+P_3)(1-P_2-P_3) \nonumber\\
&&\ \ + H_1(P_2+P_3)(1-P_2-P_3)\nonumber\\
&=& O\left(B_n+\epsilon_n^*\right), \label{eq:ew_var5}
\end{eqnarray}
since $\frac{\varphi(\mu_n^*(y),\upsilon_n(y)+\hat k_n)}{(\upsilon_n(y)+\hat k_n)}
\phi\left(\frac{y-\mu_n^*(y)}{\upsilon_n+\hat k_n}\right)\leq H_1$,
$P_2=O(B_n)$ and $P_3=O(\epsilon_n^*)$.

\end{proof}

\subsection{{\bf Proofs of results on the SB model}}

\subsubsection{{\bf Proof of Lemma 6.5}}
\begin{eqnarray}
&& Var\left(\frac{\varphi(\theta_i,\sigma+\hat k_n)}{(\sigma+\hat k_n)}
\phi\left(\frac{y-\theta_i}{\sigma+\hat k_n}\right)\bigg\vert \bY_n\right) \nonumber\\
&=& E\left(\frac{\varphi(\theta_i,\sigma+\hat k_n)}{(\sigma+\hat k_n)}
\phi\left(\frac{y-\theta_i}{\sigma+\hat k_n}\right)-
E\left(\frac{\varphi(\theta_i,\sigma+\hat k_n)}{(\sigma+\hat k_n)}
\phi\left(\frac{y-\theta_i}{\sigma+\hat k_n}\right)\bigg\vert \bY_n\right)\bigg\vert \bY_n\right)^2 \nonumber
\end{eqnarray}

As in 
(\ref{eq:sb_expect}) 
we begin with splitting up the range of $z$ and the range of integration of $\Theta_{M_n}$ and $\sigma$
in the following way:
\begin{eqnarray}
&&E\left(\frac{\varphi(\theta_i,\sigma+\hat k_n)}{(\sigma+\hat k_n)}
\phi\left(\frac{y-\theta_i}{\sigma+\hat k_n}\right)-
E\left(\frac{\varphi(\theta_i,\sigma+\hat k_n)}{(\sigma+\hat k_n)}
\phi\left(\frac{y-\theta_i}{\sigma+\hat k_n}\right)\bigg\vert \bY_n\right)\bigg\vert \bY_n\right)^2 \nonumber\\
&=& S_1^*+S_2^*+S_3^*+S_4^*+S_5^*,
\nonumber\\
\label{eq:sb_var_phi}
\end{eqnarray}
where $S_i^*$ has same ranges of $z$, $\Theta_{M_n}$ and $\sigma$ as $S_i$ in Theorem 5.1 of MB; 
only the integrand of 
the former is now
replaced with\\
$\zeta_{M_n}=\left[\frac{\varphi(\theta_i,\sigma+\hat k_n)}{(\sigma+\hat k_n)}
\phi\left(\frac{y-\theta_i}{\sigma+\hat k_n}\right)-
E\left(\frac{\varphi(\theta_i,\sigma+\hat k_n)}{(\sigma+\hat k_n)}
\phi\left(\frac{y-\theta_i}{\sigma+\hat k_n}\right)\bigg\vert \bY_n\right)\right]$, that is \\
$
S_1^*=\frac{1}{D}\sum_{R_1^*}\int_{I_1}\zeta_{M_n}^2 L(\Theta_{M_n},z, \bY_n)dH(\Theta_{M_n})dG_n(\sigma), \\
S_2^* = \frac{1}{D}\sum_{R_1^*}\int_{I_2}\zeta_{M_n}^2 L(\Theta_{M_n},z, \bY_n)dH(\Theta_{M_n})dG_n(\sigma),\\
S_3^* = \frac{1}{D}\sum_{(R_1^*)^c}\int_{I_3}\zeta_{M_n}^2 L(\Theta_{M_n},z, \bY_n)dH(\Theta_{M_n})dG_n(\sigma),\\
S_4^* = \frac{1}{D}\sum_{(R_1^*)^c}\int_{I_4}\zeta_{M_n}^2 L(\Theta_{M_n},z, \bY_n)dH(\Theta_{M_n})dG_n(\sigma),\\
S_5^* = \frac{1}{D}\sum_{z}\int_{I_5}\zeta_{M_n}^2 L(\Theta_{M_n},z, \bY_n)dH(\Theta_{M_n})dG_n(\sigma). \\
$

Then in the same way as in equations
(\ref{eq:sum_1_bound})--(\ref{eq:sum_5_bound}) it follows that
\begin{eqnarray}
S_1^*\leq 4H^2_1P\left(Z\in R_1^*, \Theta_{M_n}\in E^c, \sigma\leq \sigma_n\mid \bY_n\right)\leq 
4H^2_1(M_n-1)B_{M_n}, \nonumber\\
\label{eq:sum_1_2_bound}
\end{eqnarray}

\begin{eqnarray}
S_2^*&\leq& 4H^2_1P\left(Z\in R_1^*, \Theta_{M_n}\in E, \sigma\leq \sigma_n\mid \bY_n\right) \nonumber\\
&\leq& 4H^2_1\left(1-\frac{1}{M_n}\right)^n \left(\frac{\alpha_n+M_n}{\alpha_n}\right)^{M_n}, 
\label{eq:sum_2_2_bound}
\end{eqnarray}

\begin{eqnarray}
S_3^*\leq 4H^2_1P\left(Z\in (R_1^*)^c, \theta_i\in [-a-c, a+c]^c, \sigma\leq \sigma_n\mid \bY_n\right)\leq 4H^2_1B_{M_n}, \nonumber\\
\label{eq:sum_3_2_bound}
\end{eqnarray}

\begin{eqnarray}
S_5^* \leq 4H^2_1P\left(Z\in (R_1^*)^c, \theta_i\in [-a-c, a+c], \sigma\leq \sigma_n\mid \bY_n\right)\leq 4H^2_1\epsilon_{M_n}^*.\nonumber\\
\label{eq:sum_5_2_bound}
\end{eqnarray}

\begin{eqnarray}
&&S_4^*=\frac{1}{D}\int_{\Theta_{-iM_n}}\int_{\theta_i\in [-a-c,a+c]}\int_{\sigma\leq \sigma_n}\zeta_{M_n}^2 L(\Theta_{M_n},z, \bY_n)dH(\Theta_{M_n})dG_n(\sigma)\nonumber\\
&=& \left[\frac{\varphi(\theta_n^v(y),\sigma_n^v(y)+\hat k_n)}{(\sigma_n^v(y)+\hat k_n)}
\phi\left(\frac{y-\theta_n^v(y)}{\sigma_n^v(y)+\hat k_n}\right)-
E\left(\frac{\varphi(\theta_i,\sigma+\hat k_n)}{(\sigma+\hat k_n)}
\phi\left(\frac{y-\theta_i}{\sigma+\hat k_n}\right)\bigg\vert \bY_n\right)\right]\times \nonumber\\
&&\frac{1}{D}\int_{\Theta_{-iM_n}}\int_{\theta_i\in [-a-c,a+c]}\int_{\sigma\leq \sigma_n}\zeta_{M_n}
L(\Theta_{M_n},z, \bY_n)dH(\Theta_{M_n})dG_n(\sigma).  \label{eq:sum_4_2}
\end{eqnarray}

Let $R^{'}=$\{$\theta_i\in [-a-c,a+c]$, rest $\theta_{l}$'s are in $\mathbb S$, $\sigma\leq \sigma_n$\}.
Then 
we consider the following:
\begin{eqnarray}
&&\frac{1}{D}\int_{R^{'}}\zeta_{M_n}
L(\Theta_{M_n},z, \bY_n)dH(\Theta_{M_n})dG_n(\sigma) \nonumber\\
&=&\frac{1}{D}\int_{R^{'}}
\frac{\varphi(\theta_i,\sigma+\hat k_n)}{(\sigma+\hat k_n)}
\phi\left(\frac{y-\theta_i}{\sigma+\hat k_n}\right)L(\Theta_{M_n},z, \bY_n)dH(\Theta_{M_n})dG_n(\sigma) \nonumber\\
&& \ \ -E\left(\frac{\varphi(\theta_i,\sigma+\hat k_n)}{(\sigma+\hat k_n)}
\phi\left(\frac{y-\theta_i}{\sigma+\hat k_n}\right)\bigg\vert \bY_n\right)P\left(\theta_i\in [-a-c,a+c], \sigma\leq \sigma_n\vert  \bY_n\right)
\nonumber\\
&=& S_{1}^{''}+S_{2}^{''}, \hspace{2mm}\mbox{say}.\nonumber\\
\label{eq:eq:sum_4_3}
\end{eqnarray}

The terms $S_{1}^{''}$ and $S_{2}^{''}$ can be dealt with in the same way as $J_{1}^{''}$ and $J_{2}^{''}$ were 
handled in the corresponding 
EW case 
and it can be shown that
\begin{equation}
S_{4}^*=O\left(M_nB_{M_n}+\left(1-\frac{1}{M_n}\right)^n\left(\frac{\alpha_n+M_n}{\alpha_n}\right)^{M_n}+\epsilon_{M_n}^*\right).
\label{eq:sum_4_order}
\end{equation}

Thus, $\sum_{i=1}^{4}S_i^*=O\left(M_nB_{M_n}+\left(1-\frac{1}{M_n}\right)^n\left(\frac{\alpha_n+M_n}{\alpha_n}\right)^{M_n}+\epsilon_{M_n}^*\right)$.
Hence, the lemma follows.

\section{{\bf Proofs of results associated with Section 9 of MB}}

\subsection{{\bf EW case: Proof of Theorem 9.1 of MB}}
\begin{proof}
Note that
\begin{eqnarray}
&&E\left(\frac{\varphi(\theta_i,\sigma+\hat k_n)}{(\sigma+\hat k_n)}
\phi\left(\frac{y-\theta_i}{\sigma+\hat k_n}\right)\bigg\vert \bY_n\right) \nonumber\\
&=& \frac{1}{D}\int_{\sigma}\int_{\Theta_n}\frac{\varphi(\theta_i,\sigma+\hat k_n)}{(\sigma+\hat k_n)}
\phi\left(\frac{y-\theta_i}{\sigma+\hat k_n}\right)L(\Theta_n,\bY_n)dH(\Theta_n)dG_n(\sigma)  \nonumber\\ 
&\leq& H_1.  \label{param_ew_second}
\end{eqnarray}
As a result,
\begin{eqnarray}
&&\frac{1}{\alpha_n+n}\sum_{i=1}^{n}E\left(\frac{\varphi(\theta_i,\sigma+\hat k_n)}{(\sigma+\hat k_n)}
\phi\left(\frac{y-\theta_i}{\sigma+\hat k_n}\right)\bigg\vert \bY_n\right) \nonumber\\
&\leq& \frac{n}{\alpha_n+n} H_1 \nonumber\\
&\rightarrow& 0.
\label{eq:param_conv_second}
\end{eqnarray}
Now consider
\begin{eqnarray}
&\tilde A_n & \nonumber\\
&=& \frac{1}{D}\int_{\theta_{n+1}}\int_{\Theta_n}\int_{\sigma}\frac{\varphi(\theta_{n+1},\sigma+\hat k_n)}{(\sigma+\hat k_n)\sqrt{2\pi}}
e^{-\frac{\left(y-\theta_{n+1}\right)^2}{2(\sigma+\hat k_n)^2}}
L(\Theta_n,\bY_n)dG_0(\theta_{n+1})dH(\Theta_n)dG_n(\sigma)\nonumber\\
&=& \frac{1}{D}\int_{\theta_{n+1}}\int_{\Theta_n}\int_{\sigma<\sigma_n}\frac{\varphi(\theta_{n+1},\sigma+\hat k_n)}{(\sigma+\hat k_n)\sqrt{2\pi}}e^{-\frac{\left(y-\theta_{n+1}\right)^2}{2(\sigma+\hat k_n)^2}}
L(\Theta_n,\bY_n)dG_0(\theta_{n+1})dH(\Theta_n)dG_n(\sigma)\nonumber\\
&+& \frac{1}{D}\int_{\theta_{n+1}}\int_{\Theta_n}\int_{\sigma>\sigma_n}\frac{\varphi(\theta_{n+1},\sigma+\hat k_n)}{(\sigma+\hat k_n)\sqrt{2\pi}}e^{-\frac{\left(y-\theta_{n+1}\right)^2}{2(\sigma+\hat k_n)^2}}
L(\Theta_n,\bY_n)dG_0(\theta_{n+1})dH(\Theta_n)dG_n(\sigma)\nonumber\\
 &=& W_1+W_2\ \ \mbox{(say)}.\nonumber\\
 \label{eq:a_n_param}
\end{eqnarray}

\begin{eqnarray}
W_2 &=& \frac{1}{D}\int_{\theta_{n+1}}\int_{\Theta_n}\int_{\sigma>\sigma_n}
\frac{\varphi(\theta_{n+1},\sigma+\hat k_n)}{(\sigma+\hat k_n)\sqrt{2\pi}}e^{-\frac{\left(y-\theta_{n+1}\right)^2}{2(\sigma+\hat k_n)^2}}
L(\Theta_n,\bY_n)dG_0(\theta_{n+1})dH(\Theta_n)dG_n(\sigma)\nonumber\\
&\leq& H_1 P\left(\sigma> \sigma_n\vert\bY_n\right). \nonumber
\end{eqnarray}
Thus, by Lemma \ref{lemma:epsilon_ew}, 
\begin{equation}
W_2=O(\epsilon^*_n).
\label{eq:a_n_second} 
\end{equation}
As regards $W_1$, an application of $GMVT$ yields
\begin{eqnarray}
W_1 &=& \frac{1}{D}\int_{\theta_{n+1}}\int_{\Theta_n}\int_{\sigma<\sigma_n}
\frac{\varphi(\theta_{n+1},\sigma+\hat k_n)}{(\sigma+\hat k_n)\sqrt{2\pi}}
e^{-\frac{\left(y-\theta_{n+1}\right)^2}{2(\sigma+\hat k_n)^2}}
L(\Theta_n,\bY_n)dG_0(\theta_{n+1})dH(\Theta_n)dG_n(\sigma)\nonumber\\
&=& \int_{\theta_{n+1}}\frac{\varphi(\theta_{n+1},\sigma_n^{*}(y)+\hat k_n)}{(\sigma_n^{*}(y)+\hat k_n)\sqrt{2\pi}}
e^{-\frac{\left(y-\theta_{n+1}\right)^2}{2(\sigma_n^{*}(y)+\hat k_n)^2}}dG_0(\theta_{n+1})\times
P\left(\sigma< \sigma_n\vert\bY_n\right),\nonumber
\end{eqnarray}
%
DCT ensures that
\begin{eqnarray}
\int_{\theta_{n+1}}\frac{\varphi(\theta_{n+1},\sigma_n^{*}(y)+\hat k_n)}{(\sigma_n^{*}(y)+\hat k_n)\sqrt{2\pi}}
e^{-\frac{\left(y-\theta_{n+1}\right)^2}{2(\sigma_n^{*}(y)+\hat k_n)^2}}dG_0(\theta_{n+1}) \nonumber\\
\rightarrow \int_{\theta_{n+1}}\frac{\varphi(\theta_{n+1},k)}{k\sqrt{2\pi}}e^{-\frac{\left(y-\theta_{n+1}\right)^2}{2k^2}}dG_0(\theta_{n+1}). \label{eq:dct_W_1}
\end{eqnarray}
It then follows from (\ref{eq:dct_W_1}) and the fact that $P\left(\sigma< \sigma_n\vert\bY_n\right)\rightarrow 1$, that
\begin{eqnarray}
W_1\rightarrow \int_{\theta_{n+1}}\frac{\varphi(\theta_{n+1},k)}{k\sqrt{2\pi}}e^{-\frac{\left(y-\theta_{n+1}\right)^2}{2k^2}}dG_0(\theta_{n+1}). \label{eq:W_1_converge}
\end{eqnarray}
Finally, (\ref{eq:a_n_second}) and (\ref{eq:W_1_converge}) guarantee Theorem 12.1. 
\end{proof}

\subsection{{\bf SB case: Proofs of results associated with Section 9.2 of MB}}

\begin{lemma}
\label{lemma:param}
Let $\{r_n\}$ be a sequence tending to zero such that 
$O\left(\log\left(\frac{1}{\epsilon_n}\right)\right)\prec -n\log(r_n)-n\log(n)$,
and let $C_n=O\left(\frac{1}{r^s_n n^2}\right)$; $s>2$.   
Then
\begin{equation}
P(Z\in R_1^*, \Theta_{M_n}\in E, \sigma\leq \sigma_n\vert \bY_n)\gtrsim \left(1-\frac{1}{M_n}\right)^n\left(\frac{\alpha_n+M_n}{\alpha_n}\right)^{M_n}.
\label{eq:sb_param1}
\end{equation}
\end{lemma}

\begin{proof}

\begin{eqnarray}
P(Z\in R_1^*, \Theta_{M_n}\in E, \sigma\leq \sigma_n\vert \bY_n)  
&=& P(Z\in R_1^*, \Theta_{M_n}\in E\vert \bY_n) \nonumber\\
&&\ \ -P(Z\in R_1^*, \Theta_{M_n}\in E, \sigma> \sigma_n\vert \bY_n) \label{eq:prob_split}
\end{eqnarray}

We first obtain a lower bound for $P(Z\in R_1^*, \Theta_{M_n}\in E\vert \bY_n)$. 
\begin{eqnarray}
&&P\left(Z\in R_1^*, \Theta_{M_n}\in E\vert \bY_n\right)  \nonumber\\ 
&=& \dfrac{\sum_{z\in R_1^*}\int_{\sigma\leq \sigma_n}\int_{\Theta_{M_n}\in E} 
L(\Theta_{M_n}, z, \bY_n)dG_n(\sigma)dH(\Theta_{M_n})}{\sum_{z}\int_{\sigma}\int_{\Theta_{M_n}} L(\Theta_{M_n}, z, \bY_n) 
dG_n(\sigma)dH(\Theta_{M_n})} \nonumber\\
&=& \frac{\sum_{z\in R_1^*}N}{\sum_{z} D}, \label{eq:split_prob_param}
\end{eqnarray}
where $N=\int_{\sigma\leq \sigma_n}\int_{\Theta_{z}\in E}\int L(\Theta_{M_n}, z, \bY_n)dG_n(\sigma)dH(\Theta_{M_n})$ and \\
$D=\int_{\sigma}\int_{\Theta_{M_n}} L(\Theta_{M_n}, z, \bY_n) dG_n(\sigma)dH(\Theta_{M_n})$. \\

Let $k_n$ be a sequence of constants such that $k_n\rightarrow \infty$ as $n\rightarrow\infty$. For
$1< d< M_n$, where $d$ stands for the number of $\theta_j$'s associated with the likelihood for 
a given $z$, denote $E^* = \{\theta_j\in [-a-c ,a+c]\cap [\bar y_j-k_n, \bar y_j+k_n], j=1,\ldots,d; 
\theta_j\in (-\infty, \infty), j=d+1,\ldots,M_n\}$
and for $j=1,\ldots,d$, let $E_j=\{\theta_j\in [-a-c ,a+c]\cap [\bar y_j-k_n, \bar y_j+k_n]\}$.\\ 
Note that,
\begin{eqnarray}
N&\geq& \int_{\Theta_{M_n}\in E^{*}}\int_{nk_n}^{2nk_n}\frac{1}{\sigma^n}
e^{-\dfrac{\sum_{j=1}^{d}\sum_{t:z_t=j}(y_t-\bar y_j)^2}{2\sigma^2}} \nonumber\\
&&\ \ \times e^{-\dfrac{\sum_{j=1}^{d}n_j(\bar y_j-\theta_j)^2}{2\sigma^2}}dH(\Theta_{M_n})dG_n(\sigma)\nonumber\\
&\geq&\left(\frac{1}{2nk_n}\right)^{n}e^{-\frac{C_n^{(2)}}{8n^2k_n^2}}\times e^{-\frac{1}{2n}}\times \left(\frac{\alpha_n}{\alpha_n+M_n}\right)^{M_n}\times
\prod_{j=1}^{d}G_0\left(E_j\right)\times O(\epsilon_n), \nonumber\\
\label{eq:param_N}
\end{eqnarray}
assuming $\int_{nk_n}^{2nk_n}dG_n(\sigma)=O(\epsilon_n)$.
In the above, $C^{(2)}_n=\sup_{z\in R^*_1}\sum_{j=1}^{M_n}\sum_{t:z_t=j}(Y_t-\bar Y_j)^2$, 
as defined before in the proof of Lemma \ref{lemma:sb_lemma9}.\\

Since $k_n\rightarrow \infty$, as $n\rightarrow \infty$, 
$G_0\left(E_j\right)\sim G_0\left([-a-c, a+c]\right)=H_0$ (say), and
\begin{eqnarray}
N\gtrsim \left(\frac{1}{2nk_n}\right)^{n}e^{-\frac{C_n^{(2)}}{8n^2k_n^2}}\times e^{-\frac{1}{2n}}\times 
\left(\frac{\alpha_n}{\alpha_n+M_n}\right)^{M_n}\times H^{M_n}_0\times O(\epsilon_n). \label{eq:param_N1}
\end{eqnarray}

To obtain an upper bound of $D$ note that,
\begin{eqnarray}
e^{-\dfrac{\sum_{j=1}^{M_n}\sum_{t:z_t=j}(y_t-\bar y_j)^2}{2\sigma^2}}\times \frac{1}{\sigma^n}
e^{-\dfrac{\sum_{j=1}^{M_n}n_j(\bar y_j-\theta_j)^2}{2\sigma^2}}\leq 1\times \left(\frac{n}{C_n^{(1)}}\right)^{\frac{n}{2}}e^{-\frac{n}{2}},
\label{eq:D_N1}
\end{eqnarray}
for $0<\sigma<\infty$, where $C^{(1)}_n=\inf_{z\in R^*_1}\sum_{j=1}^{M_n}\sum_{t:z_t=j}(Y_t-\bar Y_j)^2$,
as defined in the proof of Lemma \ref{lemma:sb_lemma9}. This implies
\begin{eqnarray}
D\leq \left(\frac{n}{C_n^{(1)}}\right)^{\frac{n}{2}}e^{-\frac{n}{2}}.  \label{eq:deno_param1}
\end{eqnarray}

Since $C_n^{(1)}\sim C_n^{(2)}\sim C_n$ for large $n$, 
let us obtain the condition under which 
\begin{equation}
\left(\frac{1}{2nk_n}\right)^{n}\times e^{-\frac{C_n}{8n^2k_n^2}}\times H^{M_n}_0
\times O(\epsilon_n)\geq \left(\frac{n}{C_n}\right)^{\frac{n}{2}}
e^{-\frac{n}{2}}
\label{eq:param_cond1}
\end{equation}

Let $k_n=r_nC_n$, where $r_n\rightarrow 0$ and $r_nC_n\rightarrow \infty$.\\
Then,
\begin{eqnarray}
&&\left(\frac{1}{2nk_n}\right)^{n}\times e^{-\frac{C_n}{8n^2k_n^2}}\times H^{M_n}_0\times O(\epsilon_n)\geq \left(\frac{n}{C_n}\right)^{\frac{n}{2}}
e^{-\frac{n}{2}} \nonumber\\
&\Leftrightarrow& \left(\frac{1}{2nr_nC_n}\right)^{n}\times e^{-\frac{C_n}{8n^2r_n^2C_n^2}}
\times H^{M_n}_0\times O(\epsilon_n)\geq 
\left(\frac{n}{C_n}\right)^{\frac{n}{2}}e^{-\frac{n}{2}} \nonumber\\
&\Leftrightarrow& -n\log\left(C_n\right)+\frac{n}{2}\log\left(C_n\right)-\frac{C_n}{8n^2r_n^2C_n^2} \nonumber\\
&&\hspace{3cm}\geq\frac{n}{2}\log(n)-\frac{n}{2}+n\log(2n)+n\log(r_n)
-M_n\log(H_0)-O\left(\log(\epsilon_n)\right)\nonumber\\
&\Leftrightarrow& \frac{n}{2}\log(C_n)+\frac{1}{8n^2r_n^2C_n} \nonumber\\
&&\ \ \leq-\frac{n}{2}\log(n)+\frac{n}{2}-n\log(2)-n\log(n)-n\log(r_n)+M_n\log(H_0)-
O\left(\log\left(\frac{1}{\epsilon_n}\right)\right)  \nonumber\\
&\Leftrightarrow& \frac{n}{2}\log(C_n)+\frac{1}{8n^2r_n^2C_n} \nonumber\\
&&\ \ \leq n\left(-\frac{3}{2}\log(n)+\frac{1}{2}-\log(2)\right)-n\log(r_n)+M_n\log(H_0)-
O\left(\log\left(\frac{1}{\epsilon_n}\right)\right).  \nonumber\\
\label{eq:param_cond2}
\end{eqnarray}

In the R.H.S of the inequality (\ref{eq:param_cond2}) 
we can choose $r_n$ sufficiently small such that term 
$-n\log(r_n)\sim n\left(-\frac{3}{2}\log(n)+\frac{1}{2}-\log(2)\right)-n\log(r_n)+M_n\log(H_0)-
O\left(\log\left(\frac{1}{\epsilon_n}\right)\right).$\\

%

Let $C_n=\frac{1}{r_n^sn^{2}}$, where $s> 2$.\\

Then $k_n=r_nC_n=\frac{1}{r_n^{s-1}n^{2}}\rightarrow \infty$, for $r_n$ going to zero at a sufficiently
fast rate.\\

Also, $n^2r_n^2C_n=\frac{n^2r_n^2}{r_n^s n^{2}}=\frac{1}{r_n^{s-2}}\rightarrow \infty$, for $s > 2$. \\

And,
\begin{equation}
\frac{n}{2}\log(C_n) = -\frac{ns}{2}\log(r_n)-n\log(n)< -n\log(r_n)-n\log(n)<-n\log(r_n). \label{eq:param_cond4}
\end{equation}

So, for $C_n = O\left(\frac{1}{r_n^{s}n^{2}}\right)$; $s > 2$, if $r_n$ is fixed to be sufficiently small such that 
for large $n$, $\frac{1}{8n^2r_n^2C_n}\approx 0$, and 
$n\left(-\frac{3}{2}\log(n)+\frac{1}{2}-\log(2)\right)+
M_n\log(H_0)+O\left(\log\left(\frac{1}{\epsilon_n}\right)\right)\prec -n\log(r_n)$, 
%
then, as $n\rightarrow \infty$, (\ref{eq:param_cond2}) holds, and 
\begin{eqnarray}
\left(\frac{1}{2nk_n}\right)^{n}\times e^{-\frac{C_n}{8n^2k_n^2}}\times O(\epsilon_n)\gtrsim \left(\frac{n}{C_n}\right)^{\frac{n}{2}}
e^{-\frac{n}{2}}. \nonumber
\end{eqnarray}

Hence, it follows that
\begin{eqnarray}
&&P\left(Z\in R_1^*, \Theta_{M_n}\in E\vert \bY_n\right) \nonumber\\
&\gtrsim& \dfrac{\left(M_n-1\right)^{n}\left(\frac{1}{2nk_n}\right)^{n}e^{-\frac{C_n^{(2)}}{8n^2k_n^2}}\times e^{-\frac{1}{2n}}\times \left(\frac{\alpha_n}{\alpha_n+M_n}\right)^{M_n}
\times O(\epsilon_n)}{M_n^{n}\left(\frac{n}{C_n^{(1)}}\right)^{\frac{n}{2}}e^{-\frac{n}{2}}}  \nonumber\\
&\gtrsim& \left(\frac{\alpha_n}{\alpha_n+M_n}\right)^{M_n}
\left(1-\frac{1}{M_n}\right)^{n}. \label{eq:lower_param}
\end{eqnarray}
\\

Now we obtain an upper bound for $P\left(Z\in R_1^*, \Theta_{M_n}\in E, \sigma> \sigma_n\vert \bY_n\right)$.

\begin{eqnarray}
&&P\left(Z\in R_1^*, \Theta_{M_n}\in E, \sigma> \sigma_n\vert \bY_n\right)  \nonumber\\
&=&\dfrac{\sum_{z\in R_1^*}\int_{\sigma_n}^{\infty}\int_{\Theta_{z}\in E}\int 
L(\Theta_{M_n}, z, \bY_n)dG_n(\sigma)dH(\Theta_{M_n})}{\sum_{z}\int_{\sigma}\int_{\Theta_{M_n}} L(\Theta_{M_n}, z, \bY_n) 
dG_n(\sigma)dH(\Theta_{M_n})} \nonumber\\
&=& \frac{\sum_{z\in R_1^*}N}{\sum_{z} D}  \nonumber\\
&\leq& \frac{\sum_{z\in R_1^*}N}{\sum_{z\in R_1^*} D}, \label{eq:param_up1}
\end{eqnarray}
where $N=\int_{\sigma_n}^{\infty}\int_{\Theta_{z}\in E}\int L(\Theta_{M_n}, z, \bY_n)dG_n(\sigma)dH(\Theta_{M_n})$ and \\
$D=\int_{\sigma}\int_{\Theta_{M_n}} L(\Theta_{M_n}, z, \bY_n) dG_n(\sigma)dH(\Theta_{M_n})$.

Note that, in the same as we have obtained equation (\ref{eq:D_N1}), it can shown that
\begin{equation}
N\leq \left(\frac{n}{C_n^{(1)}}\right)^{\frac{n}{2}} e^{-\frac{n}{2}}\times O(\epsilon_n).  \label{eq:param_up2}
\end{equation}

\begin{eqnarray}
D&\geq&  \int_{\sigma> \sigma_n}\int_{\Theta_{M_n}} L(\Theta_{M_n}, z, \bY_n) dG_n(\sigma)dH(\Theta_{M_n}) \nonumber\\
&\geq& \left(\frac{1}{2nk_n}\right)^{n} e^{-\frac{C_n^{(2)}}{8n^2k_n^2}}\times e^{-\frac{1}{2n}} \left(\frac{\alpha_n}{\alpha_n+M_n}\right)^{M_n}\times
\prod_{j=1}^{d}G_0\left(E_j\right) \times O(\epsilon_n)  \nonumber\\
&\gtrsim & \left(\frac{1}{2nk_n}\right)^{n} e^{-\frac{C_n^{(2)}}{8n^2k_n^2}}\times e^{-\frac{1}{2n}} \left(\frac{\alpha_n}{\alpha_n+M_n}\right)^{M_n}\times H^{M_n}_0\times
O(\epsilon_n),
\label{eq:param_up3}
\end{eqnarray}
since for large $n$, $G_0\left(E_j\right)\sim H_0$.\\ 

Since $C_n^{(2)}\sim C_n^{(2)}\sim C_n$ it follows that 
\begin{eqnarray}
\frac{\sum_{z\in R_1^*}N}{\sum_{z} D}\lesssim \dfrac{(M_n-1)^n\left(\frac{n}{C_n}\right)^{\frac{n}{2}} e^{-\frac{n}{2}}
\times O(\epsilon_n)}{M^n_n\left(\frac{1}{2nk_n}\right)^{n} e^{-\frac{C_n}{8n^2k_n^2}}\times e^{-\frac{1}{2n}} 
\left(\frac{\alpha_n}{\alpha_n+M_n}\right)^{M_n}\times H^{M_n}_0\times O(\epsilon_n)}.
\end{eqnarray}

Choose $C_n$ such that
\begin{equation}
\left(\frac{1}{2nk_n}\right)^{n} e^{-\frac{C_n}{8n^2k_n^2}}\times H^{M_n}_0\times O(\epsilon_n)\gtrsim \left(\frac{n}{C_n}\right)^{\frac{n}{2}}e^{-\frac{n}{2}},
\end{equation}
which is exactly the same condition as in the last case of the lower bounds. So, as $n\rightarrow \infty$,
\begin{equation}
P\left(Z\in R_1^*, \Theta_{M_n}\in E, \sigma\geq \sigma_n\right)\lesssim 
\left(\frac{\alpha_n+M_n}{\alpha_n}\right)^{M_n}\left(1-\frac{1}{M_n}\right)^{n}\times 
O(\epsilon_n).  \label{eq:param_upper}
\end{equation}

Hence,
\begin{eqnarray}
&&P\left(Z\in R_1^*, \Theta_{M_n}\in E, \sigma\leq \sigma_n\vert \bY_n\right) \nonumber\\
&=& P\left(Z\in R_1^*, \Theta_{M_n}\in E, \sigma\leq \sigma_n\vert \bY_n\right)-
P\left(Z\in R_1^*, \Theta_{M_n}\in E, \sigma> \sigma_n\vert \bY_n\right)\nonumber\\
&\gtrsim& \left[\left(\frac{\alpha_n}{\alpha_n+M_n}\right)^{M_n}-\left(\frac{\alpha_n+M_n}{\alpha_n}\right)^{M_n}\times O(\epsilon_n)
\right]
\left(1-\frac{1}{M_n}\right)^{n}.\nonumber\\
\label{eq:param_final}
\end{eqnarray}

We must have
\begin{eqnarray}
&&\left(\frac{\alpha_n}{\alpha_n+M_n}\right)^{M_n} \gtrsim \left(\frac{\alpha_n+M_n}{\alpha_n}\right)^{M_n}\times O(\epsilon_n)\nonumber\\
&\Leftrightarrow& O(\epsilon_n) \lesssim \left(\frac{\alpha_n}{\alpha_n+M_n}\right)^{2M_n}. \label{eq:param_cond6}
\end{eqnarray}

Using L' Hospital's rule it can be shown that $\left(\frac{\alpha_n}{\alpha_n+M_n}\right)^{2M_n}\rightarrow 1$, if $M_n=n^b$, $\alpha_n=n^{\omega}$,
$\omega> 0$, $b> 0$ and $\omega-b> b$. Since $\epsilon_n\rightarrow 0$,
for large $n$, (\ref{eq:param_cond6}) holds and does not contradict the assumptions regarding $\epsilon_n$.
Actually, for the above choices, we have 
$\left(\frac{\alpha_n}{\alpha_n+M_n}\right)^{M_n}\sim 
\left(\frac{\alpha_n}{\alpha_n+M_n}\right)^{M_n}-\left(\frac{\alpha_n+M_n}{\alpha_n}\right)^{M_n}
\times O(\epsilon_n)$.\\

Summing up all the results we have,
\begin{equation}
P\left(Z\in R_1^*, \Theta_{M_n}\in E, \sigma\leq \sigma_n\vert \bY_n\right) \gtrsim \left(\frac{\alpha_n}{\alpha_n+M_n}\right)^{M_n}
\left(1-\frac{1}{M_n}\right)^{n},  \label{eq:param_cond7}
\end{equation}
for $C_n=O\left(\frac{1}{r_n^sn^2}\right)$, $s > 2$.\\

\end{proof}

\subsubsection{{\bf Proof of Theorem 9.2 of MB}}

Consider the integral
\begin{eqnarray}
&&\int_{\Theta_{M_n}\in E}\int_{0}^{\sigma_n} \frac{\varphi(\theta_i,\sigma+\hat k_n)}{(\sigma+\hat k_n)}
\phi\left(\frac{y-\theta_i}{\sigma+\hat k_n}\right)
L(\Theta_{M_n}, z, \bY_n)dH(\Theta_{M_n})dG_n(\sigma)  \nonumber\\
&=&  \frac{\alpha_n}{\alpha_n+M_n-1}\int_{\Theta_{M_n}\in E}\int_{0}^{\sigma_n}
\frac{\varphi(\theta_i,\sigma+\hat k_n)}{(\sigma+\hat k_n)}\phi\left(\frac{y-\theta_i}{\sigma+\hat k_n}\right)
\nonumber\\
&&\hspace{3cm} \times L(\Theta_{M_n}, z, \bY_n)dG_0(\theta_{i})dH_{-i}(\Theta_{-iM_n})dG_n(\sigma) \nonumber\\
&&\ \ +\frac{1}{\alpha_n+M_n-1}\sum_{j=1, j\neq i}^{M_n}\int_{\Theta_{-iM_n}\in E_{-i}}\int_{0}^{\sigma_n}
\frac{\varphi(\theta_i,\sigma+\hat k_n)}{(\sigma+\hat k_n)}\phi\left(\frac{y-\theta_j}{\sigma+\hat k_n}\right) \nonumber\\
&&\hspace{3cm} \times L(\Theta_{-iM_n}, z, \bY_n)dH_{-i}(\Theta_{-iM_n})dG_n(\sigma),  \nonumber\\
\label{eq:param_cond8}
\end{eqnarray}
using the Polya urn representation of $H(\Theta_{M_n})$, given by
\begin{equation}
H(\Theta_{M_n})=\left[\frac{\alpha_n}{\alpha_n+M_n-1}G_0(\theta_i)+\frac{1}{\alpha_n+M_n-1}\sum_{j=1, j\neq i}^{M_n}\delta_{\theta_j}(\theta_i)\right]\times
H_{-i}(\Theta_{-iM_n}),
\label{eq:param_dirich}
\end{equation}
where $\Theta_{-iM_n}=\Theta_{M_n}\setminus \theta_i$ and $H_{-i}(\Theta_{-iM_n})$ is the joint distribution of $\Theta_{-iM_n}$ 
and $E_{-i}$ is the set $E$ excluding $\theta_i$.

Let $D = \sum_{z}\int_{\sigma}\int_{\Theta_{M_n}} L(\Theta_{M_n}, z, \bY_n) dG_n(\sigma)dH(\Theta_{M_n})$. Note that
for $z\in R_1^*$, $\theta_i$ is not present in likelihood and hence 
$\{\Theta_{M_n}\in E\} = \{\theta_i\in\mathbb S\}\cap \{\Theta_{-iM_n}\in E_{-i}\}$.  
Then,
\begin{eqnarray}
&&\frac{1}{D}\sum_{z\in R_1^*}\int_{\Theta_{M_n}\in E}\int_{0}^{\sigma_n}
\frac{\varphi(\theta_i,\hat k_n)}{(\sigma+\hat k_n)}\phi\left(\frac{y-\theta_i}{\sigma+\hat k_n}\right)
\nonumber\\
&&\hspace{3cm} \times L(\Theta_{M_n}, z, \bY_n)dG_0(\theta_{i})dH_{-i}(\Theta_{-iM_n})dG_n(\sigma) \nonumber\\
&=& \frac{1}{D}\sum_{z\in R_1^*}\int_{\Theta_{-iM_n}\in E_{-i}}\int_{\theta_i\in\mathbb S}\int_{0}^{\sigma_n}
\frac{\varphi(\theta_i,\hat k_n)}{(\sigma+\hat k_n)}
\phi\left(\frac{y-\theta_i}{\sigma+\hat k_n}\right) dG_0(\theta_{i}) \nonumber\\
&&\hspace{3cm} \times L(\Theta_{M_n}, z, \bY_n)dH_{-i}(\Theta_{-iM_n})dG_n(\sigma) \nonumber\\
&=&\int_{\theta_i\in\mathbb S}\frac{\varphi(\theta_i,\sigma_n^*(y)+\hat k_n)}{(\sigma_n^*(y)+\hat k_n)}
\phi\left(\frac{y-\theta_i}{\sigma_n^*(y)+\hat k_n}\right)dG_0(\theta_{i}) \nonumber\\
&&\hspace{3cm} \times P_{M_n-1}\left(Z\in R_1^*, \Theta_{-iM_n}\in E_{-i}, \sigma\leq \sigma_n\big\vert \bY_n\right), \label{eq:param_cond9}
\end{eqnarray}
where $P_{M_n-1}(\cdot\vert \bY_n)$ is the posterior probability when the mixture model has $M_n-1$ components.

It can be shown that exactly under the same conditions as in Lemma \ref{lemma:param},\\
$P_{M_n-1}\left(Z\in R_1^*, \Theta_{-iM_n}\in E_{-i}, \sigma\leq \sigma_n\big\vert \bY_n\right)$ has the same lower bound with only $M_n$ 
replaced with $M_n-1$. Using L' Hospital's rule it can be easily shown that 
$P_{M_n-1}\left(Z\in R_1^*, \Theta_{-iM_n}\in E^*, \sigma\leq \sigma_n\big\vert \bY_n\right)$ also converges to 1.

DCT ensures that
\begin{eqnarray}
\int_{\theta_i\in\mathbb S}\frac{\varphi(\theta_i,\sigma_n^*(y)+\hat k_n)}{(\sigma_n^*(y)+\hat k_n)}
\phi\left(\frac{y-\theta_i}{\sigma_n^*(y)+\hat k_n}\right)dG_0(\theta_{i})\rightarrow 
\int_{\theta_{i}\in\mathbb S}\frac{\varphi(\theta_i,k)}{k\sqrt{2\pi}}
e^{-\frac{\left(y-\theta_{i}\right)^2}{2k^2}}dG_0(\theta_{i}),\nonumber\\
\label{eq:converge_sb}
\end{eqnarray}
almost surely.

Again,
\begin{eqnarray}
&&\frac{1}{\alpha_n+M_n-1}\frac{1}{D}\sum_{z\in R_1^*}\sum_{j=1, j\neq i}^{M_n}\int_{\Theta_{-iM_n}\in E_{-i}}\int_{0}^{\sigma_n}
\frac{\varphi(\theta_j,\sigma+\hat k_n)}{(\sigma+\hat k_n)}\phi\left(\frac{y-\theta_j}{\sigma+\hat k_n}\right) \nonumber\\
&&\hspace{5cm} \times L(\Theta_{-iM_n}, z, \bY_n)dH_{-i}(\Theta_{-iM_n})dG_n(\sigma)  \nonumber\\
&\leq& H_1\times \frac{1}{\alpha_n+M_n-1} \nonumber\\
&&\ \ \times\sum_{j=1, j\neq i}^{M_n}\frac{1}{D}\sum_{z\in R_1^*}\int_{\Theta_{-iM_n}\in E_{-i}}\int_{0}^{\sigma_n}
L(\Theta_{-iM_n}, z, \bY_n)dH_{-i}(\Theta_{-iM_n})dG_n(\sigma)  \nonumber\\
&\leq& H_1\times  \frac{M_n-1}{\alpha_n+M_n-1}. \nonumber\\
\label{eq:param_alpha_n}
\end{eqnarray}
Note that for $\alpha_n\succ O(M_n)$, $\frac{\alpha_n}{\alpha_n+M_n-1}\rightarrow 1$ and $\frac{M_n-1}{\alpha_n+M_n-1}\rightarrow 0$.

From (\ref{eq:converge_sb}) and (\ref{eq:param_alpha_n}) we conclude that, almost surely,
\begin{eqnarray}
&&\frac{1}{D}\int_{\Theta_{M_n}\in E}\int_{0}^{\sigma_n} \frac{\varphi(\theta_i,\sigma+\hat k_n)}{(\sigma+\hat k_n)}
\phi\left(\frac{y-\theta_i}{\sigma+\hat k_n}\right)
L(\Theta_{M_n}, z, \bY_n)dH(\Theta_{M_n})dG_n(\sigma)  \nonumber\\
&&\rightarrow  \int_{\theta_{i}}\frac{\varphi(\theta_i,k)}{k\sqrt{2\pi}}
e^{-\frac{\left(y-\theta_{i}\right)^2}{2k^2}}dG_0(\theta_{i}),
\label{eq:converge_param}
\end{eqnarray}
as $n\rightarrow \infty$.
The result then follows by boundedness of $E\left(\hat f_{SB}(y|\Theta_{M_n},\sigma)\bigg\vert\bY_n\right)$.

\section{{\bf Overview of asymptotic calculations associated with Section 10 of MB}}

It is easy to see that the upper bounds of the probabilites given in 
Lemmas \ref{lemma:sb_epsilon}--\ref{lemma:sb_lemma10} remain 
the same for this modified model.
For the modified SB model 
the likelihood function $L(\Theta_{M_n}, z, \bY_n, \Pi)$ 
is given by
\begin{eqnarray}
L(\Theta_{M_n}, z, \bY_n, \Pi)=\prod_{\ell=1}^{M_n} \pi_{\ell}^{n_{\ell}+\beta_{\ell}-1}
\prod_{j=1}^{M_n}\frac{1}{\sigma^{n_j}}e^{-\frac{1}{2}\sum_{t:z_t=j}\left(\frac{Y_{t}-\theta_j}{\sigma}\right)^2}.
\label{eq:sb_mod_like}
\end{eqnarray}
From the form (\ref{eq:sb_mod_like}) it is clear that given $z$, the posterior of $\Pi$ is independent of $\Theta_{M_n}$. 
Hence, it is easy to see that, the same calculations as in Lemma \ref{lemma:sb_epsilon} 
yield the following bounds for the modified model:
\begin{eqnarray}
N\leq \frac{1}{(\sigma_n)^n}P(\sigma> \sigma_n)
\sum_{z} \int_{\Pi}\prod_{\ell=1}^{M_n} \pi_{\ell}^{n_{\ell}+\beta_{\ell}-1}d\Pi \nonumber
\end{eqnarray}
and
\begin{eqnarray}
D\geq \dfrac{\exp\left(\frac{-n(a+c_1)^2}{2(b_n)^2}\right)}{(b_n)^n}P(b_n<\sigma\leq \sigma_n)\left(\frac{\alpha_n}{\alpha_n+M_n}\right)^{M_n} H_0^{M_n}
\sum_{z} \int_{\Pi}\prod_{\ell=1}^{M_n} \pi_{\ell}^{n_{\ell}+\beta_{\ell}-1}d\Pi.
\nonumber
\end{eqnarray}

Hence, the upper bound for $P\left(\sigma> \sigma_n\vert \bY_n\right)$ does not change. Similarly, the same argument shows that the upper 
bounds remain the same for all the probabilities except for $P(Z\in R_1^*, \Theta_{M_n}\in E, \sigma\leq \sigma_n\vert \bY_n)$.
For $P(Z\in R_1^*, \Theta_{M_n}\in E, \sigma\leq \sigma_n\vert \bY_n)$, the same calculations as in 
Lemma \ref{lemma:sb_lemma9} show that
\begin{eqnarray}
N&\leq& \left(\frac{1}{\sigma_n}\right)^ne^{-
\frac{C_n^{(1)}}{2\sigma_n^2}}\times G_0([-a-c, a+c])\times O(1-\epsilon_n) \nonumber\\
&& \hspace{5cm} \times\sum_{z\in R_1^*} \int_{\Pi}\prod_{\ell=1}^{M_n} \pi_{\ell}^{n_{\ell}+\beta_{\ell}-1}d\Pi. \nonumber\\
\label{eq:N_new}
\end{eqnarray}
Similarly,
\begin{eqnarray}
D &\geq& \left(\frac{1}{2nk_n}\right)^n\times e^{-\frac{C_n^{(2)}}{8n^2k_n^2}}\times e^{-\frac{1}{2n}}\times 
\left(\frac{\alpha_n}{\alpha_n+M_n}\right)^{M_n}
\nonumber\\
&&\hspace{10mm}\times\prod_{j=1}^{d}G_0([\bar Y_j-k_n, \bar Y_j+k_n]\cap\mathbb S)\times O(\epsilon_n) \nonumber\\
&&\hspace{20mm}\times\sum_{z} \int_{\Pi}\prod_{\ell=1}^{M_n} \pi_{\ell}^{n_{\ell}+\beta_{\ell}-1}d\Pi.
\label{eq:D_new}
\end{eqnarray}


Since 
\begin{equation}
\frac{\sum_{z\in R^*_1} \int_{\Pi}\prod_{\ell=1}^{M_n} \pi_{\ell}^{n_{\ell}+\beta_{\ell}-1}d\Pi}
{\sum_{z} \int_{\Pi}\prod_{\ell=1}^{M_n} \pi_{\ell}^{n_{\ell}+\beta_{\ell}-1}d\Pi}
= P(Z\in R^*_1) = \frac{(M_n-1)^n}{M_n^n},\nonumber
\end{equation}
the upper bound remains the same as before.\\

It can also be shown that $E\left(\hat f^*_{SB}(y\mid \Theta_{M_n}, \Pi,\sigma)\right)$ converges to $\frac{1}{k}\phi\left(\frac{y-\theta^*(y)}{k}\right)$.
We will split the expectation in the same way as in the proof of Theorem 5.2 
into $S_1, S_2, S_3, S_4, S_5$, with the integrand
$\frac{\varphi(\theta_i,\sigma+\hat k_n)}{(\sigma+\hat k_n)}\phi\left(\frac{y-\theta_i}{\sigma+\hat k_n}\right)$ replaced with 
$\sum_{i=1}^{M_n}\pi_i\frac{\varphi(\theta_i,\sigma+\hat k_n)}{(\sigma+\hat k_n)}
\phi\left(\frac{y-\theta_i}{\sigma+\hat k_n}\right)$. The upper bounds of $S_i$ for $i\neq 4$,
will be same. We illustrate this with $S_1$; for the others the same arguments will hold. 
\begin{eqnarray}
S_1 &=& \frac{1}{D}\sum_{R_1^*}\int_{\Pi}\int_{I_1}\sum_{i=1}^{M_n}\pi_i\frac{\varphi(\theta_i,\sigma+\hat k_n)}{(\sigma+\hat k_n)}\phi\left(\frac{y-\theta_i}{\sigma+\hat k_n}\right) 
\nonumber\\
&& \hspace{1cm} \times L(\Theta_{M_n}, z, \bY_n, \Pi)dH(\Theta_{M_n})dG_n(\sigma)d\Pi \nonumber\\
&\leq& H_1 \frac{1}{D}\sum_{R_1^*}\int_{\Pi}\int_{I_1}\sum_{i=1}^{M_n}\pi_i L(\Theta_{M_n}, z, \bY_n, \Pi)dH(\Theta_{M_n})dG_n(\sigma)d\Pi \nonumber\\
&=& H_1 \frac{1}{D}\sum_{R_1^*}\int_{I_1} L(\Theta_{M_n}, z, \bY_n)dH(\Theta_{M_n})dG_n(\sigma), \nonumber
\end{eqnarray}
using the fact that $\sum_{i=1}^{M_n}\pi_i= 1$. 
Hence $S_1$ has same order as $P(Z\in R_1^*, \Theta_{M_n}\in E, \sigma\leq \sigma_n\vert \bY_n)$ for the modified model also.\\

To investigate the form of the density where the modified SB model converges to, note that
\begin{eqnarray}
S_4 &=& \frac{1}{D}\sum_{(R_1^*)^c}\int_{\Pi}\int_{I_4}\sum_{i=1}^{M_n}\pi_i
\frac{\varphi(\theta_i,\sigma+\hat k_n)}{(\sigma+\hat k_n)}\phi\left(\frac{y-\theta_i}{\sigma+\hat k_n}\right)
\nonumber\\
&& \hspace{2cm}\times L(\Theta_{M_n}, z, \bY_n, \Pi)dH(\Theta_{M_n})dG_n(\sigma)d\Pi \nonumber\\
&=& \frac{1}{D}\int_{\Pi}\int_{I_4}\sum_{i=1}^{M_n}\pi_i\frac{\varphi(\theta_i,\sigma+\hat k_n)}{(\sigma+\hat k_n)}
\phi\left(\frac{y-\theta_i}{\sigma+\hat k_n}\right) \nonumber\\
&& \hspace{2cm}\times\sum_{(R_1^*)^c}L(\Theta_{M_n}, z, \bY_n, \Pi)dH(\Theta_{M_n})dG_n(\sigma)d\Pi. \nonumber\\
\label{eq:modified_S4}
\end{eqnarray}
 
For each $i$, using $GMVT$ we get
\begin{eqnarray}
&& \frac{1}{D}\int_{\Pi}\int_{I_4}\pi_i\frac{\varphi(\theta_i,\sigma+\hat k_n)}{(\sigma+\hat k_n)}
\phi\left(\frac{y-\theta_i}{\sigma+\hat k_n}\right) \nonumber\\
&& \hspace{1cm}\times\sum_{(R_1^*)^c}L(\Theta_{M_n}, z, \bY_n, \Pi)dH(\Theta_{M_n})dG_n(\sigma)d\Pi \nonumber\\
&=& \frac{\varphi(\theta_i,\sigma_n^*(y)+\hat k_n)}{(\sigma_n^*(y)+\hat k_n)}
\phi\left(\frac{y-\theta_n^*(y)}{\sigma_n^*(y)+\hat k_n}\right) \nonumber\\
&& \hspace{1cm}\times\frac{1}{D}\int_{\Pi}\int_{I_4}\pi_i\sum_{(R_1^*)^c}L(\Theta_{M_n}, z, \bY_n, \Pi)dH(\Theta_{M_n})dG_n(\sigma)d\Pi, \nonumber\\
\label{eq:new_convergence}
\end{eqnarray}
where, for every $y$, $\theta_n^*(y)\in (-a-c, a+c)$, and $\sigma_n^*(y)\in (0,\sigma_n)$.\\ 

Hence, $S_4$ given by (\ref{eq:modified_S4}) becomes
\begin{eqnarray}
S_4 &=& \frac{\varphi(\theta_i,\sigma_n^*(y)+\hat k_n)}{(\sigma_n^*(y)+\hat k_n)}
\phi\left(\frac{y-\theta_n^*(y)}{\sigma_n^*(y)+\hat k_n}\right) \nonumber\\
&& \hspace{1cm}\times\frac{1}{D}\int_{\Pi}\int_{I_4}\sum_{(R_1^*)^c}L(\Theta_{M_n}, z, \bY_n, \Pi)dH(\Theta_{M_n})dG_n(\sigma)d\Pi, \nonumber\\
&=& \frac{\varphi(\theta_i,\sigma_n^*(y)+\hat k_n)}{(\sigma_n^*(y)+\hat k_n)}\phi\left(\frac{y-\theta_n^*(y)}{\sigma_n^*(y)+\hat k_n}\right) \nonumber\\
&& \hspace{1cm}\times P\left((R_1^*)^c,I_4\vert\bY_n\right),
\label{eq:new_convergence2}
\end{eqnarray}
again using the fact that $\sum_{i=1}^{M_n}\pi_i=1$.

Since it is already shown in connection with the proofs of Theorems 5.1 (Section \ref{subsubsec:proof_theorem_5.1}) 
and 5.2 (\ref{subsubsec:proof_theorem_5.2}) that, almost surely,
$\frac{\varphi(\theta_i,\sigma_n^*(y)+\hat k_n)}{(\sigma_n^*(y)+\hat k_n)}
\phi\left(\frac{y-\theta_n^*(y)}{\sigma_n^*(y)+\hat k_n}\right)
\rightarrow \frac{\varphi(\theta^*(y),k)}{k}\phi\left(\frac{y-\theta^*(y)}{k}\right)$ and
$P\left((R_1^*)^c,I_4\vert\bY_n\right)\rightarrow 1$, it follows that
$S_4\rightarrow \frac{1}{k}\phi\left(\frac{y-\theta^*(y)}{k}\right)$.
With very minor adjustments to the proof of Theorem 5.3, here it can be proved that the EW model
and the modified SB model converge to the same distribution.